\documentclass[11pt]{amsart}

\RequirePackage[colorlinks,citecolor=blue,urlcolor=blue]{hyperref}

\usepackage{amsmath,amsthm,amsfonts,amssymb, mathrsfs}

\usepackage[margin=1in]{geometry}
\usepackage[shortlabels]{enumitem}

\calclayout

\numberwithin{equation}{section}

\usepackage{color}

\newtheorem{thm}{Theorem}[section]
\newtheorem{lem}{Lemma}[section]
\newtheorem{assum}{Assumption}[section]
\newtheorem{prop}{Proposition}[section]

\newtheorem{defn}{Definition}[section]
\newtheorem{rem}{Remark}[section]
\theoremstyle{definition}

\newcommand{\Cov}{\text{Cov}}
\newcommand{\Var}{\text{Var}}

\numberwithin{equation}{section}

\newcommand{\be}{\begin{equation}}
\newcommand{\ee}{\end{equation}}

\newcommand{\bes}{\begin{equation*}}
\newcommand{\ees}{\end{equation*}}

\newcommand{\R}{\mathbf{R}}
\newcommand{\T}{\mathbf T}
\newcommand{\Z}{\mathbf{Z}}

\newcommand{\N}{\mathbf{N}}

\newcommand{\E}{\mathbb{E}}

\newcommand{\bean}{\begin{eqnarray*}}
\newcommand{\eean}{\end{eqnarray*}}

\newcommand{\ha}{\#}

\newcommand{\HO}{\mathcal{H}}

\newcommand{\HS}{\mathscr{H}}

\newcommand{\cov}{\text{Cov}}

\newcommand{\cu}{\mathscr C_2}
\newcommand{\cl}{\mathscr C_1}
\newcommand{\du}{\mathscr D}

\begin{document}

\title[Small ball probability estimates for the H\"older semi-norm of SHE]{Small ball probability estimates for the H\"older semi-norm of the 
stochastic heat equation}
\author{Mohammud Foondun \and Mathew Joseph \and Kunwoo Kim}
\address{Mohammud Foondun\\ Department of Mathematics and Statistics \\26 Richmond St, Glasgow G1 1XH, United Kingdom} \email{mohammud.foondun@strath.ac.uk}

\address{Mathew Joseph\\ Statmath Unit\\ Indian Statistical Institute\\ 8th Mile Mysore Road\\ Bangalore 560059, India
} \email{m.joseph@isibang.ac.in}

\address{Kunwoo Kim \\Pohang University of Science and Technology (POSTECH), Pohang, Gyeongbuk, South Korea 37673
}
\email{kunwoo@postech.ac.kr}

\begin{abstract}
We consider the stochastic heat equation on $[0,\,1]$ with periodic boundary conditions and driven by space-time white noise. Under various natural conditions, we study small ball probabilities for the H\"older semi-norms of the solutions, and provide near optimal bounds on these probabilities. As an application, we prove a support theorem in these H\"older semi-norms.
\end{abstract}
\maketitle

\section{Introduction and Main results}
We consider the stochastic heat equation (SHE) on $\T:=[0,1]$ with periodic boundary condition and driven by space-time white noise (we identify $\T$ as the one-dimensional torus, i.e., $\T:=\R/\Z$).  This is the real-valued random field $u(t,x),\; t\in \R_+,\, x\in \T$ which solves
\begin{equation}\label{eq:she}
\begin{split} \partial_t u (t,x)&= \frac12\,\partial_x^2 u(t,x) + \sigma\big(t,x, u(t,x)\big) \cdot \dot{W}(t,x),\quad t
\in \R_+,\, x\in \T, 
\end{split}
\end{equation}
with given initial profile $u(0,\cdot)=u_0: \T \to \R$ and satisfying $u(t,0)=u(t,1)$ for all $t\in \R_+$. The space-time white noise $\dot W$ is a centered generalized Gaussian random field with $\E[\dot W(t,x) \dot W(s,y)]= \delta_0(x-y)\delta_0(t-s)$. We will make the following two assumptions on the function $\sigma: \R_+\times \T\times \R \to \R $.

\begin{assum}\label{assump1}
The function $\sigma$ is uniformly elliptic, that is, there are constants $\mathscr{C}_1>0$ and $\mathscr C_2>0$ such that 
\begin{equation} \label{eq:ellip} \mathscr C_1 \le \sigma(t,x,u) \le \mathscr C_2 \quad \text{for all } t,\, x,\, u.\end{equation}
\end{assum}
\begin{assum}\label{assump2}
The function $\sigma$ is Lipschitz continuous in the third variable, that is there is a constant $\mathscr D>0$ such that  
\begin{equation}\label{eq:lip}
|\sigma(t,x,u) - \sigma(t,x,v)| \le \mathscr D|u-v| \quad \text{for all } t, \, x,\, u,\, v.
\end{equation}
\end{assum}

The existence and uniqueness of solutions to \eqref{eq:she} under the above assumptions are well known. See for example \cite{spde-mini} or \cite{wals} for the proofs and various other properties. It is also known that the solutions are H\"older($\frac12-$) in space and H\"older($\frac14-$) in time. 

In this paper we study the probabilities of the events that the SHE \eqref{eq:she} is unusually regular, as measured in certain H\"older semi-norms, up to a fixed time. To the best of our knowledge, our paper is the first to carry out such a study even though regularity properties of SPDEs have been very well studied. See for instance \cite{HSWX} and \cite{TX} where very precise information about the modulus of continuity is given. 

Our study will be framed as small ball probabilities of these semi-norms. Small ball probabilities have been very well studied in many different settings but most of the results in the literature are for Gaussian processes; see \cite{LS} for an extensive survey on various developments and applications of small ball probabilities for Gaussian processes. However there has not been much progress for non-Gaussain processes. Only a handful of papers have looked at these types of questions for SPDEs. The paper closest to ours is that of the very recent \cite{athr-jose-muel} where the sup norm is considered. In another recent paper \cite{Lot}, heat equations with additive noise are considered under different norms and in \cite{Mar}, the stochastic wave equation is studied.

To phrase our results precisely we need some notations. Fix $0<\theta<\frac12$ and a terminal time $T>0$. For a function $f:[0, T] \times \T\to \R$ and for every $t\in [0, T]$ and $x\in \T$,  let
\[ \mathcal H_t^{(\theta)}(f) := \sup_{x\ne y}\frac{|f(t,x)-f(t,y)|}{|x-y|^{\frac12-\theta}}\]
be the {\it spatial} H\"older semi-norm and let 
\[\HS_x^{(\theta)}(f) := \sup_{0\le s\ne t\le T} \frac{|f(t,x)-f(s,x)|}{|t-s|^{\frac14-\frac\theta 2}}\]
be the {\it temporal} H\"older semi-norm. The H\"older($\frac12-$) regularity in space and the H\"older($\frac14-$) regularity in time of $u$ imply that $\sup_{t\in [0,T]} \HO_t^{(\theta)} (u)$ and $\sup_{x\in \T} \HS_x^{(\theta)} (u)$ are finite a.s. 
The above quantities provide a quantitative way of measuring regularity of functions. This is a natural measure; indeed the smaller $\mathcal H_t^{(\theta)}(f)$ is the closer $f(t,\cdot)$ is to a constant function, and similarly for $\HS_x^{(\theta)}(f)$. We investigate the probability that they are exceptionally small for solutions to the stochastic heat equations. 

There have been a few papers which investigate small ball probabilities for H\"older norms and various Sobolev norms, see for example \cite{kuel-li-shao}, \cite{kuelbs-Li}, \cite{bald-royn},  \cite{AGM} and the references in \cite{LS}. We emphasize that the above results are for Gaussian processes but in general the SPDEs that we consider here will be non-Gaussian. Another novel feature of our work is that we are able to obtain bounds on the probabilities that the $\mathcal H_t^{(\theta)}(f)$ (resp. $\HS_x^{(\theta)}(f)$ ) are {\it uniformly} small in time (resp. space), whereas  \cite{kuel-li-shao}, \cite{kuelbs-Li}, \cite{bald-royn} find bounds on the small ball probabilities of the H\"older norms of Gaussian processes $X(t)$ indexed by only one parameter $t\in\R_+$.

Before we state the main result, we introduce one more notation.  For each $\theta\in (0,\, 1/2)$,  we let $\Lambda=\Lambda(\theta)$ be given by 
 \be \label{eq:lambda} \Lambda(\theta):=\int_{\R} p(1,w) |w|^{\frac12-\theta} \, dw= \frac{2^{\frac{1}{2}-\theta}}{\sqrt{\pi}}\Gamma(1-\theta),\ee
 where $p(t,x)$ is the Gaussian density \eqref{eq:gaussian} and $\Gamma(t)$ is the Gamma function. 
 
\begin{thm}
\label{thm:main}
Let $0<\theta<\frac12$ and $0<\epsilon<1$. Suppose that the initial profile satisfies $\HO_0^{(\theta)}(u) \le \frac{\epsilon}{2}\left(1\wedge \frac{1}{\Lambda}\right)$. Then for any $\eta>0$ there exist positive constants $C_1,C_2,C_3,C_4>0$ dependent on $\cl, \cu,\du, \theta,\eta$ such that 
\begin{equation} \label{eq:holder:full} C_1\exp\left(- \frac{C_2 T}{\epsilon^{\frac3\theta+\eta}}\right) \le P \left(\sup_{\substack{0\le s, t\le T \\ x, y\in  \T\\(t,x)\ne (s,y)}} \frac{|u(t,x)-u(s,y)|}{|x-y|^{\frac12-\theta}+|t-s|^{\frac14-\frac\theta 2}}\le \epsilon \right)\le C_3\exp\left(- \frac{C_4T}{\epsilon^{\frac3\theta}|\log \epsilon|^{\frac32}}\right).
\end{equation}
\end{thm}
One can improve the lower and upper bounds in \eqref{eq:holder:full} by imposing more restrictions on $\sigma$; see Theorems \ref{thm1}, \ref{thm2} and \ref{thm3} below, and Remark \ref{rem:main}. In Section \ref{sec:ext} below we provide several {\it support} theorems where bounds on the probability that $u$ is close (in the above H\"older semi-norm) to a function $h$ in certain classes (such as H\"older spaces) are provided. 

The reader might wonder whether the upper and lower bounds in \eqref{eq:holder:full} hold when we consider the H\"older norm 
\be \label{eq:holder:norm} \| u\|_{\theta, T} : = \|u\|_{\infty, T}+ \sup_{\substack{0\leq s, t \leq T \\ x, y \in \T \\  (t, x) \neq (s, y) }} \,  \frac{|u(t, x)-u(s, y)|}{ |t-s|^{\frac14-\frac\theta 2}+ |x-y|^{\frac12-\frac\theta 2}}, \ee
instead of semi-norm considered in \eqref{eq:holder:full}, where $\|u\|_{\infty, T}:=\sup_{(t,x) \in [0,T]\times \T} |u(t,x)|$. This is {\it not} the case. Indeed, due to the $1:2:4$ scaling of fluctuations, space and time for the SHE, the H\"older semi-norm of the SHE in time-space regions of the form $[t, t+\epsilon^{\frac2\theta}] \times [x, x+\epsilon^{\frac1\theta} ]$ fluctuates by order $\epsilon$. However in these regions, the solution $u$ itself fluctuates by order $\epsilon^{\frac{1}{2\theta}}$. Intuitively, what we try to show in this article is that the $T\epsilon^{-\frac3\theta}$ time-space boxes obtained by dividing $[0,T]\times \T$ into subintervals of the form $[t, t+\epsilon^{\frac2\theta}] \times [x, x+\epsilon^{\frac1\theta} ]$ can be {\it somewhat} viewed as  independent regions. This explains the $T\epsilon^{-\frac3\theta}$ that we obtain in the exponents in \eqref{eq:holder:full}. Moreover, by this reasoning, one should expect similar bounds on the probability $P(\|u\|_{\infty, T} \le \epsilon^{\frac{1}{2\theta}} )$ if we start at $u_0\equiv 0$, for example. In fact this is what was proved in \cite{athr-jose-muel}. Therefore, while it is {\it not} true that we have the same bounds as \eqref{eq:holder:full} for the H\"older norm \eqref{eq:holder:norm}, we do have the same bounds for 
\[ P\left(\|u\|_{\infty,T}\le \epsilon^{\frac{1}{2\theta}},  \;\sup_{\substack{0\le s, t\le T \\  x, y\in \T\\(t,x)\ne (s,y)}} \frac{|u(t,x)-u(s,y)|}{|x-y|^{\frac12-\theta}+|t-s|^{\frac14-\frac\theta 2}}\le \epsilon \right),\]
if we start with $u_0$ such that $\|u_0\|_{\infty} \le \frac{\epsilon^{\frac{1}{2\theta}} }{2}$ and $\HO_0^{(\theta)}(u) \le \frac{\epsilon}{2}\left(1\wedge \frac{1}{\Lambda}\right)$.

\subsection{Results}Instead of looking at the probability of the event in \eqref{eq:holder:full} directly, we consider the probabilities of the events  $\left\{ \sup_{t\in [0,T]} \HO_t^{(\theta)} (u)\le \epsilon\right\}$ and $\left\{\sup_{x\in \T} \HS_x^{(\theta)} (u)\le \epsilon\right\}$ separately. The bounds in \eqref{eq:holder:full} will then follow from the bounds on the probabilities of these two events (see Section \ref{sec:main}). It turns out that the regularity in time (as measured by the smallness of $\HS_x^{(\theta)}(u)$) is intimately connected to the regularity in space (as measured by the smallness of  $\HO_t^{(\theta)}(u)$). Our arguments indicate that for the solution to be {\it regular} in time it is necessary for it to be {\it regular} in space. 
We now state the small ball probability estimates for $\sup_{t\le T} \HO_t^{(\theta)}(u)$ and $\sup_{x\in \T} \HS^{(\theta)}(u)$ with varying assumptions on the nonlinearity $\sigma$. 
\begin{thm}\label{thm1}
Assume that the function $\sigma(t,x,u)$ is independent of $u$ and satisfies Assumption \ref{assump1}. Let $0<\theta<\frac12$ and $0<\epsilon<1$. 
\begin{enumerate}[(a)]
\item \label{part1 }Suppose that the initial profile satisfies $\HO_0^{(\theta)}(u) \le \frac{\epsilon}{2}$. Then there exist positive constants $C_1, C_2,C_3,C_4 >0$ dependent only on $\cl,\cu,\theta$ such that 
\begin{equation} \label{eq:unif:t} C_1\exp\left(- \frac{C_2 T}{\epsilon^{\frac3\theta}}\right) \le P \left(\sup_{0\le t\le T} \HO_t^{(\theta)}\left(u\right) \le \epsilon \right)\le C_3\exp\left(- \frac{C_4T}{\epsilon^{\frac 3\theta}}\right).
\end{equation}
\item
Suppose that the initial profile satisfies $\HO_0^{(\theta)}(u) \le  \frac{\epsilon}{2\Lambda}$. Then there exist positive constants $C_1, C_2,C_3,C_4 >0$ dependent only on $\cl,\cu,\theta$ such that 
\begin{equation} \label{eq:unif:x} C_1\exp\left(- \frac{C_2 T}{\epsilon^{\frac3\theta}}\right) \le P \left(\sup_{x\in \T} \HS_x^{(\theta)}\left(u\right) \le \epsilon \right)\le C_3\exp\left(- \frac{C_4T}{\epsilon^{\frac 3\theta}}\right).
\end{equation}
\end{enumerate}
\end{thm}

It can be shown that $u$ is a Gaussian random field when $\sigma$ does not depend on $u$. The proof of the above theorem takes up a significant part of this paper and hinges on well known results specific to Gaussian processes, as well as the recently proved Gaussian correlation inequality (\cite{roye}). We next consider the case when $\sigma$ can also depend on $u$. 

\begin{thm} \label{thm2} Suppose that $\sigma(t,x,u)$ satisfies both Assumptions \ref{assump1} and \ref{assump2}. Let $0<\theta<\frac12$ and $0<\epsilon<1$. 
 \begin{enumerate}[(a)]
 \item Assume that the initial profile satisfies $\HO_0^{(\theta)}(u) \le \frac{\epsilon}{2}$. Then for any $\eta>0$ there exist positive constants $C_1,C_2,C_3,C_4>0$ dependent on $\cl, \cu,\mathscr D, \theta,\eta$ such that 
\begin{equation} \label{eq:unif:t:2} C_1\exp\left(- \frac{C_2 T}{\epsilon^{\frac3\theta+\eta}}\right) \le P \left(\sup_{0\le t\le T} \HO_t^{(\theta)}\left(u\right) \le \epsilon \right)\le C_3\exp\left(- \frac{C_4T}{\epsilon^{\frac3\theta}|\log \epsilon|^{\frac32}}\right).
\end{equation}
\item  Assume that the initial profile satisfies $\HO_0^{(\theta)}(u) \le \frac{\epsilon}{2\Lambda}$. Then for any $\eta>0$ there exist positive constants $C_1,C_2,C_3,C_4>0$ dependent on $\cl, \cu, \mathscr D, \theta,\eta$ such that 
\begin{equation} \label{eq:unif:x:2} C_1\exp\left(- \frac{C_2 T}{\epsilon^{\frac3\theta+\eta}}\right) \le P \left(\sup_{x\in \T} \HS_x^{(\theta)}\left(u\right) \le \epsilon \right)\le  C_3\exp\left(- \frac{C_4T}{\epsilon^{\frac3\theta}|\log \epsilon|^{\frac32}}\right).
\end{equation}

\end{enumerate}
\end{thm}

Our next result sharpens the lower bound of the above theorem by imposing a further restriction on the Lipschitz coefficient of $\sigma$.  

\begin{thm}\label{thm3}Suppose that $\sigma(t,x,u)$ satisfies both Assumptions \ref{assump1} and \ref{assump2}. Let $0<\theta<\frac12$ and $0<\epsilon<1$.  
\begin{enumerate}[(a)]
\item Assume that the initial profile satisfies $\HO_0^{(\theta)}(u) \le \frac{\epsilon}{2}$. Then there is a  $\du_0>0$ such that for all $\du<\du_0$, there exist positive constants $C_1, C_2, C_3, C_4  >0$ dependent only on $\cl,\cu,\theta$ such that 
\begin{equation}
 \label{eq:unif:t:3}
 C_1\exp\left(- \frac{C_2T}{\epsilon^{\frac 3\theta}}\right) \le  P \left(\sup_{0\le t\le T} \HO_t^{(\theta)}\left(u\right) \le \epsilon \right)\le C_3\exp\left(- \frac{C_4T}{\epsilon^{\frac3\theta}|\log \epsilon|^{\frac32}}\right).
 \end{equation}
 \item  Assume that the initial profile satisfies $\HO_0^{(\theta)}(u) \le \frac{\epsilon}{2\Lambda}$. Then there is a  $\du_1>0$ such that for all $\du<\du_1$, there exist positive constants $C_1, C_2, C_3, C_4  >0$ dependent only on $\cl,\cu,\theta$ such that 
\begin{equation} \label{eq:unif:x:2} C_1\exp\left(- \frac{C_2 T}{\epsilon^{\frac3\theta}}\right) \le P \left(\sup_{x\in \T} \HS_x^{(\theta)}\left(u\right) \le \epsilon \right)\le C_3\exp\left(- \frac{C_4T}{\epsilon^{\frac3\theta}|\log \epsilon|^{\frac32}}\right).
\end{equation}
 \end{enumerate}
\end{thm}

We now say a few words about the proofs of our theorems. As mentioned above, Theorem \ref{thm1} relies heavily on the fact that when $\sigma$ is independent of $u$, the solution $u(t, x)$  is a Gaussian random field. The proof of the upper bound is essentially contained in the proof of Lemma \ref{lem:a1} which among other things, relies on the Gaussianity of $u$.  Another crucial element is the sharp bound given by Lemma \ref{lem:Nt:tail} whose proof uses some well known ideas presented in \cite{athr-jose-muel}.  It is also interesting to note that when $\sigma$ is a constant, one can further simplify the proof of the upper bound by resorting to Slepian's lemma; see Remark \ref{rem:32} for more details.  The lower bounds rely even more heavily on Gaussianity of the solution in that we use the Gaussian correlation inequality in an essential way. This is done in Lemma \ref{lem:N:Nt}. Another key ingredient is the use of a change of measure argument similar to \cite{athr-jose-muel}. Intuitively, this allows us to keep the solution small which gives us a better handle on the estimates required.

Under the conditions of Theorems \ref{thm2} and \ref{thm3}, the solutions are no longer Gaussian, so the corresponding proofs require different strategies. For the lower bounds, we use a  {\it perturbation} argument together with the proof of the lower bound in Theorem \ref{thm1}. We note that the sharper lower bound in Theorem \ref{thm3} is also a consequence of the very same perturbation argument.  

The proofs of the upper bounds in Theorems \ref{thm2} and \ref{thm3} are entirely different and make use of certain auxiliary random variables which have nice independence properties. These random variables are indexed by the spatial variables.  Their construction is inspired by \cite{conu-jose-khos}. 

In the final section of this paper, we present some extensions and prove a support theorem in the H\"older semi-norm.  It will be clear later that our paper raises several questions. One such open question is whether the bounds \eqref{eq:unif:t} and \eqref{eq:unif:x} continue to hold in the general case, that is for any $\sigma$ satisfying Assumptions \ref{assump1} and \ref{assump2}. We have assumed that $\sigma$ is bounded below and above by positive constants. Another avenue of investigation is to replace these assumptions by less stringent ones. Let us point that here, when $\theta=\frac12$ the above theorems match the results recently obtained in \cite{athr-jose-muel} for the small ball probabilities of the sup norm of $u$. 

We have studied the small ball probability estimates of $\sup_{t\in [0,T]} \HO_t^{(\theta)} (u)$ and $\sup_{x\in \T} \HS_x^{(\theta)} (u)$. We next consider the small ball probability estimates of $\HO_T^{(\theta)}(u)$ for a fixed time $T$, and $\HS_X^{(\theta)}(u)$ for a fixed spatial point $X$. We start with the Gaussian case.

\begin{thm} \label{thm4} Assume that the function $\sigma(t,x,u)$ is independent of $u$ and satisfies Assumption \ref{assump1}, and fix a time $T>0$. Let $0<\theta<\frac12$ and $0<\epsilon<1$.
\begin{enumerate}[(a)]
\item  Assume that the initial profile satisfies $\HO_0^{(\theta)}(u) \le \frac{\epsilon}{2}$. Then there exist positive constants $C_1, C_2,C_3,C_4 >0$ dependent only on $\cl,\cu,\theta, T$ such that 
\begin{equation} \label{eq:t} C_1\exp\left(- \frac{C_2 }{\epsilon^{\frac1\theta}}\right) \le P \left( \HO_T^{(\theta)}\left(u\right) \le \epsilon \right)\le C_3\exp\left(- \frac{C_4}{\epsilon^{\frac 1\theta}}\right).
\end{equation}
\item Assume that the initial profile satisfies $\HO_0^{(\theta)}(u) \le \frac{\epsilon}{2\Lambda}$. Then there exist positive constants $C_1, C_2,C_3,C_4 >0$ dependent only on $\cl,\cu,\theta$ such that 
\begin{equation} \label{eq:x} C_1\exp\left(- \frac{C_2 T }{\epsilon^{\frac2\theta}}\right) \le P \left( \HS_X^{(\theta)}\left(u\right) \le \epsilon \right)\le C_3\exp\left(- \frac{C_4 T}{\epsilon^{\frac 2\theta}}\right).
\end{equation}
\end{enumerate}
\end{thm}
The upper bounds are in fact an immediate consequence of the proof of Theorem \ref{thm1}, and we will see that the constants $C_3$ and $C_4$ can be chosen independently of $T$. The  lower bounds follow from the arguments in the proof of  Theorem 2.2. of \cite{kuel-li-shao}. The proof of the lower bound above is specific to Gaussian processes and cannot be directly extended to the general case. 
\begin{thm} \label{thm5} Suppose that $\sigma(t,x,u)$ satisfies both Assumptions \ref{assump1} and \ref{assump2}, and fix a time $T>0$. Let $0<\theta<\frac12$ and $0<\epsilon<1$. 
\begin{enumerate}[(a)]
\item  Assume that the initial profile satisfies $\HO_0^{(\theta)}(u) \le \frac{\epsilon}{2}$. Then there exist positive constants $C_1,C_2>0$ dependent on $\cl, \cu,\theta, T$ such that 
\begin{equation} \label{eq:t2} 
P \left(\HO_T^{(\theta)}\left(u\right) \le \epsilon \right)\le C_1\exp\left(- \frac{C_2}{\epsilon^{\frac1\theta}|\log \epsilon|^{\frac32}}\right).
\end{equation}
\item Assume that the initial profile satisfies $\HO_0^{(\theta)}(u) \le \frac{\epsilon}{2\Lambda}$.  Then there exist positive constants $C_1,C_2>0$ dependent on $\cl, \cu,\theta$ such that 
\begin{equation} \label{eq:x2} 
P \left(\HS_X^{(\theta)}\left(u\right) \le \epsilon \right)\le  C_1\exp\left(- \frac{C_2T}{\epsilon^{\frac{2}{\theta}}}\right).
\end{equation}

\end{enumerate}
\end{thm}

A trivial lower bound is obtained from either Theorem \ref{thm2} or Theorem \ref{thm3} depending on whether $\sigma$ satisfies the assumptions of Theorem \ref{thm2} or Theorem \ref{thm3}. However this is very far from the lower bound obtained in Theorem \ref{thm4}.

\begin{rem} As the reader observes the bounds in Theorems \ref{thm:main},  \ref{thm2}, \ref{thm3} and \ref{thm5} are close to optimal but not sharp. While optimal results can be obtained in the Gaussian case (i.e. when $\sigma$ does not depend on $u$) using Gaussian-specific techniques, a perturbation argument is the main tool for the lower bounds in the non-Gaussian case. When the perturbation (as measured by the Lipschitz constant $\du$) is small, one can get similar lower bounds as in the Gaussian case (see Theorem \ref{thm3}). However, the perturbation argument works only when the time interval under consideration is small and therefore does not work for Theorem \ref{thm5}. Moreover, a similar perturbation argument cannot be implemented for the upper bounds as we don't have good control of the tail probabilities $\left\{|u(t,x)-u(s,x)|>\epsilon\right\}$ when $|t-s|=O(\epsilon^{\frac{2}{\theta}})$ (see Section \ref{sec:lb:nonlinear} for the perturbation argument for lower bounds). We thus have to resort to a general non-Gaussian argument which only gives us close to optimal upper bounds.
\end{rem}

\begin{rem}The dependence of the constants in the theorems on the parameters $\cl,\cu,\mathscr D
, \theta$ etc. is in general quite complicated. See for example Remark \ref{rem:a1:ubd}. We have tried to indicate the dependence on the parameters wherever we could.
\end{rem}

\begin{rem} Note that $v(t,x) = u(\rho t, x)$ satisfies 
\[ \partial_t v(t,x)
 = \frac{\rho}{2} \partial_x^2 v(t,x)
  +\tilde
   \sigma( t,x, v(t,x))\dot{ \tilde W}
   (t,x)
\]
for some other white noise
 $\dot{\tilde W}
 $, and $\tilde \sigma(t,x,v) :=
 \rho^{\frac12} \sigma(\rho t,x, v)$. The function $\tilde \sigma$ satisfies $\rho^{\frac12}\cl \le \tilde \sigma(t,x,v) \le \rho^{\frac12}\cu$ and $|\tilde \sigma(t,x,u) -\tilde \sigma(t,x,v)|\le \rho^{\frac12}\mathscr D
 |u-v|$. Thus one can obtain similar results with the inclusion of a diffusion parameter $\rho$ by converting it to the form \eqref{eq:she}.
\end{rem}

\begin{rem} Note that it is sufficient to prove the above theorems for all sufficiently small $\epsilon <\epsilon_0$, where $\epsilon_0$ is dependent on $\cl, \cu, \theta$ (and maybe additionally on $\eta$ in the case of Theorem \ref{thm2} and $T$ in the case of Theorems \ref{thm4} (a) and \ref{thm5} (a)). The conclusion for any $0<\epsilon<1$ follows from the fact that the probabilities $P \left(\HO_T^{(\theta)}\left(u\right) \le \epsilon \right)$, $P \left(\HS_X^{(\theta)}\left(u\right) \le \epsilon \right)$, $P \left(\sup_{0\le t\le T} \HO_t^{(\theta)}\left(u\right) \le \epsilon \right)$ and $P \left(\sup_{x\in \T} \HS_x^{(\theta)}\left(u\right) \le \epsilon \right)$ are nondecreasing in $\epsilon$. 
\end{rem}

The following table highlights the main differences between the main theorems and their extensions.
\begin{table}[ht]
\caption{Summary of results}
\centering 
\begin{tabular}{|l|l|p{6cm}|}
\hline
\textbf{Main Theorems:} & \textbf{Conditions on $\sigma$:} & \textbf{Types of small ball probabilities (SBP):}\\
\hline
Theorem 1.1& Dependent on $u$ &upper and lower bounds on space-time H\"older seminorms\\
\hline 
Theorem 1.2& Independent of $u$ &matching upper and lower bounds on SBP of  $\sup_{0\le t\le T}\HO_t^{(\theta)}$ and $\sup_{x\in \T}\HS_x^{(\theta)}$\\
\hline
Theorem 1.3& Dependent on $u$ &upper and lower bounds on SBP of $\sup_{0\le t\le T}\HO_t^{(\theta)}$ and $\sup_{x\in \T}\HS_x^{(\theta)}$\\
\hline
Theorem 1.4& Dependent on $u$ with small $\mathscr D$& improved lower bounds on SBP of $\sup_{0\le t\le T}\HO_t^{(\theta)}$ as compared to Theorem 1.3. Same upper bound as Theorem 1.3\\
\hline
Theorem 1.5&Independent of u & matching upper and lower bounds on SBP of  $\HO_T^{(\theta)}$ and $\HS_X^{(\theta)}$\\
\hline
Theorem 1.6&Dependent on u & Upper bound on  SBP of  $\HO_T^{(\theta)}$ and $\HS_X^{(\theta)}$\\
\hline
\textbf{Extensions:}& \textbf{Extra conditions on the equation:} & \textbf{Results:}\\
\hline
Theorem 8.1 & Presence of nice drifts& Theorems 1.1-1.4 hold\\
\hline
Theorem 8.2 &Presence of H\"older continuous drift & Bounds on SBP of $\sup_{0\le t\le T}\HO_t^{(\theta)}$\\
\hline
Theorem 8.3 &Presence of H\"older continuous drift & Bounds on SBP of $\sup_{x\in \T}\HS_x^{(\theta)}$\\
\hline
\end{tabular}
\end{table}

\begin{rem}\label{openq}
In this paper, we have studied bounds on various small ball probabilities. A related question would be to study the small ball constant, that is to find  the constant $C$ in the following 
\begin{equation*}
\lim_{\epsilon\rightarrow 0+}\epsilon^{\frac 3\theta}\ln P \left(\sup_{0\le t\le T} \HO_t^{(\theta)}\left(u\right) \le \epsilon \right)=-CT,
\end{equation*}
under appropriate conditions on $\sigma$. These types of questions are challenging and are beyond the scope of this current paper. Using the techniques in \cite{LiSha98}, we are confident that one can successfully tackle the small ball constant problem in some cases, but we leave this for future work. 
\end{rem}

{\bf Plan:} Section \ref{sec:pre} contains some preliminary estimates. The proofs of the upper bounds in Theorems \ref{thm1}, \ref{thm2} are given in Section \ref{sec:ubd}, while the correponding lower bounds are given in Section \ref{sec:lbd}. The proof of Theorem \ref{thm3} is given in Section \ref{sec:thm3}. After this, we give the proof of Theorem \ref{thm:main} in Section 6. The proofs of Theorems \ref{thm4} and \ref{thm5} are presented in Section \ref{sec:thm45}. Finally in Section \ref{sec:ext}, we give some extensions and prove a support theorem as a corollary of our results.

{\bf Notation:} Throughout this paper, $C$ with or without subscripts will denote positive constants whose value might change from line to line. Unless mentioned otherwise they will be independent of the parameters $\epsilon, \cl, \cu, \du$ etc. We will sometimes emphasize that  the dependence of the constants on specific parameters will be denoted by specifying the parameters in brackets, e.g. $C(\delta)$. For a random variable $X$ we denote $\|X\|_p:=E[|X|^p]^{1/p}$.

\section{Preliminaries}\label{sec:pre}
We define  the heat kernel $G(t,x)$ as  the fundamental solution of the heat equation on the torus $\T$
\bes
\begin{split}
\partial_t G(t,x) &=\frac12  \partial_x^2 G(t,x), \\
G(0,x) &= \delta_0(x).
\end{split}
\ees
Let 
\be
\label{eq:gaussian}
p(t,x) = (2\pi t)^{-1/2} \exp\left(-\frac{x^2}{2t}\right)
\ee
be the fundamental solution of the heat equation on $\R$. It is known that the heat kernel on $\T$ is given explicitly by 
\be \label{eq:G:p} G(t,x) = \sum_{k\in \Z} p(t,x+k).\ee
We interpret the solution to \eqref{eq:she} in the sense of Walsh (\cite{wals}) as a random field which satisfies 
\be  \label{eq:mild} u(t,x) = \big(G_t*u_0\big)(x) + N(t,x),  \quad \text{a.s.} \ee
for each $t$ and $x$, where the first term on the right is the space convolution of the heat kernel with the initial profile $u_0(x)$, i.e., 
\[ \big(G_t*u_0\big)(x) = \int_{\T} G(t, x-y)\cdot u_0(y)\, dy, \]
and the second term which we call the {\it noise term} is the space-time convolution of the heat kernel with the product of $\sigma\left(s, y, u(s,y)\right)$ and white noise:
\be \label{eq:N} N(t,x)=  \int_0^t \int_{\T} G(t-s, x-y)\cdot \sigma\big(s, y, u(s,y)\big) W(ds dy).\ee
We are working on the torus $\T:=\R/\Z$, so in the above two expressions $x-y$ should be interpreted as the unique point $z$ in $[0,1)$ such that $x-y=z+k$ for some $k\in \Z$.

We now show that it is enough to prove our main results under the assumption that $u_0\equiv 0$. For a function $g: \T\to \R$ dependent only on the spatial variable $x$,  define 
\begin{equation}\label{H-norm} \HO^{(\theta)} (g) :=\sup_{x\ne y\in \T} \frac{|g(x)-g(y)|}{|x-y|^{\frac12-\theta}}.
\end{equation}
(Note the absence of subscript $t$ in $\HO^{(\theta)}$). The first lemma is a simple observation about the spatial H\"older regularity of $G_t*u_0$.
\begin{lem} \label{lem:hold:det} If for some $a>0$ one has $\mathcal H^{(\theta)}(u_0)\le a$ then $\HO^{(\theta)}\big(G_t*u_0\big) \le a$ for all $t>0$.
\end{lem}
\begin{proof}Let $\tilde u_0:\R\to \R$ be the periodization of $u_0$, that is $\tilde u_0(x+k)=u_0(x)$ for all $k \in \Z$ and $x\in \T$. We have
\begin{equation*}
\begin{split}
& \left|\big(G_t*u_0\big)(x) -\big(G_t*u_0\big)(y)\right| \\
& = \left|\sum_{k \in \Z} \int_{\T} \Big[ p(t, x-z+k) -p(t, y-z+k)\Big]\cdot u_0(z)\, dz\right| \\
&= \left|\sum_{k \in \Z} \int_{\T} \Big[p(t, x-z+k) -p(t, y-z+k)\Big] \cdot \tilde u_0(z)\, dz\right| \\
&= \left| \sum_{k \in \Z} \int_{-k}^{-k+1} \Big[p(t,x-w) -p(t, y-w)\Big]\cdot \tilde u_0(w+k)\, dw\right| \\
&= \left|\int_{\R} p(t, w) \cdot \Big[\tilde u_0(x-w)-\tilde u_0(y-w)\Big]dw\right|
\end{split} 
\end{equation*}
We have $\left|\tilde u_0(x-w)-\tilde u_0(y-w) \right| \le a |x-y|^{\frac12-\theta} $ by assumption and $p(t,\cdot)$ integrates to $1$, therefore the result follows.
\end{proof}
We now prove a similar result for the temporal H\"older regularity of $(G_{\cdot} *u_0)(x)$. Recall the constant $\Lambda$ introduced in \eqref{eq:lambda}.
\begin{lem} \label{lem:hold:det:t}  If $\HO^{(\theta)}(u_0)\le a$ then
\[ \frac{\left |\left(G_t*u_0\right)(x) -\left(G_s*u_0\right)(x)\right|}{|t-s|^{\frac14-\frac\theta 2}} \le  \Lambda a\]
\end{lem}
\begin{proof} Without loss of generality assume that $s<t$. Let $g(x) = \left(G_s*u_0\right)(x)$ and $\tilde g$ be the periodization of $g$. Then by arguments similar to Lemma \ref{lem:hold:det} we have
\begin{align*}
\left| \left(G_t*u_0\right)(x) -\left(G_s*u_0\right)(x)\right| &= \left|\left(G_{t-s}*g\right)(x) - g(x)\right| \\
 &=  \int_{\R} p(t-s,w) \cdot \left| \tilde g(x-w)- \tilde g(x)\right| \\
 &\le a \int_{\R} p(t-s,w) |w|^{\frac12-\theta}\, dw \\
 &\le \Lambda a |t-s|^{\frac14-\frac\theta 2},
\end{align*}
by a simple change of variables. 
\end{proof}

Now consider the random field 
\begin{equation}\label{reduction}
v(t,x)= u(t,x)- (G_t*u_0)(x),
\end{equation}
where $u(t,x)$ solves \eqref{eq:she} with the initial profile $u_0$. One can easily check that 
\[ \partial_t v(t,x)= \frac12\partial_x^2 v(t,x) +\widetilde \sigma\big(t,x,v(t,x)\big)\cdot \dot{W}(t,x),\]
with initial profile $v_0\equiv 0$, where
\[ \widetilde \sigma(t,x,v) = \sigma\Big(t,x,v+(G_t*u_0)(x)\Big).\]
Furthermore $|\widetilde \sigma(t,x,v) -\widetilde\sigma(t,x,w)|\le \du |v-w|$ and $\widetilde \sigma$ is bounded below and above by $\cl, \cu$.

 Assume $\HO^{(\theta)}(u_0) \le \frac\epsilon2$. Then from Lemma \ref{lem:hold:det} and \eqref{reduction}, we have the following implications:
\bes
\begin{split}
\sup_{0\le t\le T} \HO^{(\theta)}_t(v) \le \frac\epsilon2 \qquad &\text{implies}\qquad  \sup_{0\le t\le T} \HO^{(\theta)}_t(u) \le \epsilon,\\
\sup_{0\le t\le T} \HO^{(\theta)}_t(u) \le \frac\epsilon2 \qquad &\text{implies}\qquad  \sup_{0\le t\le T} \HO^{(\theta)}_t(v) \le \epsilon.
\end{split}
\ees
Similar implications hold when we just consider $\HO_T^{(\theta)}(u)$ and $\HO_T^{(\theta)}(v)$ (without the supremum in $t$).

Similary if  $\HO^{(\theta)}(u_0) \le \frac{\epsilon}{2\Lambda}$, then from Lemma  \ref{lem:hold:det:t},
\bes
\begin{split}
\sup_{x\in \T} \HS^{(\theta)}_x(v) \le \frac\epsilon2 \qquad &\text{implies}\qquad  \sup_{x\in \T} \HS^{(\theta)}_x(u) \le \epsilon,\\
\sup_{x\in \T} \HS^{(\theta)}_x(u) \le \frac\epsilon2 \qquad &\text{implies}\qquad  \sup_{x\in \T} \HS^{(\theta)}_x(v) \le \epsilon.
\end{split}
\ees
Similar implications hold when we just consider $\HS_X^{(\theta)}(u)$ and $\HS_X^{(\theta)}(v)$ (without the supremum in $x$).

\begin{rem} \label{rem:reduction} {\bf (Important)} From the above discussion we observe that it is sufficient to prove the main theorems stated in the introduction with $u_0\equiv 0$. This is of course not surprising since the Laplacian is known to have smoothing effects. The above argument is merely a weak manifestation of this. {\bf We will assume that the initial profile $u_0\equiv 0$ for the rest of the article.}
\end{rem}

We will need the following lemmas which were proved in \cite{athr-jose-muel}. 
\begin{lem}[Lemma 3.3 in \cite{athr-jose-muel}] \label{lem:N:holder} There exist constants $C_1, C_2$ such that for all time points $0<s<t<1$, spatial points $x, y\in \T$, and $\lambda>0$, 
\begin{align*}
P\left( \big|N(t,x)- N(t,y)\big|>\lambda\right)& \le C_1\exp\left(- \frac{C_2\lambda^2}{\cu^2|x-y|}\right) \\
P\left( \big|N(t,x)- N(s,x)\big|>\lambda\right)& \le C_1\exp\left(- \frac{C_2\lambda^2}{\cu^2|t-s|^{1/2}}\right).
\end{align*}
\end{lem}
\begin{lem}[Lemma 3.4 in \cite{athr-jose-muel}]\label{lem:eq:N:tail} There exist universal constants $\mathbf K_1, \, \mathbf K_2>0$ such that for all $\alpha, \lambda, \epsilon>0$, $\theta >0$, and for all $a \in [0, 1)$ with $a+\epsilon^{1/\theta}<1$ we have
\begin{align}\label{eq:N:tail}
P \left(\sup_{\stackrel{0\le t\le \alpha\epsilon^{\frac2\theta}}{x\in [a, a+\epsilon^{\frac1\theta}]}} |N(t,x)|>\lambda \epsilon^{\frac{1}{2\theta}}\right) \le    \frac{\mathbf K_1}{1\wedge\sqrt \alpha}\exp\left( - \mathbf K_2 \frac{\lambda^2}{\mathscr C_2^2\sqrt\alpha}\right).
\end{align}
\end{lem}
\begin{rem} \label{rem:N:tail} Note that Lemma 3.4 in \cite{athr-jose-muel} provides \eqref{eq:N:tail} when $a=0$, however, one can follow exactly the same proof to get \eqref{eq:N:tail} for any $a\in (0,1)$ with $a+\epsilon^{1/\theta}<1$.  It was also pointed out in \cite[Remark 3.1]{athr-jose-muel} that if  $\big|\sigma\big(s, y, u(s,y)\big)\big|\le C_1\epsilon^{\frac{1}{2\theta}}$ then one can bound the right hand side of \eqref{eq:N:tail} by $\frac{\mathbf K_1}{1\wedge\sqrt \alpha}\exp\left( - \mathbf K_2 \frac{\lambda^2}{C_1^2\epsilon^{\frac1\theta}\sqrt\alpha}\right)$. 
\end{rem}

The analysis of the following function will play a crucial role in this paper. 
\begin{equation} \label{eq:N:til}
\begin{split}
 \widetilde N(t,x,y) &:=\int_0^t \int_{\T} \frac{G(t-r, x-z)- G(t-r,y-z) }{|x-y|^{\frac12-\theta}} \, \sigma\left(r, z, u(r,z)\right) \dot{ W}(dr dz)\\
 &= \frac{N(t,x)-N(t,y)}{|x-y|^{\frac12-\theta}}.
 \end{split}
 \end{equation}
 Although we have not made it explicit, the function $\widetilde N(t,x,y)$ clearly depends also on $\theta$. The following lemma is used several times in the paper. 
 
 \begin{lem} \label{lem:Nt:tail} Let $\theta \in (0, \frac12)$. There exist constants $\mathbf K_3, \mathbf K_4$ dependent only on $\theta$ such that for all $\alpha, \lambda, \epsilon>0$ and for all $a\in [0,1)$ with $a+\epsilon^{1/\theta}<1$,  we have 
 \be\label{eq:Nt:tail} 
 P\left(\sup_{\stackrel{0\le t\le \alpha \epsilon^{\frac2\theta}}{x,y\in [a, a+\epsilon^{\frac1\theta}],\, x\ne y }}\big |\widetilde N(t,x,y)\big| > \lambda\epsilon\right) \le \frac{\mathbf K_3}{1\wedge \sqrt \alpha} \exp \left(-\mathbf K_4\frac{\lambda^2}{\cu^2  \alpha^\theta} \right).
 \ee
 \end{lem}
 \begin{proof}  We will show \eqref{eq:Nt:tail} when $a=0$ and the same proof works for the general $a\in (0,1)$, which is similar to Lemma \ref{lem:eq:N:tail} and Remark \ref{rem:N:tail}.  Let us first consider the case when $\alpha\ge 1$. Consider the following grid on $[0,\alpha \epsilon^{\frac2\theta}]\times [0, \epsilon^{\frac1\theta}]$, where the first coordinate is time and the second is space: 
 \[ \mathbb G_n =\left\{\left(\frac{j}{2^{2n}}, \, \frac{k}{2^n}\right): 0\le j\le \alpha \epsilon^{\frac2\theta}2^{2n},\, 0\le k \le \epsilon^{\frac1\theta}2^n\right\}. \]
 The grid $\mathbb G_n$ will consist of only the point $(0,0)$ if  $n<n_0$ where
 \begin{equation}\label{n_0}
 n_0:=\lceil \log_2 (\alpha^{-\frac12}\epsilon^{-\frac1\theta})\rceil
 \end{equation}
 Any $p\in \mathbb G_n$ will be of the form $\left(\frac{j}{2^{2n}}, \, \frac{k}{2^n}\right)$ , so for notational convenience, we set
 \begin{equation*}
 N(p):=N\left(\frac{j}{2^{2n}}, \, \frac{k}{2^n}\right).
 \end{equation*}
 We will choose two parameters $0<\gamma_0(\theta)<\gamma_1(\theta)<\frac12$ which depend only on $\theta$ which satisfy the following constraint
 \be\label{eq:gamma}
 \frac12-\theta=\gamma_1-\gamma_0.
 \ee
 We will fix the constant
 \be \label{eq:K} K = \frac{1-2^{-\gamma_1}}{2^{1+\gamma_0 n_0}},\ee
and consider the events 
 \[ A(n,\lambda) =\left\{|N(p)-N(q)|\le \lambda K \epsilon 2^{-\gamma_1n}2^{\gamma_0n_0},\,  \text{for all } p, q \in \mathbb G_n \text{ spatial neighbors}\right\}.\]
 By $p, q$ being spatial neighbors in the grid $\mathbb G_n$, we mean that $p, q$ have the same time coordinate but their spatial coordinates are adjacent in $\mathbb G_n$. For instance $\left(\frac{j}{2^{2n}}, \, \frac{k-1}{2^n}\right)$ and $\left(\frac{j}{2^{2n}}, \, \frac{k+1}{2^n}\right)$ are spatial neighbors of $\left(\frac{j}{2^{2n}}, \, \frac{k}{2^n}\right)$ .

  The number of such pairs of points $p, q$ is bounded by $2\cdot \alpha \epsilon^{\frac2\theta}2^{2n} \cdot \epsilon^{\frac1\theta} 2^n \le  2^4\cdot 2^{3(n-n_0)}$, where we have used \eqref{n_0}. Therefore a union bound along with the first tail bound in Lemma \ref{lem:N:holder} gives 
 \bes
 \begin{split}
 P\big(A(n,\lambda)^c\big) & \le C_1  2^{3(n-n_0)} \exp \left(-\frac{C_2\lambda^2K^2\epsilon^22^{-2\gamma_1n}2^{2\gamma_0n_0}}{\cu^22^{-n}}\right) \\
 & \le C_1  2^{3(n-n_0)} \exp \left(-\frac{C_2\lambda^2K^2 \alpha^{-\theta} 2^{-2n_0\theta}2^{-2\gamma_1n}2^{2\gamma_0n_0}}{\cu^22^{-n}}\right)  \\
 & \le C_1 2^{3(n-n_0)} \exp\left(-\frac{C_2\lambda^2K^22^{(1-2\gamma_1)(n-n_0)}}{\cu^2\alpha^\theta}\right),
 \end{split}
 \ees
where the second inequality follows since $\alpha^\theta \epsilon^2 2^{2n_0\theta} \ge 1$ by our choice of $n_0$ and the final inequality is obtained using the choice $\gamma_0, \gamma_1$ in \eqref{eq:gamma}. 
 
We now let $A(\lambda):= \cap_{n\ge n_0} A(n, \lambda)$ and use a union bound once again to obtain
 \bes
 \begin{split}
 P\big(A(\lambda)^c\big) &\le \sum_{n\ge n_0} P\big(A(n,\lambda)^c\big) \\
 &\le C_1\sum_{n\ge n_0} 2^{3(n-n_0)} \exp\left(-\frac{C_2\lambda^2K^22^{(1-2\gamma_1)(n-n_0)}}{\cu^2\alpha^\theta}\right) \\
 &\le C_3 \exp\left(-\frac{C_4 \lambda^2K^2}{\cu^2\alpha^\theta}\right).
 \end{split}
 \ees
 Now on the event $A(\lambda)$, one has for $p,q$ spatial neighbors in $\mathbb G_n$ 
 \bes
 \frac{|N(p)-N(q)|}{|p-q|^{\frac12-\theta}} \le \frac{\lambda K \epsilon 2^{-\gamma_1 n}2^{\gamma_0n_0}}{2^{-n(\frac12-\theta)}}  \le \lambda \epsilon,
 \ees
 by our choice of $\gamma_0, \gamma_1$ and $K$ in \eqref{eq:gamma} and \eqref{eq:K}. 
 
We now show that the above bound continues to hold when $p, q\in \mathbb G_n$ are no longer spatial neighbors but have the same time coordinate.  Let the spatial coordinate of $p$ be $k2^{-n}$ and let the spatial coordinate of $q$ be $l2^{-n}$, and without loss of generality assume $k<l$. Find the smallest positive integer $n_1$ with $n_0\le n_1\le n$ such that 
\be\label{eq:n1}\frac{k}{2^n}\le \frac{k_1}{2^{n_1}} <\frac{k_1+1}{2^{n_1}} \le \frac{l}{2^n}\ee
for some nonnegative integer $k_1$. First note that we must have
\be \label{eq:k-l} \frac{1}{2^{n_1}} \le \left|\frac{k}{2^n} -\frac{l}{2^n}\right| \le \frac{4}{2^{n_1}}.\ee
The lower bound is clear by \eqref{eq:n1} and the upper bound follows from the minimality of $n_1$, for if the difference between $k2^{-n}$ and $l2^{-n}$ was larger than $2^{2-n_1}$ then there would be two spatial neighbors in $\mathbb G_{n_1-1}$ between them.   

One next observes that we can find a sequence of points $p_i, \, n_1\le i \le n$ and $q_i,\, n_1\le i\le n$ with the same time coordinates as $p, q$, such that $p_i, p_{i+1}$ (resp. $q_i, q_{i+1}$) are either equal or adjacent spatial points in $\mathbb G_i$. In addition at most one such adjacent spatial pair $(p_i, p_{i+1})$ (resp. $(q_i, q_{i+1})$) 
is in each $\mathbb G_j,\, n_1\le j\le n$, and $p_n=p, q_n=q$. 
 Therefore 
\bes
\begin{split}
\left\vert N(p) -N(q)\right\vert & \le \sum_{i=n_1}^n \big|N(p_i) - N(p_{i+1})\big| +\sum_{i=n_1}^n \big|N(q_i) - N(q_{i+1})\big|  \\
&\le 2\sum_{i=n_1}^n \lambda K \epsilon 2^{-\gamma_1 i} 2^{\gamma_0 n_0}
\end{split}
\ees
on the event $A(\lambda)$. As a consequence, on this event
\bes
\begin{split}
\frac{\left\vert N(p) -N(q)\right\vert}{|p-q|^{\frac12-\theta}} &\le \frac{2\lambda K\epsilon}{1-2^{-\gamma_1}} \cdot \frac{2^{\gamma_0n_0}2^{-\gamma_1n_1}}{2^{-n_1(\frac12-\theta)}} \le \lambda\epsilon
\end{split}
\ees
by \eqref{eq:k-l} and our choice of $\gamma_0,\gamma_1$ and $K$ in \eqref{eq:gamma} and \eqref{eq:K}. This completes the proof in the case $\alpha\ge1$.

In the case $0<\alpha<1$, we divide the spatial interval into smaller intervals of length $\sqrt{\alpha} \epsilon^{\frac1\theta}$ to get that 
\begin{align*}
 P\left(\sup_{\stackrel{0\le t\le \alpha \epsilon^{\frac2\theta}}{x,y\in [0, \epsilon^{\frac1\theta}],\, x\ne y }} \big|\widetilde N(t,x,y)\big| > \lambda\epsilon\right) & \le \sum_{i=1}^{1/\sqrt{\alpha}} P\left(\sup_{\stackrel{0\le t\le \alpha \epsilon^{\frac2\theta}}{x,y\in \left[i\sqrt{\alpha}\epsilon^{\frac1\theta}, (i+1)\sqrt \alpha\epsilon^{\frac1\theta}\right],\, x\ne y }} \big|\widetilde N(t,x,y)\big|> \frac{\lambda(\alpha^{\frac\theta 2} \epsilon)}{\alpha^{\frac\theta 2}}\right).
\end{align*}
Then, as explained in the beginning of the proof of this lemma (see also Remark \ref{rem:N:tail}), the probabilities inside the sum on the right hand side above have the same upper bound as for \[ P\left(\sup_{\stackrel{0\le t\le \alpha \epsilon^{\frac2\theta}}{x,y\in [0, \sqrt \alpha\epsilon^{\frac1\theta}],\, x\ne y }} \big|\widetilde N(t,x,y)\big|> \frac{\lambda(\alpha^{\frac\theta 2} \epsilon)}{\alpha^{\frac\theta 2}}\right).\] We now apply the previous argument to finish the proof.
\end{proof}
 \begin{rem} \label{rem:Nt:tail} From Remark \ref{rem:N:tail}, it also follows from the proof that if  $\big|\sigma\big(s, y, u(s,y)\big)\big|\le C_1\epsilon^{\frac{1}{2\theta}}$ then one can bound the right hand side of \eqref{eq:Nt:tail} by $\frac{\mathbf K_3}{1\wedge\sqrt \alpha}\exp\left( - \mathbf K_4 \frac{\lambda^2}{C_1^2\epsilon^{\frac1\theta}\alpha^\theta}\right)$.
\end{rem}

Define 
\begin{equation}\label{eq:N:hash}
N^{\#}(s,t,x):=\frac{N(t,x)-N(s,x)}{|t-s|^{\frac14-\frac\theta 2}}.
\end{equation}
The proof of the following lemma is similar to that of Lemma \ref{lem:Nt:tail} and is therefore omitted. 
\begin{lem} \label{lem:Nh:tail} Let $\theta \in (0, \frac12)$. There exist constants $\mathbf K_7, \mathbf K_8$ dependent only on $\theta$ such that for all $\alpha, \lambda, \epsilon>0$ and for all $a\in [0,1)$ with $a+\epsilon^{1/\theta}<1$,  we have 
 \be\label{eq:Nh:tail} 
 P\left(\sup_{\stackrel{0\le s, t\le \alpha \epsilon^{\frac2\theta},\, s\ne t}{x\in [a, a+\epsilon^{\frac1\theta}] }}\big | N^{\#}(s,t,x)\big| > \lambda\epsilon\right) \le \frac{\mathbf K_7}{1\wedge \sqrt \alpha} \exp \left(-\mathbf K_8\frac{\lambda^2}{\cu^2  \alpha^\theta} \right).
 \ee
 \end{lem}
 \begin{rem}\label{rem:Nh:tail}Note also here that a similar statement to that of Remark \ref{rem:Nt:tail} also holds in this case. That is, if $\big|\sigma\big(s, y, u(s,y)\big)\big|\le C_1\epsilon^{\frac{1}{2\theta}}$, then one can bound the right hand side of \eqref{eq:N:tail} by $\frac{\mathbf K_7}{1\wedge\sqrt \alpha}\exp\left( - \mathbf K_8 \frac{\lambda^2}{C_1^2\epsilon^{\frac1\theta}\alpha^\theta}\right)$.
 \end{rem}

We will also need some estimates concerning $G(t,\,x)$, which come from Lemmas 3.1 and 3.2 of \cite{athr-jose-muel}. For the lemma below let
\begin{equation*}
x_*=\begin{cases} x,\, & 0\le x\le \frac12\\ x-1, \,& \frac12<x\le 1.\end{cases}
\end{equation*}
We have

\begin{lem}\label{lem:heat-kernel}
 There exist positive constants $C_0, C_1, \tilde C_1, C_2, C_3$  such that
\begin{equation}\label{eq:heat-kernel}
G(t,\,x)\leq C_0 p(t,\,x_*)\;\;\text{ for all } x\in [0,\,1],\, t\in \T,
\end{equation}
\begin{equation} \int_0^{t} \int_{\T} |G(s, x-z)-G(s, y-z)|^2 dz \, ds \leq C_0|x-y| \;\;\text{ for all } x\in[0, \, 1] \text{ and } t\ge 0,
\end{equation}
 \begin{equation} \label{eq:var:t_incr} \tilde C_1 \sqrt{t-s} \le\int_s^t \int_{\T} G^2(r,x) dx  dr\le C_1\sqrt{t-s} \;\;\text{ for } 0<s\le t\le s+1 ,
 \end{equation}
\begin{equation}\label{eq:var:t:g-g}
C_2\sqrt{t-s}\le \int_0^s\int_{\T} \left[G(t-r,z)- G(s-r,z)\right]^2 dzdr \le C_3 \sqrt{t-s}\;\; \text{ for all } 0<s\le t<\infty.
\end{equation}

\end{lem}

With the preliminaries in place we can now move on to proving the theorems stated in the introduction.

\section{Upper bounds}\label{sec:ubd}

\subsection{Upper bound in Theorem \ref{thm1} (\lowercase{a})} \label{sec:thm1:ubd}
We are assuming that the function $\sigma(t,x,u)=\sigma(t,x)$ does not depend on the third variable so the random field $u(t,x)$ is Gaussian. Before proving the required estimates, we describe the main strategy behind the proof.

Fix parameters $c_0>0, c_1\ge 4$ to be specified later, and let
 \[\delta:=\epsilon^{\frac1\theta}.\] 
 We consider discrete time-space points $(t_i, x_j)$, where the time points $t_i$ are uniformly spaced in $[0,T]$ and space points $x_j$ are uniformly spaced in $\T$:
\be \label{eq:t:x}
\begin{split}
& t_i= ic_0 \delta^2,\qquad  i=0,\, 1,\, \cdots, I:= \left[\frac{T}{c_0\delta^2}\right]\\
& x_j= j c_1\delta,\qquad j=0,\,1,\, \cdots, J:= \left[\frac{1}{c_1\delta}\right].
\end{split}
\ee
We clearly have 
\begin{equation}\label{eq:ubd:dis} \begin{split}P \left(\sup_{0\le t\le T}\HO^{(\theta)}_t(u) \le \epsilon\right) & \le P\left(\max_{\stackrel{i=0,1,\cdots, I}{j=0,1\cdots J}} \frac{\big|u(t_i, x_j+\delta) -u(t_i, x_j)\big|}{\delta^{\frac12-\theta}}\le \epsilon\right) \\
& \le P\left(\max_{\stackrel{i=0,1,\cdots, I}{j=0,1\cdots J}} \big|u(t_i, x_j+\delta) -u(t_i, x_j)\big|\le \epsilon^{\frac{1}{2\theta}}\right),
\end{split}
\end{equation}
Consider the events
\begin{equation}\label{eq:ai} A_i := \left\{\max_{j=0,1,\cdots, J} \big|u(t_i, x_j+\delta)- u(t_i, x_j)\big|\le \epsilon^{\frac{1}{2\theta}}\right\}.\end{equation}
From the above 
\be \label{eq:ubd:dis} \begin{split}
P \left(\sup_{0\le t\le T}\HO^{(\theta)}_t(u) \le \epsilon\right) & \le P \left( \bigcap_{i=0}^I A_i\right) \\
&= \prod_{i=0}^I P\left(A_i \big \vert A_0, A_1\cdots A_{i-1} \right).
\end{split}
\ee
We will show in Lemma \ref{lem:a1} below that for some $0<\eta<1,$
\begin{equation} \label{eq:ai:unif}
P\left(A_i \;\Big \vert \;u(s,x),\, s\le t_{i-1}, \, x\in \T \right) \le \eta^J
\end{equation}
{\it uniformly in $i$}. Since the above bound holds regardless of the profile up to time $t_{i-1}$ one can conclude that the right hand side of \eqref{eq:ubd:dis} is bounded by $\eta^{J(I+1)}$, which gives us the required upper bound in Theorem \ref{thm1}.

We have the following lemma which plays an important role along with the fact the solution is Gaussian.  For $k\in \N^+$ and $\delta>0$, we define 
\begin{equation}\label{noise_diff}
\tilde\Delta_k := N(t_1, x_k+\delta) - N(t_1, x_k).
\end{equation}
\begin{lem}\label{lem:var:covar} Fix $c_0>0$ and $c_1\ge 4$. Then there exist positive constants $C_0,\,C_1,\, C_2$ such that for all $\delta$ small enough,
\be \label{eq:var:bd}  C_0\mathscr C_1^2 \sqrt{c_0}\delta\le \textnormal{Var}\big(\tilde\Delta_k \big)\le C_1 \mathscr C_2^2\delta  
\ee
uniformly in $k$. If $0<|x_k-x_l|< \frac12$ then 
\be
\label{eq:cov:bd}
\begin{split}
\left|\textnormal{Cov}\big(\tilde\Delta_k,\, \tilde\Delta_l \big)\right| \le C_2 \sqrt{c_0}\mathscr C_2^2\delta \exp\left(-\frac{|x_k-x_l|^2}{64 t_1}\right). 
\end{split}
\ee
\end{lem}
\begin{proof}
Since $\tilde\Delta_k$ is a mean zero random variable, we can use It\^{o}'s isometry along with the bound on $\sigma$ given by \eqref{eq:ellip} to obtain 
\bes
\begin{split}
\textnormal{Var}\big(\tilde\Delta_k \big) &\le \mathscr C_2^2 \int_0^{t_1} \int_{\T}\big[G(s, y+\delta)-G(s, y)\big]^2 dy ds\\
&\leq \mathscr C_2^2 C \delta,
\end{split}
\ees
where the last inequality comes from Lemma \ref{lem:heat-kernel}. 
This gives the required upper bound in \eqref{eq:var:bd}

Next, we have
\bes 
 \begin{split}
 \textnormal{Var}\big(\tilde\Delta_k \big) & \ge \mathscr C_1^2 \int_0^{t_1} \int_{\T}\big[G(s, y+\delta)-G(s, y)\big]^2 dy ds \\
 &=2\mathscr C_1^2 \int_0^{t_1}[G(2s,0)-G(2s,\delta)] ds\\
 &\geq C \mathscr C_1^2 \int_0^{t_1}\frac{1}{\sqrt{s}} ds,
 \end{split} 
 \ees which gives the lower bound in \eqref{eq:var:bd}. Let us explain how we obtain the last inequality above.  For $k\ge 1$ and $\delta$ small enough one has $(k-\delta)^2\ge \frac{1}{10}(k^2+\delta^2)$. Thus for all $0\leq s\leq 2t_1$,
\begin{align*} 
G(s, 0)-G(s, \delta) & = \frac{1}{\sqrt{2\pi s}}  \left\{ \sum_{k\in \Z} e^{-k^2/2s} - \sum_{k\in\Z} e^{-(k+\delta)^2/2s} \right\}\\
&\geq  \frac{1}{\sqrt{2\pi s}}\left\{ 1 - e^{-\delta^2/2s}- \sum_{k=1}^\infty e^{-(k-\delta)^2/2s} \right\} \\
&\geq \frac{1}{\sqrt{2\pi s}}\left\{ 1 -  e^{-1/4c_0}-  e^{-\delta^2/20s} \sum_{k=1}^\infty e^{-k^2/20s} \right\}\\
&\geq \frac{C}{\sqrt{2\pi s}} \left\{ 1- e^{-1/4c_0}-\frac{\sqrt{40\pi t_1}}{2} \cdot e^{-1/(40c_0)} \right\}\\
&\geq \frac{C}{\sqrt{s}}.
\end{align*}
The second last inequality is a consequence of bounding the sum from above by an appropriate integral from a Riemann sum approximation.

We next turn to the bound on the covariance. Observe that if we assume that $k>l$, we have $a=x_k-x_l = (k-l)c_1\delta$ and therefore $a+\delta >a-\delta >a/4$.  Using this, the semigroup property of the heat kernel and \eqref{eq:heat-kernel} we have 
 \begin{align*}
 \left|\textnormal{Cov} \big(\tilde\Delta_k, \tilde\Delta_l\big) \right| & \le  \mathscr C_2^2 \int_0^{t_1} \int_{\T} \left |G(s, y+\delta) -G(s,y)\right|\cdot \left|G(s, y+a+\delta) -G(s, y+a)\right| \, dy \, ds\\
&\leq   \mathscr C_2^2 \int_0^{t_1}  \left( 2 G(2s, a) + G(2s, a-\delta) + G(2s, a+\delta)   \right) \, ds \\
&\leq C\mathscr C_2^2 \int_0^{t_1} \left( 2 p(2s, a) + p(2s, a-\delta) + p(2s, a+\delta)\right) \, ds\\
&\leq C \mathscr C_2^2 \sqrt{t_1} \exp\left( - \frac{|a|^2}{64 t_1} \right),
 \end{align*}which completes the proof since $t_1=c_0\delta^2$ and $|a|=|x_k-x_l|$. 
\end{proof}

For the next lemma recall the events $A_i$ defined in \eqref{eq:ai}
\begin{lem}\label{lem:a1} Let $c_0=1$ and $c_1=K \frac{\cu^3}{\cl^3}$. We can find a $K>0$ large enough and $0<\eta<1$ such that for an arbitrary initial profile $u_0$, 
\[P(A_1\vert u_0) \le \eta^J.\]
\end{lem}
\begin{rem} The above lemma implies \eqref{eq:ai:unif} because the initial profile is allowed to be arbitrary, and by the Markov property, $A_i$ depends only on the profile $u(t_{i-1},\cdot)$. 
\end{rem}
\begin{rem}\label{rem:markov} {\bf (Important)} Note that the arbitrariness of $u_0$ assumed in the lemma is so that we can apply the Markov property in \eqref{eq:ai:unif}, since a priori there is no restriction on $u(t_{i-1},\cdot)$. The Markov property is used several times in this paper and we shall use the arbitrariness of the profile at time $t_{i-1}$ often. {\it However} the reader should not be confused with Remark \ref{rem:reduction} where we assumed that the initial profile (i.e. at time $0$) for \eqref{eq:she} is $u_0\equiv 0$.
\end{rem}
\begin{proof}
For an arbitrary initial profile $u_0$
\be \label{eq:a1}
P(A_1) = \prod_{j=0}^{J-1} P\left(B_j \,\vert \,B_1, B_2,\cdots, B_{j-1} \right),
\ee
where  
\[ B_j= \left\{\vert u(t_1, x_j+\delta) - u(t_1, x_j)| \le \epsilon^{\frac{1}{2\theta}} \right\}.\]
We will show that each of the terms inside the product sign in \eqref{eq:a1} is uniformly (in $j$) bounded away from $1$, which will imply the lemma.  We will in fact prove a stronger statement that  $P\left(B_j \,\vert\, \mathcal{G}_{j-1}\right)$ is uniformly (in $j$) bounded away from $1$, where  $\mathcal G_{j-1}$ is the $\sigma$ algebra generated by the random variables $\tilde \Delta_k = N(t_1, x_k+\delta) - N(t_1, x_k), \, k\le j-1$.  
We thus need to show the existence of some $0< \eta<1$ such that
\be\label{eq:prob:ubd}
\left(|\Delta_j|\le \epsilon^{\frac{1}{2\theta}} \,\big\vert\, \mathcal G_{j-1}\right) \le  \eta, \ee
where $ \Delta_k := u(t_1, x_k+\delta) - u(t_1, x_k)$. We will obtain this by showing 
\begin{equation}\label{eq:cvar:lbd}\Var \left( \Delta_j\, \Big \vert \, \mathcal G_{j-1} \right) \ge C\epsilon^{\frac{1}{\theta}}, \end{equation}
for some constant $C$ independent of $j$. We can use general properties of Gaussian random vectors to write
\begin{equation} \begin{split}\label{square_brack}
 \Delta_j &=\Big[\left(G_{t_1}*u_0\right)(x_j+\delta) -\left(G_{t_1}*u_0\right)(x_j)\Big] +\tilde \Delta_j\\
 &= \Big[\left(G_{t_1}*u_0\right)(x+j+\delta) -\left(G_{t_1}*u_0\right)(x_j)\Big]+ X+Y,
\end{split}\end{equation}
where 
\[ X= \sum_{k=0}^{j-1} \beta_k \tilde \Delta_k \]
is the conditional expectation of $\tilde \Delta_j$ given $\mathcal G_{j-1}$. The variance of $Y$ is the conditional variance of $\tilde \Delta_j$ given $\mathcal G_{j-1}$, which is also the conditional variance in \eqref{eq:cvar:lbd}. Moreover $Y$ is independent of $\mathcal G_{j-1}$ and thus
\[ \Cov\Big(Y, \,\tilde \Delta_l\Big)=0,\quad  l=0,1,\cdots, j-1.\]
Therefore for all $l=0,1,\cdots, j-1$ we have 
\begin{equation}\label{eq:cov:beta}
 \Cov\Big(\tilde \Delta_j,\, \tilde \Delta_l\Big) 
= \sum_{k=0}^{j-1} \beta_k \Cov \Big(\tilde \Delta_k,\, \tilde \Delta_l\Big)
\end{equation}
Let $\mathbf y= (y_0, y_1,\cdots, y_{j-1})^T$ where $y_l$ represents the entry on the left hand side above, $\boldsymbol \beta$ be the vector of the $\beta_l$'s, and let 
\[ \mathbf S= \left(\left(\Cov (\tilde \Delta_k,\, \tilde \Delta_l) \right)\right)_{0\le k, l\le j-1} \]
be the covariance matrix of the increments $\tilde \Delta_l$. We can thus rewrite \eqref{eq:cov:beta} in matrix form 
\be\label{eq:ysb} \mathbf y = \mathbf S\boldsymbol \beta.\ee

Let us next show that $\bf S$ is invertible. Write $\bf S=D(I-A)D$, where $\mathbf{D}\in \R^{j \times j}$ is the diagonal matrix with diagonal entries
\[ \text{Std}\Big(\tilde \Delta_k\Big),\quad k=0,1, \cdots, j-1,\]
and $\bf I-A$ is the correlation matrix of $\tilde \Delta_l$. Above $\text{Std}$ denotes the standard deviation of the random variable in parentheses. 
Denote by $\|\cdot\|_{1,1}$ the norm on matrices in $\R^{j\times j}$ induced by the $\ell_1$ norm $\|\cdot \|_1$ on $\R^j$. Now $\|{\bf A}\|_{1,1} = \max_j \sum_{i=1}^n |a_{i,j}|$ (see page 259 in \cite{bhim-rao}), we can use \eqref{eq:var:bd} and \eqref{eq:cov:bd} to obtain 
\[\|{\bf A}\|_{1,1} \le \frac{C_2 \mathscr C_2^2}{C_0 \mathscr C_1^2} \sum_{k\ge 1} \exp\left(-\frac{c_1^2k^2}{64 }\right) \le 1/3 \] by choosing $c_1=K \frac{\mathscr C_2^3}{\mathscr C_1^3}$ for a large $K$. In that case the inverse of $\mathbf{I}-\mathbf{A}$ exists and moreover
\[ \|({\bf I-A})^{-1}\|_{1,1} \le \frac{1}{1-\|{\bf A}\|_{1,1}}\le \frac32.\]
Using this along with the lower bound in \eqref{eq:var:bd} we obtain that $\bf S$ is invertible and moreover
\be \label{eq:s}  \|{\bf S}^{-1}\|_{1,1} \le \|{\bf D}^{-1}\|_{1,1} \cdot \|({\bf I-A})^{-1}\|_{1,1}\cdot \|{\bf D}^{-1}\|_{1,1} \le \frac{1}{C_0\mathscr C_1^2\delta}.\ee
Note also from \eqref{eq:cov:bd} 
\be\label{eq:y} \|\mathbf y\|_1\le  C_2\cu^2 \delta \sum_{k=1}^\infty \exp \left(- \frac{ c_1^2 k^2}{64}\right) \leq \frac{\tilde C_2\cu^2\delta}{c_1}\ee
for another constant $\tilde C_2$.

Let us now return to \eqref{eq:ysb} which we write as $\boldsymbol\beta=\mathbf{S}^{-1} \mathbf{y}$. From this we obtain 
\[ \|\boldsymbol \beta\|_1\le \|\mathbf{S}^{-1}\|_{1,1} \cdot \|\mathbf{y}\|_1 \le \frac{\tilde C_2}{C_0}\frac{ \cu^2}{c_1\cl^2 }.\]  
The above quantity can be made arbitrarily small by a choice of a large $K$ in $c_1=K \frac{\cu^3}{\cl^3}$.
 Thus
\[ \text{Std}(X) \le  \|\boldsymbol \beta\|_1\cdot \sup_k \;\text{Std} \big(\tilde \Delta_k\big) \le \frac{\tilde C_2\sqrt{C_1}}{C_0 } \frac{\cu^3}{c_1\cl^2} \sqrt{\delta} \] 
can be made a small multiple of $\cl\sqrt\delta$ by a large choice of $K$. We have used the upper bound in \eqref{eq:var:bd} above. Using this along with the lower bound in \eqref{eq:var:bd} once again we obtain
\begin{equation*}
\text{Std}(Y) \ge \text{Std}\big(\tilde \Delta_j\big) - \text{Std}(X) \ge\left(
\sqrt{C_0}\cl \sqrt{\delta}- \frac{\tilde C_2\sqrt{C_1}}{C_0 } \frac{\cu^3}{c_1\cl^2} \right)\sqrt{\delta}
\end{equation*} 
uniformly in $j$, and this proves \eqref{eq:cvar:lbd} if we choose K large enough. We then get 
the bound in \eqref
{eq:prob:ubd} with 
$\eta =P\left(|N(0,1)|\le \frac{1}{C_3\cl}
\right)
$ for some constant $C_3$.
\end{proof}

\begin{rem}\label{rem:a1:ubd}
One can check that one obtains for small $\delta$
\[ P(A_1\vert u_0) \le  P\left(|N(0,1)|\le \frac{1}{C_3\cl}\right)^{\frac{\cl^3}{K\cu^3}}.
\]
In particular if $\cl$ is large one obtains a bound of 
$\exp\left(-C_4 \frac{\cl^3\log \cl}{K\cu^3}\right)$.
\end{rem}

\begin{rem}\label{rem:32}
We note that if $\sigma(s, y)$ is a constant function, then  Lemma \ref{lem:a1} can be easily proved by using Slepian's inequality. For instance, if $\sigma(s, y)=1$, then it is easy to see that there exist $c_0>0$ and $c_1>0$ such that $G(t_1, z)$ is convex for all $|z|\geq c_1\delta$, and  then the convexity implies that 
\[ \Cov( \tilde\Delta_k, \tilde\Delta_l) \leq 0.\]
Using now Slepian's inequality, we get 
\[ P\left(  \max_{k} |\tilde\Delta_k| \leq \epsilon^{1/2\theta} \right) \leq P\left(  \max_{k} \tilde\Delta_k \leq \epsilon^{1/2\theta} \right) \leq  \prod_{k} P\left( \tilde\Delta_k \leq \epsilon^{1/2\theta} \right),\]
which provides an upper bound on $P(A_1)$ as in \eqref{eq:a1} since $\tilde\Delta_k$ is mean-zero Gaussian with variance estimated in \eqref{eq:var:bd}.
\end{rem}

\subsection{Upper bound in Theorem \ref{thm1} (\lowercase{b})}

We provide the outline of the proof of the upper bound. The details are quite similar to that of the proof of the upper bound of Theorem \ref{thm1} (a) and are left to the reader.  Using the same discrete time-space points as in \eqref{eq:t:x} we obtain 
\begin{equation}\label{eq:ubd:dis:t} P \left(\sup_{x\in \T}\HS^{(\theta)}_x(u) \le \epsilon\right) 
 \le P\left(\max_{\stackrel{i=0,1,\cdots, I}{j=0,1\cdots J}} \big|u(t_i+\delta^2, x_j) -u(t_i, x_j)\big|\le \epsilon^{\frac{1}{2\theta}}\right).
\end{equation} 
Defining
\[A_i^{\#} := \left\{\max_{j=0,1,\cdots, J} \big|u(t_i+\delta^2, x_j)- u(t_i, x_j)\big|\le \epsilon^{\frac{1}{2\theta}}\right\},\]
the upper bound will follow once we show the existence of a $0<\tilde \eta<1$ such that 
\[P\left( A_i^{\#} \;\Big \vert \;u(s,x),\, s\le t_{i}, \, x\in \T \right) \le \tilde\eta^J.\]
Note here the slight change from \eqref{eq:ai:unif}; here we condition on the profile up to $t_i$. One could have conditioned up to time $t_{i-1}$ but conditioning up to time $t_i$ makes the argument simpler. Define
\[ \tilde \Delta^{\#}_k := N(\delta^2, x_k).\] 
Using \eqref{eq:var:t_incr}, one obtains a similar result to Lemma \ref{lem:var:covar} with the random variables $\tilde \Delta_k$ replaced by $\tilde \Delta^{\#}_k$. We also need 
\begin{align*}
\cov\left(\tilde \Delta^{\#}_k ,\tilde \Delta^{\#}_l\right) &\le \cu^2 \int_0^{\delta^2} \int_{\mathbf T} G(\delta^2-s, x_k-y) G(\delta^2-s,x_l-y) \, dy ds \\
&\le \cu^2 \int_0^{\delta^2} G(2s, x_k, x_l) ds \\
&\le C\cu^2\delta \exp\left(-\frac{|x_k-x_l|^2}{2\delta^2}\right).
\end{align*} 
By the Markov property again, we only need to show 
\[P(A_0^{\#}\vert u_0) \le \tilde\eta^J\]
for some $0<\tilde \eta<1$. For this we note 
\bes 
P( A_0^{\#}) = \prod_{j=0}^{J-1} P\left(B_j^{\#} \,\vert \,B_1^{\#}, B_2^{\#},\cdots,B_{j-1}^{\#} \right),
\ees
where  
\[B_j^{\#}:= \left\{\vert u(\delta^2, x_j) - u(0, x_j)| \le \epsilon^{\frac{1}{2\theta}} \right\}.\]
Define $\Delta^{\#}_k=u(\delta^2, x_k)-u(0, x_k)$ and then show 
\[ P\left(|\Delta^{\#}_j|\le \epsilon^{\frac{1}{2\theta}} \,\big\vert\, \mathcal G_{j-1}\right) \le \tilde \eta\]
for some $0<\tilde \eta<1$, where  $\mathcal G_{j-1}$ is the $\sigma$ algebra generated by the random variables $\tilde \Delta^{\#}_k,\, k\le j-1$. Note that although 
\[ \Delta^{\#}_j =\Big[\left(G_{\delta^2}*u_0\right)(x_j) -u_0(x_j)\Big] +\tilde \Delta^{\#}_j\]
is of a slightly different form than that of \eqref{square_brack}, the term in the square brackets does not play any role in the argument of Lemma \ref{lem:a1}.

\subsection{Upper bound in Theorem \ref{thm2} (\lowercase{a})}

The function $\sigma(t,x,u)$ now depends  on the third variable, so the resulting random field is no longer Gaussian.  Therefore, we will need an alternative argument based on an approximation procedure.  For $\beta>0$, we define the following equation,
\begin{equation}\label{truncate}
V^{(\beta)}(t,\,x)=(G_{t}* u_0)(x)+\int_0^{t} \int_{[x-\sqrt{\beta t}, x+\sqrt{\beta t}]}G(t-s, x-y)\sigma\left(s,y,V^{(\beta)}(s,y)\right) W(dsdy).
\end{equation}
Of course, here, we treat $x\pm \sqrt{\beta t} \in \T$. 

Existence and uniqueness of a  solution to the above equation is not an issue. In fact, this can be easily proved by considering the following Picard iterates:
\begin{equation}\label{iterates}
V^{(\beta),l}(t,\,x)=(G_{t}* u_0)(x)+\int_0^{t} \int_{[x-\sqrt{\beta t}, x+\sqrt{\beta t}]}G(t-s, x-y)\sigma\left(s,y,V^{(\beta),l-1}(s,y)\right) W(dsdy),
\end{equation}
with $V^{(\beta), 0}(t,\,x):=(G_{t}* u_0)(x).$ We will need the following result.
\begin{prop}\label{prop:V-Vl}
Assume $\beta t< \frac14$. There exist positive constants $C_1$ and $C_2$ that are independent of $\beta$ and $t$ such that 
\begin{equation*}
\sup_{x\in \T}E\left[| V^{(\beta), l}(t,\,x)-V^{(\beta)}(t,\,x)|^p\right]\leq \left(\frac{C_1\cu^2}{\du^2}\right)^{p/2} e^{C_2\du^4p^3 t} \left(\frac{1}{2}\right)^{lp/2}
\end{equation*}
\end{prop}
\begin{proof}
 We use \eqref{truncate} and \eqref{iterates} to write 
 \begin{align*}
 &V^{(\beta), l}(t,\,x)-V^{(\beta)}(t,\,x)\\
 &:=\int_0^{t} \int_{[x-\sqrt{\beta t}, x+\sqrt{\beta t}]} G(t-s, x-y)\left[\sigma\left(s,y,V^{(\beta),l-1}(s,y)\right)-\sigma\left(s,y,V^{(\beta)}(s,y)\right)\right] W(dsdy).
 \end{align*}
For notational convenience, we set $f(l,t):=\sup_{x\in \T}\| V^{(\beta), l}(t,\,x)-V^{(\beta)}(t,\,x)\|_p^2$. We now use Burkholder's inequality and the fact that $\sigma$ is globally Lipschitz to obtain
\begin{align*}
f(l,t)&\leq C\du^2p\int_0^{t}f(l-1,s) \left[ \sup_{x\in \T} \int_{[x-\sqrt{\beta t}, x+\sqrt{\beta t}]} G^2(t-s, x-y) dy \right] ds\\
&\leq C\du^2p \int_0^{t}f(l-1,s) \int_{\T} G^2(t-s, y) \, dy \\ 
&\leq C\du^2p\int_0^{t}\frac{f(l-1,s)}{\sqrt{t-s}}d s, 
\end{align*} 
where we have used the heat kernel estimate \eqref{eq:heat-kernel} to get the last bound in the above (here the value of the  constant $C$ changes from  line to line, and is independent of $\beta$ and $t$). Upon setting $F(l):=\sup_{s>0}e^{-ks}f(l,s)$, the above immediately yields 
\begin{equation*}
F(l)\leq \frac{C\du^2p}{\sqrt{k}}F(l-1).
\end{equation*}
Upon choosing $k=C\du^4p^2$ with some large constant $C$ and iterating, we obtain $F(l)\leq C\left(\frac{1}{2}\right)^l.$ This along with 
\[F(0) \le \cu^2\sup_{t\ge 0} e^{-kt} \int_0^t \int_{[x-\sqrt{\beta t} , x+\sqrt{\beta t} ]}G^2(t-s, x-y)\, dy ds \le \frac{C\cu^2}{\du^2} \]
gives the result. 
\end{proof}

We also have the following error estimate on the difference between $u$ and $V^{(\beta)}$.
\begin{prop}\label{prop:u-V}
Assume $\beta t< \frac14$. Then there exist positive constants $C_1$ and $C_2$ that are independent of $\beta$ and $t$ such that  
\begin{equation*}
\sup_{x\in\T}E\left[| u(t,\,x)-V^{(\beta)}(t,\,x)|^p\right]\leq \left(\frac{C_1\cu^2}{\du^2}\right)^{p/2} e^{C_2\du^4p^3 t} e^{-\beta p/4}
\end{equation*}
\end{prop}

\begin{proof}
We use \eqref{truncate} and the mild formulation of the $u(t,x)$ to write 
\begin{align*}
&u(t,x)-V^{(\beta)}(t,\,x)\\
&=\int_0^t\int_{[x-\sqrt{\beta t}, x+\sqrt{\beta t}]} G(t-s, x-y)\left[\sigma\left(s,y,u(s,y)\right)-\sigma\left(s,y,V^{(\beta)}(s,y)\right)\right] W(dsdy)\\
&\quad+\int_0^t\int_{[x-\sqrt{\beta t}, x+\sqrt{\beta t}]^c} G(t-s, x-y)\cdot\sigma\left(s,y,u(s,y)\right)W(dsdy).
\end{align*}
We now use Burkholder's inequality together with the Lipschitz continuity of $\sigma$ to write
\begin{align*}
\|u(t,x) & -V^{(\beta)}(t,\,x)\|_p^2\\
&\leq C\du^2 p\int_0^t\int_{[x-\sqrt{\beta t}, x+\sqrt{\beta t}]} G^2(t-s, x-y)\|u(s,y)-V^{(\beta)}(s,y)\|^2_p dsdy\\
&\quad+Cp\int_0^t\int_{[x-\sqrt{\beta t}, x+\sqrt{\beta t}]^c} G^2(t-s, x-y)\|\sigma\left(s,y,u(s,y\right)\|_p^2 dsdy\\
&:=I_1+I_2.
\end{align*}
We bound $I_2$ first. We now use the heat kernel estimate \eqref{eq:heat-kernel} and the fact that $\sigma$ is bounded above to get that  
\begin{align*}
I_2&\leq C\cu^2pe^{-\beta/2}\int_0^t\frac{1}{\sqrt{t-s}}d s\\
&\leq C\cu^2 pe^{-\beta/2}\sqrt t.
\end{align*}
We now set $F(k):=\sup_{t>0,\,x\in\T}e^{-kt}\|u(t,x)-V^{(\beta)}(t,\,x)\|_p^2$ and bound $I_1$ as in the Proposition above to obtain
\begin{equation*}
F(k)\leq \frac{C\du^2p}{\sqrt{k}}F(k)+ \frac{C\cu^2 p}{\sqrt{k}}e^{-\beta/2}.
\end{equation*}
This finishes the proof upon choosing the $k=C\du^4 p^2$ for a large constant $C$.
\end{proof}
We will use the following straightforward consequence of the above: 
\begin{equation}\label{b}
\sup_{x\in\T}E\left[| u(t,\,x)-V^{(\beta),l}(t,\,x)|^p\right]\leq \left(\frac{D_1\cu^2}{\du^2}\right)^{p/2} e^{D_2\du^4p^3 t}\left[ e^{-\beta p/4} +\left(\frac{1}{2}\right)^{lp/2}
\right]
\end{equation}
where $D_1$ and $D_2$ are some positive constants. The following lemma along with \eqref{b} suggests that we can construct independent random variables that are close to $u(t, x)$. The proof of Lemma \ref{lemma:ind} is essentially the same as that of Lemma 4.4 of \cite{conu-jose-khos}. 

\begin{lem}\label{lemma:ind}
Let $\beta, t>0$ and $l\geq0$. Fix a collection of points $x_1,x_2,\cdots \in \T$ such that the distance between $x_i$ and $x_j$ is greater than $2l\sqrt{\beta t}$ whenever $i\not=j$. Then $\{V^{(\beta),l}(t,\,x_j)\}$ forms a collection of independent random variables.
\end{lem}

We  can now prove the upper bound. 
Recall  $\delta=\epsilon^{\frac 1\theta}$ and the time points $t_i:=i \delta^2$ as in \eqref{eq:t:x}; we have chosen $c_0=1$. We shall consider now the spatial points $x_{2j}:=j(\delta+\rho)$ and $x_{2j-1}:=j(\delta+\rho)-\rho$ for $j=1, \dots, J$ where $J:=[1/2(\delta+\rho)]$ and $\rho:= 2|\alpha\log \epsilon|^{\frac3 2}\delta$. Here $\alpha>4(D_2\du^4+1)+16\theta$ is a constant which is independent of $\epsilon$ and  $i, j$ where $D_2$ is in \eqref{b}. From this definition, we have $|x_{2j+1} - x_{2j}| = \delta$ and $|x_{2j+2}-x_{2j+1}| =\rho$ for $j=0, \dots, J$.  As in the proof of the upper bound for the Gaussian case, we have 

\begin{equation*} \begin{split}P \left(\sup_{\stackrel{x\ne y\in \T}{0\le t\le T}} \frac{|u(t,x)-u(t,y)|}{|x-y|^{\frac12-\theta}} \le \epsilon\right) & \le P\left(\max_{\stackrel{i=0,1,\cdots, I}{j=0,1\cdots J-1}} \frac{\big|u(t_i, x_{2j+1}) -u(t_i, x_{2j})\big|}{\delta^{\frac12-\theta}}\le \epsilon\right) \\
& \le P\left(\max_{\stackrel{i=0,1,\cdots, I}{j=0,1\cdots J}} \big|u(t_i, x_{2j}+\delta) -u(t_i, x_{2j})\big|\le \epsilon^{\frac{1}{2\theta}}\right).
\end{split}
\end{equation*}
We will show below that {\it uniformly} over initial profiles $u_0$ (see Remark \ref{rem:markov})
\begin{equation}\label{eq:prob-main}
 P \left(\max_{j=0,1,\cdots, J} \big|u(t_1, x_{2j}+\delta)- u(t_1, x_{2j})\big|\le \epsilon^{\frac{1}{2\theta}}\right) \le \exp\left(-\frac{C}{|\log \epsilon|^{\frac32} \epsilon^{\frac1\theta}}\right)
 \end{equation}
for some positive constant $C$. One then uses \eqref{eq:ubd:dis} and the Markov property, and notes that the number of time intervals $I=\left[\frac{T}{\delta^2}\right]=\left[ \frac{T}{\epsilon^{2/\theta}}\right]$ to get the required upper bound.

Let us therefore turn to the proof of \eqref{eq:prob-main}. Using the triangle inequality, the left hand side of \eqref{eq:prob-main} is bounded above by
\begin{align*}
& 2 P\left( \max_{j=0,1,\cdots, 2J} \left| u(t_1, x_{j}) - V^{(\beta),l}(t_1, x_j)  \right|  > \epsilon^{1/2\theta} \right) \\
&\qquad \qquad + P\left( \max_{j=0,1,\cdots, J} \left| V^{(\beta),l}(t_1, x_{2j+1}) - V^{(\beta),l}(t_1, x_{2j})  \right|  \leq 3\epsilon^{1/2\theta} \right)  \\
&:= L_1 + L_2. 
\end{align*} 
Before we consider $L_1$ and $L_2$, we  define 
\begin{equation}\label{eq:beta}
 \beta=l:=\left\lfloor\alpha|\log \epsilon|\right\rfloor \quad \text{and} \quad p:=\left\lfloor\sqrt{|\log \epsilon| / \delta^2}\right\rfloor.
 \end{equation}
Let us now consider $L_1$ first. By Chebyshev's inequality and \eqref{b} there exist constants $C_1>0$ and $C_2>0$ which are independent of $\epsilon$ such that  
\begin{equation}\label{eq:u-v1}
\sup_{x\in \T} P \left( \left|u(t_1, x) - V^{(\beta),l} (t_1, x) \right|  \geq   \epsilon^{1/2\theta} \right) \leq \epsilon^{-p/2\theta}\left(\frac{D_1\cu^2}{\du^2}\right)^{p/2} e^{D_2\du^4p^3 t}\left[ e^{-\beta p/4} +\left(\frac{1}{2}\right)^{lp/2}
\right] 
\end{equation}
Since $2J\leq 1/(\delta+\rho)\leq 1/\delta=\epsilon^{-1/\theta}$, we have for some other positive constants $\tilde C_1$ and $\tilde C_2$ independent of $\epsilon$ 
\begin{equation} \label{eq:u-v1:2} L_1\leq   2J \sup_{x\in \T}P \left( \left|u(t_1, x) - V^{(\beta),l} (t_1, x) \right|  \geq  \epsilon^{1/2\theta} \right) \le \tilde C_1 \exp\left( -  \frac{\tilde C_2 \du^4|\log \epsilon|^{3/2}}{\epsilon^{1/\theta}}\right).\end{equation} 
Let us now consider $L_2$. First observe that $W_j:=\left(V^{(\beta),l}(t_1, x_{2j}), V^{(\beta),l}(t_1, x_{2j+1})  \right)$ are independent for $j=0, 1, \dots, J-1$  by Lemma \ref{lemma:ind} since the distance between $x_{2j+1}$ and $x_{2j+2}$ is greater than $2l^{3/2}\sqrt{t_1}$. Thus, we have 
\begin{align*}
L_2=&P\left( \max_{j=0,1,\cdots, J} \left| V^{(\beta),l}(t_1, x_{2j+1}) - V^{(\beta),l}(t_1, x_{2j})  \right|  \leq 3\epsilon^{1/2\theta} \right)\\
 &= \prod_{j=1}^J P\left( \left| V^{(\beta),l}(t_1, x_{2j+1}) - V^{(\beta),l}(t_1, x_{2j})  \right|  \leq 3\epsilon^{1/2\theta} \right).
\end{align*}
Using the triangle inequality, we have 
\begin{equation}\label{a2}
\begin{split}
&P\left( \left| V^{(\beta),l}(t_1, x_{2j+1}) - V^{(\beta),l}(t_1, x_{2j})  \right|  \leq 3\epsilon^{1/2\theta} \right) \\
&\leq  2 \max_{0\leq j\leq 2J} P\left( \left| u(t_1, x_{j}) - V^{(\beta),l}(t_1, x_j)  \right|  > \epsilon^{1/2\theta} \right)  + P\left( \left| u(t_1, x_{2j+1}) - u(t_1, x_{2j})  \right|  \leq 5\epsilon^{1/2\theta} \right)\\
&=:L_{21}+L_{22}.
\end{split}
\end{equation}
Let us first consider $L_{22}$. Consider the following martingale $M_s$ for  $0\le s\le t_1$:
\begin{align*}
M_s &= \left[\left(G_{t_1}*u_0\right)(x_{2j+1}) - \left(G_{t_1}*u_0\right) (x_{2j})\right] \\
&\qquad  + \int_0^s \int_{\T} \left[G(t_1-r, x_{2j+1}- y) - G(t_1-r, x_{2j}-y)\right] \cdot \sigma\left(r, y, u(r,y)\right) W(dy dr).
\end{align*}
Note that $M_{t_1}$ is $u(t_1, x_{2j+1}) - u(t_1, x_{2j})$. The quadratic variation of the martingale is given by 
\[ \langle M\rangle_s=\int_0^s \int_{\T} \left[G(t_1-r, \,x_{2j+1}- y) - G(t_1-r, x_{2j}-y)\right]^2 \sigma\left(r, y, u(r,y)\right)^2 dy dr.\]
We use \eqref{eq:var:bd} to obtain 
\[ C_0\cl^2\delta \le \langle M\rangle_{t_1} \le C_1\cu^2\delta.\]
Since $M_t$ is a continuous martingale, it is a time change of a Brownian motion $B$, i.e.,  $M_t= M_0 + B_{\langle M\rangle_t}$.
Hence, recalling $\delta=\epsilon^{1/\theta}$, we have  
\begin{equation}\label{eq:diff-u-x}
\begin{split} P\left( \left| u(t_1, x_{2j+1}) - u(t_1, x_{2j})  \right|  \leq 5\epsilon^{1/2\theta} \right) & \le P\left(\left| M_0 + B_{\langle M\rangle_{t_1}}\right| \le 5\epsilon^{1/2\theta}\right) \\
 &\le P\left(\inf_{C_0\cl^2\delta \le t\le C_1\cu^2\delta}\left|M_0+ B_t\right|\le 5\sqrt\delta\right)  \\
 &\le P\left(\inf_{C_0\cl^2\delta \le t\le C_1\cu^2\delta}\left|B_t\right|\le 5\sqrt\delta\right)=:\gamma,
\end{split}
\end{equation}
for some $\gamma<1$ independent of $\delta$ (but dependent on $\cl, \cu$). The last inequality can be obtained by a coupling argument 
as follows: Let $B^{M_0}$ be a Brownian motion starting at $M_0$ independent of a standard Brownian motion $B$ starting at $0$. Now let $X$ be the process which follows the trajectory of $B^{M_0}$ till it hits either $B$ or $
-B$,
 after which it follows the trajectory of $B$ or $-B
 $ (depending on which one it hits). Clearly $X$ has the same distribution as $B^{M_0}$. If $\inf_{C_0\cl^2\delta\le t\le C_1\cu^2\delta} |X_t
 |<5\sqrt \delta$ then both the events
  $\left\{
  \inf_{C_0\cl^2\delta\le t\le C_1\cu^2\delta} |B
  _t|>5\sqrt \delta\right\}$
  and $\left\{
  \inf_{C_0\cl^2\delta\le t\le C_1\cu^2\delta} |-B
  _t|>5\sqrt \delta\right\}$ cannot occur simulataneously
  since then $B^{M_0}$ would have hit $B$ or $-B$ before entering the strip $[-5\sqrt \delta, 5\sqrt \delta]$. Thus
 \begin{align*}
 \left\{
  \inf_{C_0\cl^2\delta\le t\le C_1\cu^2\delta} |X_t|\le 5\sqrt \delta\right\}&\subset
   \left\{
  \inf_{C_0\cl^2\delta\le t\le C_1\cu^2\delta} |B
  _t|\le 5\sqrt \delta\right\} \bigcup \left\{
  \inf_{C_0\cl^2\delta\le t\le C_1\cu^2\delta} |-B_
  t|\le 5\sqrt \delta\right\} \\
  &= \left\{
  \inf_{C_0\cl^2\delta\le t\le C_1\cu^2\delta} |B_
  t|\le 5\sqrt \delta\right\}
  \end{align*}


Let us now consider  $L_{21}$.  Here, \eqref{eq:u-v1:2} implies $L_{21}$ can be made arbitrarily small by choosing $\epsilon$ small enough. Therefore, there exists a constant $\eta<1$ independent of $\epsilon$ such that
\[ L_{21}+L_{22} \leq \eta <1,\]
which implies from \eqref{a2}
\[ L_{2} \leq \eta^J\leq \exp \left( - \frac{C}{\epsilon^{1/\theta} |\log \epsilon|^{3/2} }  \right).\]
Combining our bounds on $L_1$ and $L_2$, we  finish the proof.  \qed

\subsection{Upper bound in Theorem \ref{thm2} (\lowercase{b})} 
The proof follows a similar strategy to that of the upper bound proved above and  we will sketch the proof focusing on the main differences. Note that we use  the same choice of $\beta$ and $l$ as in part (a) (see \eqref{eq:beta}).   Then, we have 
\begin{equation*} P \left(\sup_{x\in \T}\HS^{(\theta)}_x(u) \le \epsilon\right) 
 \le P\left(\max_{\stackrel{i=0,1,\cdots, I}{j=0,1\cdots J}} \big|u(t_i+\delta^2, x_j) -u(t_i, x_j)\big|\le \epsilon^{\frac{1}{2\theta}}\right),
\end{equation*} 
where the points $t_i=\delta^2$ while $x_j= 4j |\alpha\log \epsilon|^{3/2}\delta,\; j=0,\,1,\, \cdots, J:= \left[\frac{1}{c_1|\log \epsilon|\delta}\right]$. Here we choose $\alpha$ as in part (a) such that $|x_i-x_j| \geq 2\ell \sqrt{\beta(t_1+\delta^2)}$. In other words, by our choices of $x_j$, $\left\{V^{(\beta, l)}(t_1+\delta^2, x_j) -V^{(\beta),l}(t_1, x_j)\right\}_{j=0}^{J}$
is a collection of independent random variables. Now  we have 
\begin{align*}
P\left(\big|u(t_1+\delta^2, x_j) -u(t_1, x_j)\big|\le \epsilon^{\frac{1}{2\theta}}\right)&\leq P\left(\big|u(t_1+\delta^2, x_j) -V^{(\beta),l}(t_1+\delta^2, x_j)\big|\ge \epsilon^{\frac{1}{2\theta}}\right)\\
&+P\left(\big|u(t_1, x_j) -V^{(\beta),l}(t_1, x_j)\big|\ge \epsilon^{\frac{1}{2\theta}}\right)\\
&+P\left(\big|V^{(\beta), l}(t_1+\delta^2, x_j) -V^{(\beta),l}(t_1, x_j)\big|\le 3\epsilon^{\frac{1}{2\theta}}\right).
\end{align*}
For the first two terms,  we have similar upper bounds as the one given by \eqref{eq:u-v1};
\begin{align*}
P\left(\big|u(t_1+\delta^2, x_j) -V^{(\beta),l}(t_1+\delta^2, x_j)\big|\ge \epsilon^{\frac{1}{2\theta}}\right)&+P\left(\big|u(t_1, x_j) -V^{(\beta),l}(t_1, x_j)\big|\ge \epsilon^{\frac{1}{2\theta}}\right)\\
&\leq C_1 \exp\left( -  \frac{C_2|\log \epsilon|^{3/2}}{\epsilon^{1/\theta}}\right),
\end{align*}
for some positive constants $C_1$ and $C_2$. For the final term, we have
\begin{align*}
P\left(\big|V^{(\beta), l}(t_1+\delta^2, x_j) -V^{(\beta),l}(t_1, x_j)\big|\le 3\epsilon^{\frac{1}{2\theta}}\right)&\leq P\left(\big|V^{(\beta), l}(t_1+\delta^2, x_j) -u(t_1+\delta^2, x_j)\big|\ge \epsilon^{\frac{1}{2\theta}}\right)\\
&\;+P\left(\big|u(t_1, x_j) -V^{(\beta), l}(t_1, x_j)\big|\ge \epsilon^{\frac{1}{2\theta}}\right)\\
&\;+P\left(\big|u(t_1+\delta^2, x_j) -u(t_1, x_j)\big|\le 5\epsilon^{\frac{1}{2\theta}}\right).
\end{align*}
The bound for the last term is similar to the bound given by \eqref{eq:diff-u-x}.  The martingale term is slightly different. For 
$0\le s\le t_1+\delta^2$
\begin{align*}
M_s &= \left[\left(G_{t_1+\delta^2}*u_0\right)(x_{j}) - \left(G_{t_1}*u_0\right) (x_{j})\right] \\
&\qquad  + \int_0^s \int_{\T} \left[G(t_1+\delta^2-r, x_{j}- y) - G(t_1-r, x_{j} - y)1_{r\leq t_1}\right] \cdot \sigma\left(r, y, u(r,y)\right) W(dy dr).
\end{align*}
We now use \eqref{eq:var:t_incr} and \eqref{eq:var:t:g-g} to show that there exist constants $C_3$ and $C_4$ such that  
\[ C_3\cl^2\delta \le \langle M\rangle_{t_1+\delta^2} \le C_4\cu^2\delta.\]
A similar argument to that of \eqref{eq:diff-u-x} shows that 
\begin{equation*}
P\left(\big|u(t_1+\delta^2, x_j) -u(t_1, x_j)\big|\le 5\epsilon^{\frac{1}{2\theta}}\right)\leq \gamma,
\end{equation*}
where $\gamma<1.$ The proof now follows from part (a).
\section{Lower bounds}\label{sec:lbd}

\subsection{Lower bound in Theorem \ref{thm1} (\lowercase{a})}\label{sec:lb:Gaussian}
Recall our time discretizations from \eqref{eq:t:x}: $t_i=ic_0\delta^2= ic_0 \epsilon^{\frac2\theta},\, i=0,1,\cdots, I$, and consider now the events
\be\label{eq:B} B_i= U_i \cap H_i,\ee
where the event $U_i$  puts restriction on the supremum norm of $u(t,\cdot)$ in the time interval $[t_i, t_{i+1}]$:
\be \label{eq:U} U_i =  \left\{\sup_{x\in \T} \left|u(t_{i+1}, x)\right| \le \frac{ \epsilon^{\frac{1}{2\theta}}}{6} , \text{ and }\sup_{x\in \T} \left|u(t,x)\right|\le \frac{ 2\epsilon^{\frac{1}{2\theta}}}{3} \text{ for all } t\in [t_i, t_{i+1}]\right\},\ee
and the event $H_i$ puts restriction on the H\"older norm of $u$ in the time interval $[t_i, t_{i+1}]$:
\be\label{eq:H} H_i=\left\{\mathcal{H}_{t_{i+1}}^{(\theta)}(u)\le \frac\epsilon 6, \text{ and } \mathcal H_t^{(\theta)}(u)\le \frac{2\epsilon}{3} \text{ for all } t\in [t_i, t_{i+1}]\right\}.\ee 
It is clear that 
\be \label{eq:lbd:dis} \begin{split}
P \left(\sup_{0\le t\le T} \mathcal H_t^{(\theta)}(u)\le \epsilon \right)&\ge P \left(\sup_{0\le t\le T} \mathcal H_t^{(\theta)}(u)\le \epsilon,\, \sup_{\stackrel{0\le t\le T}{x\in[0,\,1]}}|u(t,\,x)|\leq \epsilon^{\frac{1}{2\theta}}\right)\\
 & \ge P \left( \cap_{i=0}^{I-1} B_i\right) \\
&= \prod_{i=0}^{I-1} P\left(B_i \big \vert B_0, B_1\cdots B_{i-1} \right).
\end{split}
\ee
Similar to the method of the upper bound, our main task will be to obtain a uniform lower bound on $ P\left(B_i \big \vert B_0, B_1\cdots B_{i-1} \right)$. It turns out that with an appropriate choice of $c_0$ one can in fact obtain such a uniform lower bound. We do this in Lemma \ref{lem:B0} below (see also Remark \ref{rem:B0}), and then the lower bound in Theorem \ref{thm1} (a) follows immediately. We first  need a couple of lemmas which we turn to next. 
 
  \begin{lem}\label{lem:ho:lbd}   There exists a constant $\mathbf K_5$ dependent only on $\theta$ and $\alpha_0 = \frac{C(\theta)}{\cu^{2/\theta}}>0$ such that for $\alpha<\alpha_0$ we have
   \be \label{eq:ho:lbd} P \left(\sup_{t\le \alpha\epsilon^{\frac 2\theta},\;x\ne y \in \T} \big\vert \widetilde N(t,x,y) \big \vert \le \epsilon\right) \ge  \exp\left(- \frac{2}{ \alpha^{\frac12}\epsilon^{\frac1\theta}}\exp\left(-\frac{\mathbf K_5}{\cu^2\alpha^\theta}\right)\right). \ee
 \end{lem}
\begin{proof} We first split $\T^2$ into squares $S$ of side length $\alpha^{\frac12}\epsilon^{\frac 1\theta}$. By the Gaussian correlation inequality (\cite{lata-matl}, \cite{roye}) we have 
\[  P \left(\sup_{t\le \alpha\epsilon^{\frac 2\theta},\;x\ne y \in \T} \big\vert \widetilde N(t,x,y) \big \vert \le \epsilon\right) \ge \prod_S   P \left(\sup_{\stackrel{t\le \alpha \epsilon^{\frac 2\theta},}{(x,y) \in S,\; x
\ne y}} \big\vert \widetilde N(t,x,y) \big \vert \le \epsilon\right).\]
For $k=0,1,\cdots,\alpha^{-\frac12}\epsilon^{- \frac 1\theta}-1$, let $S_k$ be a square in $\T^2$ whose center is $k2^{-\frac12}\alpha^{\frac12}\epsilon^{\frac 1\theta}$ from the diagonal  $x=y$. There are at most $2\alpha^{-\frac12}\epsilon^{-\frac 1\theta}$ of such squares.  Therefore the above probability is bounded below by 
\be\label{eq:Nt}\begin{split}  &\prod_{k=0}^{\alpha^{-\frac12}\epsilon^{- \frac 1\theta}-1} \left[P\left(\sup_{\stackrel{t\le \alpha\epsilon^{\frac 2\theta},}{(x, y) \in S_{k},\, x\ne y}} \big\vert \widetilde N(t,x,y) \big \vert \le \epsilon\right)\right]^{2\alpha^{-\frac12}\epsilon^{-\frac{1}{\theta}}}.
\end{split}
\ee
Let us now give a lower bound of the expression inside the square brackets. We first consider the case when $k\ge 1$. For any $(x,y)\in S_k$ one has a lower bound $|x-y|\ge \frac{1}{4}\left(k\alpha^{\frac1\theta}\epsilon^{\frac1\theta}\right)$ and therefore  
\begin{equation}
\begin{split}
  &P\left(\sup_{t\le \alpha\epsilon^{\frac 2\theta},\;(x,y) \in S_{k}} \big\vert \widetilde N(t,x,y) \big \vert \le \epsilon\right) \\
  & \ge P\left(\sup_{t\le \alpha\epsilon^{\frac 2\theta},\;(x,y) \in S_{k}} \big\vert  N(t,x) - N(t,y) \big \vert \le \frac\epsilon 4 \left(k  \alpha^{\frac12}\epsilon^{\frac 1\theta}\right)^{\frac12-\theta}\right) \\
 & \ge P\left(\sup_{t\le \alpha\epsilon^{\frac 2\theta},\;(x, y) \in S_{k}} \max\Big\{ \big\vert  N(t,x) \big \vert,\,  \big\vert N(t, y) \big\vert \Big\}\le \frac\epsilon 8 \left(k \alpha^{\frac12}\epsilon^{\frac 1\theta}\right)^{\frac12-\theta}\right)\\
 &\ge 1- 2\mathbf K_1\exp \left(- \mathbf K_2\frac{k^{1-2\theta}}{64\cu^2\alpha^\theta} \right),
\end{split}
\end{equation}
the last inequality follows from \eqref{eq:N:tail}. Therefore there exists an $\alpha_1=\frac{C(\theta)}{\cu^{2/\theta}}>0$ small enough such that for all positive $\alpha< \alpha_1$ one has
\[P\left(\sup_{t\le \alpha\epsilon^{\frac 2\theta},\;(x,y) \in S_k} \big\vert \widetilde N(t,x,y) \big \vert \le \epsilon\right)\ge 1- \exp \left(- \frac{\mathbf K_2k^{1-2\theta}}{128\cu^2\alpha^\theta} \right).\] Returning to \eqref{eq:Nt} we can obtain a lower bound on the product of terms for which $k\ne 0$ by choosing an $\alpha_2=\frac{C(\theta)}{\cu^{2/\theta}}>0$ small enough such that for $\alpha<\alpha_2$ we have
\be\label{eq:Nt:2}
\begin{split}
&\prod_{k=1}^{\alpha^{-\frac12}\epsilon^{- \frac 1\theta}-1} \left[P\left(\sup_{\stackrel{t\le \alpha\epsilon^{\frac 2\theta},}{(x, y) \in S_{k},\, x\ne y}} \big\vert \widetilde N(t,x,y) \big \vert \le \epsilon\right)\right]^{2\alpha^{-\frac12}\epsilon^{-\frac{1}{\theta}}}\\
 &\ge \exp\left(2\alpha^{-\frac12}\epsilon^{-\frac 1\theta}\sum_{k=1}^{\alpha^{-\frac12}\epsilon^{-\frac 1\theta}-1} \log \left\{1- \exp \left(- \frac{\mathbf{K_2}k^{1-2\theta}}{128\cu^2\alpha^\theta} \right)\right\}\right) \\
 &\ge \exp\left(-2\alpha^{-\frac12}\epsilon^{-\frac 1\theta}\sum_{k=1}^{\alpha^{-\frac12}\epsilon^{-\frac 1\theta}-1}  \exp \left(- \frac{\mathbf K_2k^{1-2\theta}}{128\cu^2\alpha^\theta} \right)\right) \\
 &\ge \exp\left(-2\alpha^{-\frac12} \exp\left(-\frac{\mathbf K_2}{256\cu^2\alpha^\theta}\right)\cdot \epsilon^{-\frac1\theta}\right). 
\end{split}
\ee
Finally we consider the $k=0$ term in \eqref{eq:Nt}. For a small $\alpha_3=\frac{C(\theta)}{\cu^{2/\theta}}>0$ one has for $\alpha<\alpha_3 $ 
\be\label{eq:Nt:1}
\begin{split}
&\left[P\left(\sup_{\stackrel{t\le \alpha\epsilon^{\frac 2\theta},}{(x, y) \in S_{0},\, x\ne y}} \big\vert \widetilde N(t,x,y) \big \vert \le \frac{( \alpha^{\frac12}\epsilon^{\frac1\theta})^{\theta}}{\alpha^{\frac\theta2}}\right)\right]^{2\alpha^{-\frac12}\epsilon^{-\frac1\theta}}\\
&\ge \left[1-\exp\left(-\frac{\mathbf K_4}{2\cu^2\alpha^\theta}\right)\right]^{2\alpha^{-\frac12}\epsilon^{-\frac1\theta}} \\
&\ge  \exp\left(-4\alpha^{-\frac12} \exp\left(-\frac{\mathbf K_4}{ 2\cu^2\alpha^\theta}\right)\cdot \epsilon^{-\frac1\theta}\right),
\end{split}
\ee
where the first inequality  follows by Lemma \ref{lem:Nt:tail}.  We now use the bounds \eqref{eq:Nt:1} and \eqref{eq:Nt:2} in \eqref{eq:Nt}. The statement \eqref{eq:ho:lbd} follows immediately from this by choosing $\alpha_0\le \alpha_1\wedge \alpha_2\wedge \alpha_3$ small enough. 
\end{proof} 
 \begin{lem} \label{lem:N:Nt} There exists a constant $\mathbf K_6$ dependent only on $\theta$, and a positive $\tilde \alpha_0=\frac{C(\theta)}{\max(\cu^4,\cu^{2/\theta})} $ such that for $\alpha<\tilde \alpha_0$ small enough one has 
 \be\label{eq:N:Nt}
 P\left(\sup_{t\le \alpha\epsilon^{\frac2\theta}, x\in \T} |N(t,x)|\le \epsilon^{\frac{1}{2\theta}},\;\sup_{t\le \alpha \epsilon^{\frac2\theta},\, x\ne y \in \T} |\widetilde N(t,x,y)|\le \epsilon\right) \ge \exp\left(- \frac{1}{ \alpha^{\frac12}\epsilon^{\frac1\theta}}\exp\left(-\frac{\mathbf K_6}{\cu^2\alpha^\theta}\right)\right)
 \ee
 \end{lem}
 \begin{proof} An application of the Gaussian correlation inequality (\cite{roye}, \cite{lata-matl}) gives 
 \bes
 \begin{split}
&P\left(\sup_{t\le \alpha\epsilon^{\frac2\theta}, x\in \T} |N(t,x)|\le \epsilon^{\frac{1}{2\theta}},\;\sup_{t\le \alpha \epsilon^{\frac2\theta},\, x\ne y \in \T} |\widetilde N(t,x,y)|\le \epsilon\right) \\
&\ge P\left(\sup_{t\le \alpha\epsilon^{\frac2\theta}, x\in \T} |N(t,x)|\le \epsilon^{\frac{1}{2\theta}}\right)\cdot P\left(\sup_{t\le \alpha \epsilon^{\frac2\theta},\, x\ne y \in \T} |\widetilde N(t,x,y)|\le \epsilon\right).
\end{split}
 \ees
 
 We now partition $\T$ into disjoint intervals $[a_i, a_{i+1})$ where $a_i:=i\alpha^{\frac12}\epsilon^{\frac1\theta}$ for $i=1, \dots, {\alpha^{-\frac12}\epsilon^{-\frac1\theta}}$.
 Applying the Gaussian correlation inequality once again and \eqref{eq:N:tail}, one obtains 
 \bes \begin{split}P\left(\sup_{t\le \alpha\epsilon^{\frac2\theta}, x\in \T} |N(t,x)|\le \epsilon^{\frac{1}{2\theta}}\right)&\ge \prod_{i=1}^{\alpha^{-\frac12}\epsilon^{-\frac1\theta}} P\left(\sup_{t\le \alpha\epsilon^{\frac2\theta}, x\in [0,\sqrt \alpha \epsilon^{\frac1\theta}]} |N(t,x)|\le \epsilon^{\frac{1}{2\theta}}\right) \\
 &\ge \left\{1-\exp\left(-\frac{\mathbf K_2}{2\cu^2 \alpha^{\frac12}}\right)\right\}^{\alpha^{-\frac12}\epsilon^{-\frac1\theta}} \\
 &\ge \exp \left(-2\alpha^{-\frac12}\epsilon^{-\frac1\theta} \exp\left(-\frac{\mathbf K_2}{ 2\cu^2 \alpha^{\frac12}}\right)\right)
 \end{split}
 \ees
 if $\alpha< \alpha_4 = \frac{C}{\cu^4}$ is small enough. The result now follows from this and \eqref{eq:ho:lbd}.
 \end{proof}
 For the next lemma recall the events $B_i$ defined in \eqref{eq:B}.
\begin{lem}\label{lem:B0}For all initial profiles $u_0$ with $|u_0(x)|\le \frac{\epsilon^{\frac{1}{2\theta}}}{3}$ and $\mathcal H^{(\theta)}(u_0)\le \frac{\epsilon}{3}$, one has
\[ P(B_0) \ge \exp\left(-\frac{2}{ \sqrt c_0 \epsilon^{\frac1\theta}} \exp\left(-\frac{\mathbf K_6}{36 \cu^2c_0^\theta}\right)-\frac{2}{9c_0\cl^2\epsilon^{\frac1\theta}} \right)\]
when  $ c_0 6^{\frac2\theta}<\tilde \alpha_0$, where $\tilde \alpha_0$ is defined in Lemma \ref{lem:N:Nt}.
\end{lem}
\begin{rem} \label{rem:B0}Note (see Remark \ref{rem:markov}) that the arbitrariness of $u_0$ assumed above is so that we have the same lower bound for $P(B_i\vert B_0, B_1,\cdots, B_{i-1})$. This is because given $B_{i-1}$ the profile $u(t_{i-1},\cdot)$ has sup norm at most $\frac{\epsilon^{\frac{1}{2\theta}}} {3}$ and H\"older norm at most $\frac\epsilon 3$. One can then use the Markov property and the above result. 
\end{rem}
\begin{proof}
We will use a change of measure argument inspired by a technique in large deviation theory. A similar method was employed in \cite{athr-jose-muel}. Consider the measure $Q$ defined by 
\[ \frac{dQ}{dP} = \exp \left(Z_{t_1}^{(1)}- \frac12 Z_{t_1}^{(2)}\right),\]
where 
\bes 
\begin{split}
Z_{t_1}^{(1)} &= -\int_0^{t_1} \int_{\T} \frac{1}{\sigma(r,z)} \frac{(G_r*u_0)(z)}{t_1} W(dzdr), \\
Z_{t_1}^{(2)} &= \int_0^{t_1} \int_{\T} \left |\frac{1}{\sigma(r,z)} \frac{(G_r*u_0)(z)}{t_1}\right|^2 \, dz dr.\end{split}
\ees
Define 
\[ \dot{\widetilde W}(r,z) := \dot W(r,z) + \frac{1}{\sigma(r,z)}\cdot \frac{(G_r* u_0)(z)}{t_1}.\]
It follows from Lemma 2.1 in \cite{athr-jose-muel} that $\dot{\widetilde W}$ is a white noise under the measure $Q$. 

 By change of measure
 \[ Q(B_0)= E_P\left(\frac{dQ}{dP}\cdot \mathbf 1 \{B_0\}\right),\]
and so Cauchy-Schwarz inequality gives 
\[ Q(B_0) \le \left[E_P\left(\frac{dQ}{dP}\right)^2\right]^{\frac12} \cdot P(B_0)^{\frac12},\]
from which we obtain 
\be \label{eq:pq} P(B_0) \ge Q(B_0)^2\left\{ E_P\left(\frac{dQ}{dP}\right)^2\right\}^{-1}.\ee
Now 
\be \label{eq:rn}E_P\left(\frac{dQ}{dP}\right)^2= \exp \left(\int_0^{t_1}\int_{\T}\left|\frac{1}{\sigma(s,y)}\cdot \frac{(G_s*u_0)(y)}{t_1}\right|^2dy ds\right) \le \exp\left(\frac{1}{9c_0\cl^2\epsilon^{\frac1\theta}} \right).\ee
We next provide a lower bound on $Q(B_0)$. First observe that 
\begin{equation}\label{eq:holder:1}
u(t,x)= \left(1- \frac{t}{t_1}\right) (G_t*u_0)(x) + \int_0^t \int_{\T} G(t-r, x-z) \sigma(r,z) \dot{\widetilde W}(dr dz),
\end{equation}
and
\begin{equation} \label{eq:holder:2}
\begin{split}
 \frac{u(t,x)-u(t,y) }{|x-y|^{\frac12-\theta}} &= \left(1- \frac{t}{t_1}\right)\cdot \left[\frac{(G_t* u_0)(x) -(G_t*u_0)(y)}{|x-y|^{\frac12-\theta}} \right] \\
 &\qquad + \int_0^t \int_{\T} \frac{G(t-r, x-z)- G(t-r, y-z) }{|x-y|^{\frac12-\theta}} \sigma(r,z) \dot{\widetilde W}(dr dz).
\end{split}
\end{equation} 
The deterministic term in \eqref{eq:holder:1} is bounded uniformly (in $x$) by $\frac {\epsilon^{\frac{1}{2\theta}}}{3}$ in the interval $[0,t_1]$ and is equal to $0$ at the terminal time $t_1$. Similarly, due to Lemma \ref{lem:hold:det}, the first term in \eqref{eq:holder:2} is bounded uniformly (in $x, y$) by $\frac \epsilon 3$ in the same interval and is also equal to $0$ at the terminal time $t_1$. We define $N_1(t, x)$ and $\widetilde N_1(t, x, y)$ as $N(t, x)$ and $\widetilde N(t, x, y)$ as in \eqref{eq:N} and \eqref{eq:N:til} respectively but by replacing $\dot W$ by $\dot{\widetilde W}$.
 It therefore follows 
\begin{equation}\label{eq:qe}\begin{split}
 Q(B_0) &\ge Q\left(\sup_{t\le t_1, \, x\in \T} |N_1(t,x)|\le \frac{\epsilon^{\frac{1}{2\theta}}}{6},\; \sup_{t\le t_1,\; x\ne y \in \T}\big\vert \widetilde N_1(t,x,y)\big\vert \le \frac{ \epsilon}{6}\right) \\
 &\ge Q\left(\sup_{t\le c_06^{\frac2\theta}(\epsilon/6)^{\frac2\theta}, \, x\in \T} |N_1(t,x)|\le\left(\frac{\epsilon}{6}\right)^{\frac{1}{2\theta}},\; \sup_{t\le c_06^{\frac2\theta}(\epsilon/6)^{\frac2\theta},\; x\ne y \in \T}\big\vert \widetilde N_1(t,x,y)\big\vert \le \frac{ \epsilon}{6}\right) \\
 &\ge \exp\left(-\frac{1}{\sqrt c_0 \epsilon^{\frac1\theta}} \exp\left(-\frac{\mathbf K_6}{36 \cu^2c_0^\theta}\right)\right),
\end{split}
\end{equation}  
as long as $ c_0 6^{\frac2\theta}<\tilde \alpha_0$ from Lemma \ref{lem:N:Nt}.  If we use \eqref{eq:qe} and \eqref{eq:rn} into \eqref{eq:pq} we obtain
\[ P(B_0) \ge \exp\left(-\frac{2}{\sqrt c_0 \epsilon^{\frac1\theta}} \exp\left(-\frac{\mathbf K_6}{36 \cu^2c_0^\theta}\right)-\frac{2}{9c_0\cl^2\epsilon^{\frac1\theta}} \right) \]
as long as $ c_0 6^{\frac2\theta}<\tilde \alpha_0$.
\end{proof}

\subsection{Lower bound in Theorem \ref{thm1} (\lowercase{b})}  The argument in Section \ref{sec:lb:Gaussian} has to be modified at quite a few places. We first note the following lemma which follows immediately from the Gaussian correlation inequality.
 \begin{lem} \label{lem:Nh} There is $\alpha_0^{\ha} = \frac{C(\theta)}{\cu^{2/\theta}}>0$ such that for $\alpha<\alpha_0^{\ha}$ we have
 \[ P\left(\sup_{\stackrel{0\le s, t\le \alpha\epsilon^{\frac2\theta},\, s\ne t}{x\in \T}} \left|N^{\ha}(s,t,x) \right|\le \epsilon\right) \ge  \exp\left(-\frac{1}{\alpha^{\frac12}\epsilon^{\frac1\theta}}\exp\left(-\frac{\mathbf K_9}{\cu^2\alpha^{\theta}}\right)\right)\]
 \end{lem}

We shall consider time discretizations $t_i= ic_2\delta^2=i c_2\epsilon^{\frac2\theta},\; i=0,1,\cdots, I$. The constant $c_2$ will be  appropriately chosen so as to get a uniform lower bound on $P\left(B_i^{\ha} \big \vert B_0^{\ha}, B_1^{\ha}\cdots B_{i-1}^{\ha} \right)$ in \eqref{eq:ht:markov} below. It will only depend on $\theta$ and $\cu$. In this section let
\[ B_i^{\ha} := U_i^{\ha} \cap H_i^{\ha}\cap T_i^{\ha},\]
where, similar to Section \ref{sec:lb:Gaussian}, 
\begin{align}
 \label{eq:U:ha} U_i^{\ha} &=  \left\{\sup_{x\in \T} \left|u(t_{i+1}, x)\right| \le \frac{ \epsilon^{\frac{1}{2\theta}}}{8c_2^{\frac{\theta}{2} -\frac14}} , \text{ and }\sup_{x\in \T} \left|u(t,x)\right|\le \frac{ \epsilon^{\frac{1}{2\theta}}}{4c_2^{\frac{\theta}{2} -\frac14}} \text{ for all } t\in [t_i, t_{i+1}]\right\},\\
\label{eq:H:ha}  H_i^{\ha}&=\left\{\mathcal{H}_{t_{i+1}}^{(\theta)}(u)\le \frac{\epsilon}{8\Lambda}, \text{ and } \mathcal H_t^{(\theta)}(u)\le \frac{\epsilon}{2\Lambda} \text{ for all } t\in [t_i, t_{i+1}]\right\},\\
 \label{eq:T:ha} T_i^{\ha} &=\left\{\sup_{\stackrel{x\in \T}{t_i\le s, t\le t_{i+1},\, s\ne t }} \frac{\left |u(t,x)-u(s,x)\right|}{|t-s|^{\frac14-\frac\theta 2}} \le \frac{\epsilon}{2}\right\}.
\end{align}
Here, we recall the constant $\Lambda$ in \eqref{eq:lambda}. %
Let us  first consider  the following lemma.
\begin{lem} \label{lem:hold:time} We have the following inclusion. 
\be
\label{bi:sub}
 \cap_{i=0}^{I} B_i^{\ha} \subset \left\{\sup_{x\in\T} \HS_x^{(\theta)}(u)\le \epsilon \right\}.
 \ee
\end{lem}
\begin{proof}We take a realization $u(\cdot,\cdot)$ of the left hand side.  We need to show for any $s<t\in [0,T]$,
\[ \frac{|u(t,x) -u(s,x)|}{|t-s|^{\frac14-\frac\theta 2}} \le \epsilon.\]
Suppose $s <t$ are both in $[t_i, t_{i+1}]$. Then, since the profile is in $T_i^{\ha}$,
\[ \frac{|u(t,x) -u(s,x)|}{|t-s|^{\frac14-\frac\theta 2}} \le \frac \epsilon 2.\]
Next we consider the case when $s\in [0, t_{i-1}]$ and $t\in [t_i, t_{i+1}]$. In this case we have 
\begin{align*}
  \frac{ | u(t,x) -u(s,x) |}{|t-s|^{\frac14-\frac\theta 2}} & \le\frac{ | u(t,x) -u(t_i,x) |}{|t-t_i|^{\frac14-\frac\theta 2}} +\frac{ | u(t_i,x) -u(s,x) |}{|t_i-s|^{\frac14-\frac\theta 2}}  \\
  &\le  \frac{\epsilon}{2} + \frac{2\epsilon^{\frac{1}{2\theta}}} {4c_2^{\frac{\theta}{2} -\frac14}}\cdot \left(\frac{1}{c_2 \epsilon^{\frac2\theta}}\right)^{\frac14-\frac\theta 2} \\
  &\le \epsilon,
\end{align*}
since $u \in \cap_{i=0}^{I-1} U_i^{\ha}$. Finally consider $s\in [t_{i-1}, t_i]$ and $t\in [t_i, t_{i+1}]$. In this case 
\[\frac{ | u(t,x) -u(s,x) |}{|t-s|^{\frac14-\frac\theta 2}}  \le \frac{ | u(t,x) -u(t_i,x) |}{|t-t_i|^{\frac14-\frac\theta 2}} +\frac{ | u(t_i,x) -u(s,x) |}{|t_i-s|^{\frac14-\frac\theta 2}}  \le \epsilon.\]
This shows that the realization is in $\left\{\sup_{x\in\T} \HS_x^{(\theta)}(u)\le \epsilon \right\}$.
\end{proof}
\begin{rem} Observe that the events $H_i^{\ha}$ play no role in the argument above, and we can in fact take the larger set $\cap_{i=0}^{I} (U_i^{\ha}\cap T_i^{\ha})$ in the left hand side of \eqref{bi:sub}. However, as we will see in Proposition \ref{prop:b0:ha} below, to get a lower bound on  $P(T_i^{\ha})$ we will need a control on the spatial H\"older norms of $u$ as given by the events $H_i^{\ha}$.  
\end{rem}
From the above lemma
\be\begin{split}\label{eq:ht:markov}
P \left(\sup_{x\in\T} \HS_x^{(\theta)}(u)\le \epsilon \right)
&\ge  P \left( \cap_{i=0}^{I} B_i^{\ha}\right) \\
& = \prod_{i=0}^{I} P\left(B_i^{\ha} \big \vert B_0^{\ha}, B_1^{\ha}\cdots B_{i-1}^{\ha} \right).
\end{split}
\ee
The lower bound follows from the Markov property and the following

\begin{prop} \label{prop:b0:ha} Suppose the initial profile $u_0$ satisfies
\[ \sup_{x\in \T} \left|u_0(x)\right| \le \frac{\epsilon^{\frac{1}{2\theta}}c_2^{\frac14-\frac\theta 2}}{8},\quad \HO^{(\theta)}(u_0) \le \frac{\epsilon}{8\Lambda}. \]
Then there exists a constant $\mathbf K_{10}>0$ dependent only on $\theta$ and a positive $\tilde\alpha^{\ha}_0= \frac{C(\theta)}{\max(\cu^4,\cu^{2/\theta})}$ such that for all $c_2<\tilde\alpha^{\ha}_0$ one has 
\[ P\left(B_0^{\ha}\right) \ge \exp \left(-\frac{1}{\sqrt c_2\epsilon^{\frac 1\theta}}\exp\left(-\frac{\mathbf K_{10}}{(1+\Lambda^2)\cu^2c_2^{\theta}}\right)-\frac{1}{64\cl^2c_2^{\frac12+\theta}\epsilon^{\frac1\theta}} \right) .\]
\end{prop}

\begin{proof} We work with the measure $Q$ constructed in the proof of Lemma \ref{lem:B0}. From \eqref{eq:holder:1} we have 
\begin{align*}
 &u(t,x) - u(s,x) \\
 &= \left\{\left[1-\frac{t}{t_1}\right] (G_t*u_0) (x) +N_1(t,x)\right\} -\left\{\left[1-\frac{s}{t_1}\right] (G_s*u_0) (x)+N_1(s,x)\right\},
 \end{align*}
 where we recall 
 \[ N_1(t,x) = \int_0^t\int_{\T} G(t-r, x-z) \sigma(r,z) \dot{\widetilde{W}}(dr dz). \]
 Define 
 \[N_1^{\ha}(s,t,x) := \frac{N_1(t,x)-N_1(s,x)}{|t-s|^{\frac14-\frac\theta 2}} .\]
For $s,t\in [0,t_1]$
\begin{align*}
 \frac{|u(t,x) -u(s,x) |}{|t-s|^{\frac14-\frac\theta 2}} & \le \left[1+\frac{s}{t_1}\right] \frac{|(G_t*u_0) (x)-(G_s*u_0) (x)|}{|t-s|^{\frac14-\frac\theta 2}} \\
 &\qquad +\frac{|t-s|}{t_1} \frac{ |(G_t*u_0)(x)|}{|t-s|^{\frac14-\frac\theta 2}} +|N_1^{\ha}(s,t,x)|.
\end{align*}
The first term on the right is less than $\frac\epsilon 4$ thanks to Lemma \ref{lem:hold:det:t}. The second term is less than $\frac\epsilon 8$ by the assumption on the initial profile. Now
\begin{align*}
 Q(B_0^{\ha}) &\ge  Q\bigg(\sup_{t\le t_1, \, x\in \T} |N_1(t,x)|\le \frac{\epsilon^{\frac{1}{2\theta}}}{8c_2^{\frac\theta 2-\frac14}},\; \sup_{t\le t_1,\; x\ne y \in \T}\big\vert \widetilde N_1(t,x,y)\big\vert \le \frac{ \epsilon}{8\Lambda} ,\\
 &\hspace{3cm}  \sup_{x\in \T,\, s\ne t \in [0,t_1]} |N_1^{\ha}(s,t,x)|\le \frac{\epsilon}{8} \bigg)  \\
 & \ge Q\bigg(\sup_{t\le t_1, \, x\in \T} |N_1(t,x)|\le \frac{\epsilon^{\frac{1}{2\theta}}}{8c_2^{\frac\theta 2-\frac14}},\; \sup_{t\le t_1,\; x\ne y \in \T}\big\vert \widetilde N_1(t,x,y)\big\vert \le \frac{ \epsilon}{8\Lambda} \bigg) \\
& \hspace{3cm}\times Q\bigg( \sup_{x\in \T,\, s\ne t \in [0,t_1]}|N_1^{\ha}(s,t,x)| \le \frac{\epsilon}{8} \bigg) 
\end{align*}
by the Gaussian correlation inequality. By splitting the interval $\T$ into smaller intervals of length $c_2^{\frac12} \epsilon^{\frac1\theta}$ and using Gaussian correlation inequality repeatedly (see also Lemma \ref{lem:Nh} and Lemma \ref{lem:N:Nt}) one obtains
\be \label{eq:q:bha:lbd}  Q(B_0^{\ha}) \ge  \exp\left(-\frac{1}{c_2^{\frac12}\epsilon^{\frac1\theta}} \exp\left(-\frac{\mathbf K_{11}}{(1+\Lambda^2) \cu^2c_2^\theta}\right)\right)\ee
as long as $c_2$ is small enough. 
Following the arguments of Lemma \ref{lem:B0} we obtain
\[ P\left(B_0^{\ha}\right) \ge  Q\left(B_0^{\ha}\right)^2 \left\{ E_P\left(\frac{dQ}{dP}\right)^2\right\}^{-1} \]
where $Q$ is the measure constructed there. As in \eqref{eq:rn} we have
\[ E_P\left(\frac{dQ}{dP}\right)^2= \exp \left(\int_0^{t_1}\int_{\T}\left|\frac{1}{\sigma(s,y)}\cdot \frac{(G_s*u_0)(y)}{t_1}\right|^2dy ds\right) \le \exp\left(\frac{1}{64\cl^2c_2^{\frac12+\theta}\epsilon^{\frac1\theta}} \right).\]
Using the above along with \eqref{eq:q:bha:lbd}, the proof is complete.
\end{proof}

\subsection{Lower bound in Theorem \ref{thm2} (\lowercase{a})} \label{sec:lb:nonlinear}
We begin by describing the idea  behind the proof first.  The same idea will be used for the proof of Theorem \ref{thm3}. We will consider the following modifications of the temporal discretisation given by \eqref{eq:t:x},
 \begin{align} \label{ti}
 t_i= ic_0 \delta^{2+\eta},\qquad  i=0,\, 1,\, \cdots, I:= \left[\frac{T}{c_0\delta^{2+\eta}}\right].
\end{align}
Define 
\[ R_i :=  \left\{|u(t_{i+1}, x)| \le \frac{ \epsilon^{\frac{1}{2\theta}}}{3} \text{ for all } x\in \T, \text{ and } |u(t,x)|\le \epsilon^{\frac{1}{2\theta}} \text{ for all } t\in [t_i, t_{i+1}],\, x\in \T\right\},\]
and 
\[  S_i:=\left\{\mathcal{H}_{t_{i+1}}^{(\theta)}(u)\le \frac\epsilon 3, \text{ and } \mathcal H^{(\theta)}_t(u)\le \epsilon \text{ for all } t\in [t_i, t_{i+1}]\right\}.\] 
We consider the event
\[ A_i=  R_i\cap  S_i.\]
Our goal is to provide a lower bound on $P(A_i)$. By the Markov property it is sufficient to obtain a lower bound on $P(A_0)$ under the assumption that the initial profile satisfies $|u_0(x)|\le \epsilon^{\frac{1}{2\theta}}/3$ and $\HO^{(\theta)}(u_0) \le \epsilon/3$. 

Consider the evolution of $u(t,\cdot)$ in $[0, t_1]$ and write
\[  u(t,x) = u_g(t,x)+ D(t,x), \]
where $u_g(t,x)$ solves
\[ \partial_t u_g(t,x) = \frac12 \partial_x^2 u_g(t,x) + \sigma\big(t,x,u_0(x)\big) \cdot \dot{W}(t,x),\quad t\in \R_+,\; x\in \T,\]
with initial profile $u_0(x)$. Note that the third coordinate in $\sigma$ is now $u_0(x)$ and therefore $u_g$ is a Gaussian random field. Therefore if we define as in \eqref{eq:B}
\[  B_0^{(g)}=  U_0^{(g)}\cap  H_0^{(g)}, \]
with $U_0^{(g)}$ and $H_0^{(g)}$ defined similarly as in \eqref{eq:U} and \eqref{eq:H} but for the process $u_g$ in place of $u$, and with the new value of $t_1=\delta^{2+\eta}=\epsilon^{(2+\eta)/\theta}$. 

Now $B_0^{(g)}\supset \widetilde{B}_0^{(g)}$, where $\widetilde{B}_0^{(g)}$ is similar to \eqref{eq:B} but with $u$ replaced by $u^{(g)}$, $\epsilon$ replaced by $\widetilde\epsilon=\frac\epsilon 8$, and $t_1=c_0(\widetilde\epsilon)^{\frac2\theta}$. Therefore
\be \label{eq:bg} P(B_0^{(g)}) \ge  \exp\left\{-\frac{2\cdot 8^{\frac1\theta}}{\sqrt c_0 \epsilon^{\frac1\theta}} \exp\left(-\frac{\mathbf K_6}{36 \cu^2c_0^\theta}\right)-\frac{2\cdot 8^{\frac1\theta}}{9c_0\cl^2\epsilon^{\frac1\theta}} \right\}\ee
when  $ c_0 6^{\frac2\theta}<\tilde \alpha_0$. The difference between $u$ and $u_g$ is 
\[ D(t,x) =\int_0^{t} \int_{\T} G(t-s, x-y)\cdot  \left[\sigma\big(s,y, u(s,y)\big)-\sigma\big(s,y, u_0(y)\big)\right] W(ds dy). \]
Consider the set 
\be\label{eq:V} V:=\left\{|D(t,x)|\le \frac{\epsilon^{\frac{1}{2\theta}}}{6} \text{ for all } t\in [0,t_1],\, x\in \T\right\} \cap \left\{\mathcal H_t^{(\theta)}(D)\le \frac{\epsilon}{6} \text{ for all } t\in [0,t_1]\right\}.\ee
Define now 
\[\tau :=\inf\left\{t\ge 0: \left|u(t,x) -u_0(x)\right|\ge 2\epsilon^{\frac{1}{2\theta}} \text{ for some } x\in \T\right\},\]
and let 
\be\label{eq:D:til}\widetilde D(t,x) := \int_0^{t} \int_{\T} G(t-s, x-y)\cdot\left[\sigma\big(s,y, u(s\wedge \tau,y)\big)-\sigma\big(s,y, u_0(y)\big)\right] W(ds dy).\ee
Let the event $\widetilde V$ be the same as $V$ (see \eqref{eq:V}) but with $D$ replaced by $\tilde D$. Now 
\be\label{eq:A0}
\begin{split}
P(A_0) &\ge P\left( B_0^{(g)} \cap V\right) \\
&= P\left( B_0^{(g)} \cap \widetilde V\right) \\
&\ge P\left(B_0^{(g)}\right) -P\left(\widetilde V^c\right) \\
&\ge P\left(B_0^{(g)}\right) - P\left(\sup_{\stackrel{0\le t\le t_1}{x\in \T}}|\widetilde D(t,x)|>\frac{\epsilon^{\frac{1}{2\theta}}}{6}\right) -P\left(\sup_{0\le t\le t_1} \mathcal H_t^{(\theta)}(\widetilde D) >\frac{\epsilon}{6} \right).
\end{split}
\ee
The equality holds because on the event $A_0$ we have $\|u(t,\cdot) -u_0\|_\infty<2\epsilon^{\frac{1}{2\theta}}$ (recall that our initial profile is everywhere less than $\epsilon^{\frac{1}{2\theta}}/3$), and so $D(t,\cdot)=\widetilde D(t,\cdot)$ up to time $t_1$ on the event $A_0$. Now we use Remark \ref{rem:N:tail} together with the fact that now $t_1=c_0\epsilon^{\frac{2}{\theta}+\frac\eta\theta}$ to obtain 
\be\label{eq:Dt}
\begin{split}
P\left(\sup_{\stackrel{0\le t\le t_1}{x\in \T}}|\widetilde D(t,x)|>\frac{\epsilon^{\frac{1}{2\theta}}}{6}\right) &\le \sum_{i=1}^{c_0^{-\frac12}\epsilon^{-\frac1\theta}}  P\left(\sup_{\stackrel{0\le t\le t_1}{x\in \left[(i-1)\sqrt c_0\epsilon^{\frac1\theta}\,, i\sqrt c_0\epsilon^{\frac1\theta}\right]}}|\widetilde D(t,x)|>\frac{1}{6c_0^{\frac14}}\left(c_0^{\frac14}\epsilon^{\frac{1}{2\theta}}\right) \right) \\
& \le \frac{\mathbf K_1}{\sqrt c_0\epsilon^{\frac1\theta +\frac{\eta}{2\theta}}} \exp \left(-\frac{\mathbf K_2 }{144\sqrt c_0\mathscr D^2\epsilon^{\frac1\theta+\frac{\eta}{2\theta}}}\right).
\end{split}
\ee
Next we focus on the last term in \eqref{eq:A0}. We divide $\T^2$ into squares $S$ of side length $\sqrt c_0 \epsilon^{\frac1\theta}$. 
Let 
\[\widetilde N^{(\widetilde D)}(t,x,y) := \frac{\widetilde D(t,x)-\widetilde D(t,y)}{|x-y|^{\frac12-\theta}}.\]
Using Lemma \ref{lem:Nt:tail} (more specifically, Remark \ref{rem:Nt:tail}) we obtain
\be \label{eq:HDt}
\begin{split}
P\left(\sup_{0\le t\le t_1} \mathcal H_t^{(\theta)}(\widetilde D) >\frac{\epsilon}{6} \right) & \le \frac{1}{c_0\epsilon^{\frac2\theta}} \cdot \sup_S P\left(\sup_{\stackrel{0\le t\le c_0\epsilon^{\frac2\theta+\frac{\eta}{\theta}}}{(x, y) \in S,\,x\ne y}} |\widetilde N^{(\widetilde D)}(t,x,y)|>\frac{\epsilon}{6}\right) \\
&= \frac{1}{c_0\epsilon^{\frac2\theta}} \cdot \sup_S P\left(\sup_{\stackrel{0\le t\le c_0\epsilon^{\frac2\theta+\frac{\eta}{\theta}}}{(x, y) \in S,\,x\ne y}} |\widetilde N^{(\widetilde D)}(t,x,y)|>\frac{1}{6c_0^{\frac\theta 2}}\cdot c_0^{\frac\theta2}\epsilon\right)\\
&\le \frac{\mathbf K_3}{c_0\epsilon^{\frac2\theta+\frac{\eta}{2\theta}}} \exp \left(-\frac{\mathbf K_4}{144 c_0^\theta\mathscr D^2\epsilon^{\frac1\theta+\eta}}\right).
\end{split}
\ee
We plug in the bounds \eqref{eq:HDt}, \eqref{eq:Dt} and \eqref{eq:bg} into \eqref{eq:A0} to obtain 
\bes
\begin{split}
P(A_0) &\ge  \exp\left\{-\frac{2\cdot 8^{\frac1\theta}}{\sqrt c_0 \epsilon^{\frac1\theta}} \exp\left(-\frac{\mathbf K_6}{36 \cu^2c_0^\theta}\right)-\frac{2\cdot 8^{\frac1\theta}}{9c_0\cl^2\epsilon^{\frac1\theta}} \right\} \\
&\qquad \quad- \frac{\mathbf K_1}{\sqrt c_0\epsilon^{\frac1\theta +\frac{\eta}{2\theta}}} \exp \left(-\frac{\mathbf K_2 }{144\sqrt c_0\mathscr D^2\epsilon^{\frac1\theta+\frac{\eta}{2\theta}}}\right)-\frac{\mathbf K_3}{c_0\epsilon^{\frac2\theta+\frac{\eta}{2\theta}}} \exp \left(-\frac{\mathbf K_4}{144 c_0^\theta\mathscr D^2\epsilon^{\frac1\theta+\eta}}\right). 
\end{split}
\ees
The last two terms are much smaller than the first term for small $\epsilon$. Therefore $P(A_0) \ge P(B_0^{(g)})/2$ when $\epsilon$ is small enough. We thus have a lower bound on $P(A_i)$ for all $i$. As mentioned earlier, the proof of \eqref{eq:unif:t:3} then follows from the Markov property.
\qed

\subsection{Lower bound in Theorem \ref{thm2} (\lowercase{b})} The argument follows that of Section \ref{sec:lb:nonlinear} with some modifications.  Now let $t_i=ic_2\delta^{2+\eta\theta}=\epsilon^{\frac{2}{\theta}+\eta}$. Similar to before, for $t\in [t_i, t_{i+1}]$, we write $u(t,x)= u_g^{(i)}(t,x) + D^{(i)} (t,x)$. Here 
\begin{align*}
\partial_t u_g^{(i)}& = \frac12 \partial_x^2 u_g^{(i)} + \sigma\left(t,x, u(t_i,x)\right) \dot{W}(t,x),\qquad t\in [t_i, t_{i+1}], \\
 u_g^{(i)}(t_i, \cdot)  &\equiv u(t_i, \cdot),
\end{align*}
and 
\bes
D^{(i)}(t,x) = \int_{t_i}^t\int_{\T} G(t-s, x-y) \left[\sigma(s, y, u(s,y)) -\sigma(s,y, u(t_i,y)\right] W(dyds).
\ees
Now define 
\[ B_i^{(g),\ha} = U_i^{(g), \ha} \cap H_i^{(g),\ha}\cap T_i^{(g),\ha},\]
where $ H_i^{(g),\ha},\, T_i^{(g),\ha}$ are as in \eqref{eq:H:ha} and \eqref{eq:T:ha} but with $u_g$ in place of $u$, and with the $t_i = ic_2\delta^{2+\eta\theta}=ic_2\epsilon^{\frac2\theta  +\eta}$. On the other hand we define
\begin{align*} U_i^{(g),\ha} &:= \bigg\{\sup_{x\in \T} \left|u_g^{(i)}(t_{i+1}, x)\right| \le \frac{ \epsilon^{\frac{1}{2\theta}+\eta(\frac14-\frac\theta 2)}}{8c_2^{\frac{\theta}{2} -\frac14}} , \\
&\qquad\qquad  \text{ and }\sup_{x\in \T} \left|u_g^{(i)}(t,x)\right|\le \frac{ \epsilon^{\frac{1}{2\theta}+\eta(\frac14-\frac\theta 2)}}{4c_2^{\frac{\theta}{2} -\frac14}} \text{ for all } t\in [t_i, t_{i+1}]\bigg\}.
\end{align*}
Now let 
\[V_i^{\ha} = V_{i,1}^{\ha} \cap V_{i,2}^{\ha}, \]
where
\begin{align*}
V_{i,1}^{\ha} &= \left\lbrace \sup_{\stackrel{x\in \T}{t\in [t_i, t_{i+1}]}}|D^{(i)}(t,x) |\le \frac{\epsilon^{\frac{1}{2\theta}+\eta(\frac14-\frac\theta 2)}}{4 c_2^{\frac\theta 2-\frac14}}\right\rbrace \\
V_{i,2}^{\ha} &=\left\{\sup_{\stackrel{x\in \T}{t_i\le s, t\le t_{i+1},\, s\ne t }} \frac{\left |D^{(i)}(t,x)-D^{(i)}(s,x)\right|}{|t-s|^{\frac14-\frac\theta 2}} \le \frac{\epsilon}{2}\right\}.
\end{align*}
It follows from arguments similar to Lemma \ref{lem:hold:time} that 
\be \label{eq:contain}\bigcap_{i=0}^I \mathscr{B}_i^{\ha} \subset \left\lbrace \sup_{x\in \T} \HS_x^{(\theta)}(u) \le 2\epsilon\right\rbrace,
\ee
where 
\[ \mathscr{B}_i^{\ha}:=B_i^{(g),\ha} \cap V_i^{\ha}.\]
By the Markov property it is enough to give a lower bound on $P(\mathscr B_0^{\ha})$ under the assumption that $ |u_0(x)|\le \frac{\epsilon^{\frac{1}{2\theta}+\eta(\frac14-\frac\theta 2)}}{8c_2^{\frac\theta 2-\frac14}}$ and $\HO_0^{(\theta)}(u) \le \frac{\epsilon}{8\Lambda}$. Let 
\[\tau :=\inf\left\{t\ge 0: \left|u(t,x) -u_0(x)\right|\ge \frac{2}{c_2^{\frac\theta2-\frac14}}\epsilon^{\frac{1}{2\theta}+\eta(\frac14-\frac\theta 2)} \text{ for some } x\in \T\right\},\]
Let $\widetilde V_0^{\ha}$ be defined as $V_0^{\ha}$ but with $\widetilde D^{(0)}$ in place of $D^{(0)}$.  Here $\widetilde D_0$ is as in \eqref{eq:D:til} but with the above $\tau$. 
\begin{align*}
 P(\mathscr B_0^{\ha}) &= P (B_0^{(g),\ha} \cap V_0^{\ha}) \\
 &= P (B_0^{(g),\ha} \cap \widetilde V_0^{\ha})  \\
 &\ge P\left(B_0^{(g),\ha}\right) - P\left((\widetilde V_{0,1}^{\ha})^c\right) - P\left((\widetilde V_{0,2}^{\ha})^c\right)
 \end{align*}
Using Remark \ref{rem:N:tail} and the argument in \eqref{eq:Dt} we obtain 
\begin{align}\label{eq:v01}
 P\left((\widetilde V_{0,1}^{\ha})^c\right)  &= P\left(\sup_{\stackrel{0\le t\le t_1}{x\in \T}}|\widetilde D^{(0)}(t,x)|>\frac{\epsilon^{\frac{1}{2\theta}+\eta(\frac14-\frac\theta 2)}}{4 c_2^{\frac\theta 2-\frac14}}\right) \\
 \nonumber& \le \frac{\mathbf K_1}{\sqrt c_2\epsilon^{\frac1\theta +\frac{\eta}{2}}} \exp \left(-\frac{\mathbf K_2 }{64\sqrt c_2\mathscr D^2\epsilon^{\frac1\theta+\frac{\eta}{2}}}\right).
\end{align}
Similarly using Remark \ref{rem:Nh:tail} we obtain 
\be\label{eq:v02} P\left((\widetilde V_{0,2}^{\ha})^c\right)  \le  \frac{\mathbf K_7}{\sqrt c_2\epsilon^{\frac1\theta +\frac{\eta}{2}}} \exp \left(-\frac{\mathbf K_8 }{16\sqrt c_2\mathscr D^2\epsilon^{\frac1\theta+\frac{\eta}{2}}}\right).\ee

\begin{lem} We have when $\epsilon$ is small enough
\[ P\left(B_0^{(g),\ha}\right) \ge \exp\left(-\frac{3}{c_2^{\frac12}\epsilon^{\frac1\theta +\eta(\frac12-\theta)}}\exp\left(-\frac{\mathbf K_2}{128 c_2^{\theta}\cu^2}\right)\right)\]
\end{lem}
\begin{proof} As in the proof of Proposition \ref{prop:b0:ha}, with this new choice of $t_1=c_2\epsilon^{\frac2\theta+\eta}$ and with 
\[ N_1^{(g)}(t,x) := \int_0^t\int_{\T} G(t-r, x-z) \sigma(r,z, u_0(z)) \dot{\widetilde{W}}(dr dz), \]
$\widetilde N_1^{(g)},\, N_1^{\ha, (g)}$ defined in terms of $N_1^{(g)}$, we have 
\begin{align} \label{eq:q:b0g}
 Q(B_0^{(g),\ha}) & \ge Q\bigg(\sup_{t\le t_1, \, x\in \T} |N_1^{(g)}(t,x)|\le \frac{\epsilon^{\frac{1}{2\theta}+\eta(\frac14-\frac\theta 2)}}{8c_2^{\frac\theta 2-\frac14}},\; \sup_{t\le t_1,\; x\ne y \in \T}\big\vert \widetilde N_1^{(g)}(t,x,y)\big\vert \le \frac{ \epsilon}{8\Lambda} \bigg) \\
\nonumber& \hspace{3cm}\times Q\bigg( \sup_{x\in \T,\, s\ne t \in [0,t_1]}|N_1^{\ha, (g)}(s,t,x)| \le \frac{\epsilon}{8} \bigg) 
\end{align}
A lower bound on the last probability is obtained by taking the supremum of $s\ne t$ over $[0,c_2\epsilon^{\frac2\theta}]$ instead of $[0,t_1]$. This gives  (when $c_2$ is chosen small enough)
 \be \label{eq:t:1}  Q\bigg( \sup_{x\in \T,\, s\ne t \in [0,t_1]}|N_1^{\ha, (g)}(s,t,x)| \le \frac{\epsilon}{8} \bigg)\ge  \exp\left(-\frac{1}{c_2^{\frac12}8^{\frac1\theta}\epsilon^{\frac1\theta}}\exp\left(-\frac{\mathbf K_9}{64\cu^2c_2^{\theta}}\right)\right).\ee
 Next let
 \[ \tilde t_1:= c_2\epsilon^{\frac{2}{\theta} +\eta(1-2\theta)} .\]
 Clearly $\tilde t_1> t_1$ and so a lower bound of the first term on the right of \eqref{eq:q:b0g} is (note the $t_1$ in the sup has been replaced by $\tilde t_1$)
 \begin{align}
 \label{eq:t:2}
 & Q\bigg(\sup_{t\le \tilde t_1, \, x\in \T} |N_1^{(g)}(t,x)|\le \frac{\epsilon^{\frac{1}{2\theta}+\eta(\frac14-\frac\theta 2)}}{8c_2^{\frac\theta 2-\frac14}},\; \sup_{t\le \tilde t_1,\; x\ne y \in \T}\big\vert \widetilde N_1^{(g)}(t,x,y)\big\vert \le \frac{ \epsilon}{8\Lambda} \bigg) \\
\nonumber &\ge  Q\bigg(\sup_{t\le \tilde t_1, \, x\in \T} |N_1^{(g)}(t,x)|\le \frac{\epsilon^{\frac{1}{2\theta}+\eta(\frac14-\frac\theta 2)}}{8c_2^{\frac\theta 2-\frac14}}\bigg)\cdot Q\bigg(\sup_{t\le c_2\epsilon^{\frac2\theta},\; x\ne y \in \T}\big\vert \widetilde N_1^{(g)}(t,x,y)\big\vert \le \frac{ \epsilon}{8\Lambda} \bigg) \\
\nonumber &\ge \exp\left(-\frac{1}{c_2^{\frac12}\epsilon^{\frac1\theta +\eta(\frac12-\theta)}}\exp\left(-\frac{\mathbf K_2}{128 c_2^{\theta}\cu^2}\right)\right)\cdot \exp\left(-\frac{2}{c_2^{\frac12}\epsilon^{\frac1\theta}} \exp\left(-\frac{\mathbf K_5}{64 \Lambda^2\cu^2c_2^{\theta}}\right)\right)
  \end{align}
when $c_2<\frac{C(\theta)}{\max(\cu^4, \cu^{2/\theta})}$ is chosen small enough, using the arguments in Lemma \ref{lem:N:Nt}. We have used the Gaussian correlation inequality in the second step. Note that the sup in $t$ in the second probability is over a larger time interval $[0, c_2\epsilon^{\frac2\theta}]$. 

The event $B_0^{(g),\ha}$ depends on the noise up to time $c_2\epsilon^{\frac2\theta}$, and so 
\begin{align}\label{eq:bog}
P\left(B_0^{(g),\ha}\right) &\ge Q\left(B_0^{(g),\ha}\right)^2\left\{ E_P\left(\frac{dQ}{dP}\bigg\vert_{[0, c_2\epsilon^{\frac2\theta}]}\right)^2\right\}^{-1}.
\end{align}
We have the following upper bound (similar to \eqref{eq:rn}:
\be
\label{eq:t:3} E_P\left(\frac{dQ}{dP}\bigg\vert_{[0, c_2\epsilon^{\frac2\theta}]}\right)^2\le \exp\left(-\frac{1}{64 \cl^2c_2^{\theta+\frac12}\epsilon^{\frac1\theta-\eta(\frac12-\theta)}}\right). \ee
Plugging in the bounds \eqref{eq:t:1},\eqref{eq:t:2} and \eqref{eq:t:3} into \eqref{eq:bog} we obtain the lemma. 
\end{proof}
From the above lemma as well as \eqref{eq:v01} and \eqref{eq:v02} we obtain 
\[  P(\mathscr B_0^{\ha})  \ge \exp\left(-\frac{4}{c_2^{\frac12}\epsilon^{\frac1\theta +\eta(\frac12-\theta)}}\exp\left(-\frac{\mathbf K_2}{128 c_2^{\theta}\cu^2}\right)\right) \]
when $\epsilon$ is small enough, and thus from \eqref{eq:contain} one gets 
\[ P\left(\sup_{x\in \T} \HS_x^{(\theta)}(u)\le 2\epsilon\right) \ge \exp\left(-\frac{C(\theta,\cu)T}{\epsilon^{\frac3\theta+\eta(\frac32-\theta)}}\right), \]
for some constant $C(\theta,\cu)>0$ dependent only on $\cu$ and $\theta$. This completes the proof of the lower bound since $\eta$ is arbitrary.  \qed

\section{Proof of Theorem \ref{thm3}} \label{sec:thm3}
The proofs of the upper bounds in both statements in Theorem \ref{thm3} are the same as that of Theorem \ref{thm2}. The proof of the lower bounds follows the same ideas as in the proofs for the lower bounds of Theorem \ref{thm2}. We show this only for statement (a); the proof of statement (b) is similar. The only difference as compared to the proof in Theorem 1.2 (a) is that we revert back to the discretisation given by \eqref{eq:t:x}. We therefore have 
\bes
\begin{split}
P(A_0) &\ge  \exp\left\{-\frac{2}{\sqrt c_0 \epsilon^{\frac1\theta}} \exp\left(-\frac{\mathbf K_6}{36 \cu^2c_0^\theta}\right)-\frac{2}{9c_0\cl^2\epsilon^{\frac1\theta}} \right\} \\
&\qquad \quad- \frac{\mathbf K_1}{\sqrt c_0\epsilon^{\frac1\theta}} \exp \left(-\frac{\mathbf K_2 }{144\sqrt c_0\mathscr D^2\epsilon^{\frac1\theta}}\right)-\frac{\mathbf K_3}{c_0\epsilon^{\frac2\theta}} \exp \left(-\frac{\mathbf K_4}{144 c_0^\theta\mathscr D^2\epsilon^{\frac1\theta}}\right). 
\end{split}
\ees
For any fixed $c_0$, $ \cl$ and $\cu$, we can choose $\mathscr D$ small enough so that as $\epsilon$ decreases, the final two term goes to zero much faster than the first term. Therefore for small $\epsilon$ a lower bound on $P(A_0)$ (and hence $P(A_i)$) is one half times the first term above.  An application of the Markov property then finishes the proof. \qed

\section{Proof of Theorem \ref{thm:main}} \label{sec:main}
We first prove the upper bound.  This follows immediately from 
\begin{align*}
 P \left(\sup_{\substack{0\le s, t\le T \\ 0\le x, y\le 1\\(t,x)\ne (s,y)}} \frac{|u(t,x)-u(s,y)|}{|x-y|^{\frac12-\theta}+|t-s|^{\frac14-\frac\theta 2}}\le \epsilon \right)  & \le P\left(\sup_{t\le T} \HO_t^{(\theta)}(u) \le \epsilon\right),
 \end{align*}
and Theorem \ref{thm2}. Let us turn our attention to the lower bound. In the proof of the lower bound of Theorem \ref{thm2} (b), we let
\[V_{i,3}^{\ha} :=\left\{\sup_{\stackrel{x\neq y \in \T}{t_i\le t\le t_{i+1} }} \frac{\left |D^{(i)}(t,x)-D^{(i)}(t,y)\right|}{|x-y|^{\frac12-\theta}} \le \frac{\epsilon}{2\Lambda}\right\},
\]
and redefine
\[    V_i^{\ha} := V_{i,1}^{\ha} \cap V_{i,2}^{\ha}\cap V_{i, 3}^{\ha} \quad \text{and}\quad \mathscr{B}_i^{\ha}:=B_i^{(g),\ha} \cap V_i^{\ha}
\] 
We then have
\[ \bigcap_{i=0}^I \mathscr{B}_i^{\ha} \subset \left\lbrace \sup_{x\in \T} \HS_x^{(\theta)}(u) \le 2\epsilon\right\rbrace  \cap  \left\lbrace \sup_{t\in [0,T]} \HO_t^{(\theta)}(u) \le \frac{\epsilon}{\Lambda}\right\rbrace. 
\] 
In addition, similar to \eqref{eq:v01} and \eqref{eq:v02},
Remark  \ref{rem:Nt:tail} says that 
\[ 
P\left((\widetilde V_{0,3}^{\ha})^c\right)  \le  \frac{\mathbf K_3}{\sqrt c_2\epsilon^{\frac1\theta +\frac{\eta}{2}}} \exp \left(-\frac{\mathbf K_4 }{16\sqrt c_2\Lambda^2 \mathscr D^2\epsilon^{\frac1\theta+\frac{\eta}{2}}}\right), 
\]
Now it is easy to see that 
\begin{align*}
& P\left( \sup_{\substack{0\leq s,t \leq T \\ 0\leq x,  y\leq 1\\  (t, x) \neq (s, y) }} \,  \frac{|u(t, x)-u(s, y)|}{ |t-s|^{\frac{1}{4}-\frac{\theta}{2}} + |x-y|^{\frac12-\frac{\theta}{2}}}  \leq \epsilon\left[2+\frac{1}{\Lambda}\right]   \right)  \\
 & \ge   P\left(\left\lbrace \sup_{x\in \T} \HS_x^{(\theta)}(u) \le 2\epsilon\right\rbrace  \cap  \left\lbrace \sup_{t\in [0,T]} \HO_t^{(\theta)}(u) \le \frac{\epsilon}{\Lambda}\right\rbrace\right).
\end{align*}
It then follows quite easily that  under the same assumptions of Theorem \ref{thm2}, for any $\eta>0$, there exist positive constants $C_1,C_2>0$ dependent on $\cl, \cu,\theta,\eta$ such that
\[    
P\left( \sup_{\substack{0\leq s,t \leq T \\ 0\leq x,  y\leq 1\\  (t, x) \neq (s, y) }} \,  \frac{|u(t, x)-u(s, y)|}{ |t-s|^{\frac{1}{4}-\frac{\theta}{2}} + |x-y|^{\frac12-\frac{\theta}{2}}}  \leq \epsilon   \right) \ge C_1\exp\left(- \frac{C_2 T}{\epsilon^{\frac3\theta+\eta}}\right)
\] 
This finishes the proof of Theorem \ref{thm:main}. \qed
\begin{rem}\label{rem:main} It is easy to see from the argument presented here that under the assumptions of Theorem \ref{thm1} (resp. Theorem \ref{thm3}) we have the same bounds as in \eqref{eq:unif:t} (resp.  \eqref{eq:unif:t:3}) for the H\"older semi-norm. We leave the verification to the reader.
\end{rem}

\section{Proofs of Theorems \ref{thm4} and \ref{thm5}} \label{sec:thm45}

The proof of the lower bound of Theorem \ref{thm4} relies heavily on Theorem 2.2 of \cite{kuel-li-shao}. We will use some notations from its proof and indicate only the main differences. The proofs of Theorem \ref{thm5} follow from Theorem \ref{thm4} using the same arguments used previously to deal with the non-gaussian case.

\begin{proof}[Proof of Theorem \ref{thm4} (\lowercase{a})]The upper bound is a result of Lemma \ref{lem:a1} (recall the event $A_i$ defined in \eqref{eq:ai}); note that the initial profile in Lemma \ref{lem:a1} is arbitrary. Indeed,  we might condition on the profile at time $T-c_0\delta^2$ and conclude from the above lemma that 
\[ P\left[\max_{j=0,1,\cdots, J} \big|u(T, x_j+\delta)- u(T, x_j)\big|\le \epsilon^{\frac{1}{2\theta}}\;\Big\vert\;u(T-c_0\delta^2,\cdot) \right]  \le \eta^{J}, \]
where $0<\eta<1$ and $J=\left[\frac{1}{c_1\delta}\right]$, and $c_0=1, c_1, \delta=\epsilon^{\frac1\theta}$ are as in Section \ref{sec:thm1:ubd}. From this and \eqref{eq:ubd:dis} we obtain 
\[ P\left(\HO_T^{(\theta)}(u) \le \epsilon \,\Big \vert\, u(T-c_0\delta^2,\cdot)  \right) \le \eta^J.\]
Integrating over the profile $u(T-c_0\delta^2,\cdot)$ we obtain the upper bound. 

We next turn our attention to the lower bound. As mentioned above, the proof follows along the lines of the proof of Theorem 2.2 of \cite{kuel-li-shao}, and we just sketch the necessary modifications in the proof. Recall that we assume that our initial profile $u_0\equiv 0$, and therefore 
\[ E\left[\left\{u(T,x) - u(T,y)\right\}^2\right] = E\left[\left\{N(T,x) - N(T,y)\right\}^2\right].\]
Defining $\boldsymbol \sigma^2(\gamma):=E\left[\left\{N(T,x+\gamma) - N(T,x)\right\}^2\right]$ it follows from the proof of \eqref{eq:var:bd} that 
\[ C(T) \gamma \le \boldsymbol \sigma^2(\gamma) \le \sqrt{C_1}\gamma \]
for $\gamma>0$ small enough, where $C(T)$ is a constant dependent on $T$ and $C_1$ is the constant in \eqref{eq:var:bd}. The above is the key ingredient in the proof of the lower bound. We take $\boldsymbol\beta=\theta$ and $f(x)=x^{\frac12-\theta}$ in Theorem 2.2 in \cite{kuel-li-shao}. While it is not true that $\boldsymbol \sigma(x)/x^{\boldsymbol\beta} f(x)$ is nondecreasing in $x$ as in Theorem 2.2 of \cite{kuel-li-shao}, a close examination of the proof reveals that all we require is that $\boldsymbol \sigma(ax)/f(ax) \le C_2 a^{\boldsymbol \beta} \boldsymbol \sigma(x)/ f(x)$ for some positive constant $C_2$, for all $0<a<1$ and $x$ small enough. This clearly holds for us. The sequences $x_l$ and $y_{j,l}$ encountered in the proof there should be modified by multiplying by $\frac{1}{C_2}$. Similarly, while going through the arguments of the lower bounds of the terms $A,B,C$ defined in the paper, one just gets an additional constant multiple inside the exponentials and this does not change the result.  We leave this routine checking to the interested reader. The lower bound in Theorem \ref{thm4} follows immediately from the lower bound of Theorem 2.2 in \cite{kuel-li-shao}. 
\end{proof}

\begin{proof}[Proof of Theorem \ref{thm4} (\lowercase{b})] For the upper bound, let $t_i=i\epsilon^{\frac2\theta},\, i=0,1,\cdots I= T\epsilon^{-\frac2\theta}$.
\[P \left( \HS_X^{(\theta)}\left(u\right) \le \epsilon \right) \le P\left(\frac{|u(t_{i+1},X)-u(t_i, X)|} {(t_{i+1}-t_i)^{\frac14-\frac\theta 2}} \le \epsilon,\;\; \text{ for all } i=0,1,\cdots , I\right). \]
By considering the profile at time $t_i$ we obtain 
\[ u(t_{i+1},X) =\left( G_{t_i-t_{i+1}} * u(t_i,\cdot)\right) (X) + \mathcal N(t_i, t_{i+1}, X).\]
Note that $\mathcal N(t_i, t_{i+1}, X)$ is really the noise term from time $t_i$ to $t_{i+1}$, that is thinking of time $t_i$ as the new time {\it zero} . Similar to arguments used a few times in this paper we have
\begin{align*}P\left(\frac{|u(t_{i+1},X)-u(t_i, X)|} {(t_{i+1}-t_i)^{\frac14-\frac\theta 2}} \le \epsilon \;\Big \vert\; u(s,\cdot), \, s\le t_i\right)& \le P\left(\frac{|\mathcal N(t_i, t_{i+1}, X)|} {(t_{i+1}-t_i)^{\frac14-\frac\theta 2}} \le \epsilon \Big \vert \; u(s,\cdot), \, s\le t_i\right), 
\end{align*}
which is bounded uniformly (in $i$) by a number less than $1$ (note that the variance of $ \mathcal N(t_i, t_{i+1}, X)$ is bounded above and below by constant multiples of $\epsilon^{\frac{1}{2\theta}}$). The Markov property then gives the upper bound.

 Consider the process $Y_t:= u(tT,X),\, 0\le t\le 1$. As we are under the assumption $u_0\equiv 0$ we have
 \[ E\left[\left(Y_t - Y_s\right)^2\right] = E\left[\left\{N(tT,X) - N(sT,X)\right\}^2\right].\]
Defining $\boldsymbol \sigma^2(\gamma):=E\left[\left(Y_{t+\gamma} - Y_t\right)^2\right] $ and using \eqref{eq:var:t_incr} and \eqref{eq:var:t:g-g} one obtains
\[ C_1 \sqrt{T\gamma} \le \boldsymbol \sigma^2(\gamma) \le C_2\sqrt{T\gamma} \]
for constants $C_1, C_2$ independent of $T$. One can then follow the argument of the lower bound of Theorem \ref{thm4} for the process $Y_t$, now with $f(x) = x^{\frac14-\frac\theta 2}$ and $\boldsymbol \beta=\frac\theta 2$. 
\end{proof}

\begin{proof}[Proof of Theorem \ref{thm5} (\lowercase{a})] The proof of the upper bound is similar to Theorem \ref{thm2} but instead we use \eqref{eq:prob-main} and note that this bound is uniform over the initial profiles $u_0$. We can then conclude 
\[ P\left(\HO_T^{(\theta)}(u) \le \epsilon \,\Big \vert\, u(T-c_0\epsilon^{\frac2\theta},\cdot)  \right) \le \exp\left(-\frac{C}{|\log \epsilon|^{\frac32} \epsilon^{\frac1\theta}}\right).\]
Now integrate over the profile at time $T-c_0\epsilon^{\frac2\theta}$.\end{proof}

\begin{proof}[Proof of Theorem \ref{thm5} (\lowercase{b})] 
The proof is very similar to that of the proof of the upper bound of Theorem \ref{thm4} (\lowercase{b}). The only difference is that now $\mathcal N(t_i, t_{i+1}, X)$ is no longer Gaussian. For $t_i\leq s\leq t_{i+1}$, we note that 
\begin{equation*}
\mathcal N(t_i, s, X)=\int_{t_i}^s \int_{\T} G_{t_{i+1}-r}(X, y)  \cdot \sigma\left(r, y, u(r,y)\right) W(dy dr)
\end{equation*}
is a martingale. Similar arguments to that of the proof of \eqref{eq:diff-u-x} and an application of the Markov property complete the proof.

\end{proof}

\section{Some extensions} \label{sec:ext}
In this section, we provide \emph{support theorems} in the H\"older semi-norm, which are similar to the support theorem in the sup norm in \cite{athr-jose-muel}.  We provide  probabilities  that the solution $u$ stays close to a function in H\"older spaces $C^{\gamma, \beta}$ (see Theorems \ref{thm-support-space} and \ref{thm-support-time} for the precise statements). These theorems are of a different flavour from the support theorem proved in \cite{ball-mill-sanz}, where a description of the support set of the solution is given. 

 We first consider small ball probabilities of \eqref{eq:she} with \emph{nice} drifts.  By means of a change of measure argument, we can show that all of our results are still valid when we add a bounded drift term to the equation. Consider
\begin{equation}\label{eq:she-drift}
\begin{split} \partial_t u (t,x)&= \frac12\,\partial_x^2 u(t,x) + g(t,x,u)+\sigma\big(t,x, u(t,x)\big) \cdot \dot{W}(t,x),\quad t
\in \R_+,\, x\in \T, 
\end{split}
\end{equation}
The proof of the following theorem follows exactly the argument given in Section 2.2 of \cite{athr-jose-muel} and is left to the reader to verify. 
\begin{thm} \label{thm:drift} Consider \eqref{eq:she-drift}, where the assumptions \eqref{eq:ellip} and \eqref{eq:lip} on $\sigma$ hold, and $g(t,x,u):\R^+\times \T\times \R\rightarrow \R$ is bounded in absolute value by a constant $\mathbb G$ and globally Lipschitz in the third variable (that is, there is a $\tilde \du$ such that $|g(t,x,v)-g(t,x,w)|\le \tilde \du |v-w|$). Then the statements of Theorems \ref{thm:main}, \ref{thm1}, \ref{thm2} and \ref{thm3}  still hold, with the constants now depending additionally on $\mathbb G$ but not on $\tilde \du$.
\end{thm}
 Note that the Lipschitz condition on $g$ is just to gaurantee uniqueness and existence of solutions to \eqref{eq:she-drift}. We also have the following result which is analogous to that of Theorem 1.2 of \cite{athr-jose-muel}. 
\begin{prop}\label{prop-support}
Consider the solution to \eqref{eq:she}. Let $h:\R^+\times \T \rightarrow \R$ be a smooth function such that $h, \partial_t h$ and $\partial_x^2h$ are uniformly bounded by a constant $H$. Let $0<\theta<\frac12$ and $0<\epsilon<1$ and suppose that $\HO_0^{(\theta)}(u-h) \le \frac{\epsilon}{2} \left(1\wedge  \frac{1}{2\Lambda}  \right)$ where $\Lambda$ is given in \eqref{eq:lambda}.  
\begin{itemize}
\item[(a)] Suppose that the function $\sigma(t,x,u)$ is independent of $u$ but satisfies Assumption \ref{assump1}.  Then there exist positive constants $C_1, C_2, C_3, C_4>0$ depending on $\mathscr{C}_1, \mathscr{C}_2, \theta$ and $H$ such that 
\begin{equation*}  C_1\exp\left(- \frac{C_2 T}{\epsilon^{\frac3\theta}}\right) \le P \left(\sup_{0\le t\le T} \HO_t^{(\theta)}\left(u-h\right) \le \epsilon \right)\le C_3\exp\left(- \frac{C_4T}{\epsilon^{\frac 3\theta}}\right),
\end{equation*} and

\begin{equation*}  C_1\exp\left(- \frac{C_2 T}{\epsilon^{\frac3\theta}}\right) \le P \left(\sup_{x\in \T} \HS_x^{(\theta)}\left(u-h\right) \le \epsilon \right)\le C_3\exp\left(- \frac{C_4T}{\epsilon^{\frac 3\theta}}\right).
\end{equation*}

\item[(b)] Suppose that $\sigma(t,x,u)$ is now dependent on $u$ and satisfies both Assumptions \ref{assump1} and \ref{assump2}. Then for any $\eta>0$, there exist positive constants $C_1, C_2, C_3, C_4>0$ depending on $\mathscr{C}_1, \mathscr{C}_2, \theta$ and $H$ such that 
\begin{equation*}
 C_1\exp\left(- \frac{C_2T}{\epsilon^{{\frac 3\theta}+\eta}}\right) \le  P \left(\sup_{0\le t\le T} \HO_t^{(\theta)}\left(u-h\right) \le \epsilon \right)\le C_3\exp\left(- \frac{C_4T}{\epsilon^{\frac3\theta}|\log \epsilon|^{\frac32}}\right), 
 \end{equation*} and
 \begin{equation*} C_1\exp\left(- \frac{C_2 T}{\epsilon^{\frac3\theta+\eta}}\right) \le P \left(\sup_{x\in \T} \HS_x^{(\theta)}\left(u-h\right) \le \epsilon \right)\le  C_3\exp\left(- \frac{C_4T}{\epsilon^{\frac3\theta}|\log \epsilon|^{\frac32}}\right).
\end{equation*}

 \item[(c)] Suppose that $\sigma(t,x,u)$ is again dependent on $u$ and satisfies both Assumptions \ref{assump1} and \ref{assump2}. Then there is a $\mathcal{D}_0>0$ such that for all $\mathcal{D}<\mathcal{D}_0$, there exists positive constants $C_1, C_2, C_3, C_4>0$ depending on $\mathscr{C}_1, \mathscr{C}_2, \theta$ and $H$ such that 
 
 \begin{equation*}
 C_1\exp\left(- \frac{C_2T}{\epsilon^{\frac 3\theta}}\right) \le  P \left(\sup_{0\le t\le T} \HO_t^{(\theta)}\left(u-h\right) \le \epsilon \right)\le C_3\exp\left(- \frac{C_4T}{\epsilon^{\frac3\theta}|\log \epsilon|^{\frac32}}\right), 
 \end{equation*} and 
 
 \begin{equation*}  C_1\exp\left(- \frac{C_2 T}{\epsilon^{\frac3\theta}}\right) \le P \left(\sup_{x\in \T} \HS_x^{(\theta)}\left(u-h\right) \le \epsilon \right)\le  C_3\exp\left(- \frac{C_4T}{\epsilon^{\frac3\theta}|\log \epsilon|^{\frac32}}\right).
\end{equation*}

\end{itemize}
\end{prop}

 We quickly discuss how the above follows from Theorem \ref{thm:drift} and the previous resuls. Consider $w(t,x)=u(t,x)-h(t,x)$. The reader can check that $w$ satisfies 
\[ \partial_t w = \frac12 \partial_x^2 w +\left[\frac12\partial_x^2 h-\partial_t h\right] +\tilde \sigma(t,x,w) \dot W, \]
where $w_0= u_0(x) -h(0,x)$, and $\tilde\sigma(t,x,w):= \sigma\left(t,x,w+h(t,x)\right)$. One also observes that $\tilde \sigma$ satisfies the same assumptions as that of $\sigma$, and the function $g(t,x):=\frac12\partial_x^2 h-\partial_t h$ satisfies the assumptions of Theorem \ref{thm:drift}. Proposition \ref{prop-support} follows from Theorem \ref{thm:drift} applied to $w$.

We next increase the collection of functions $h$ for which we can prove support theorems similar in spirit to Proposition \ref{prop-support}.

\begin{defn}\label{assumpf}
We say that  $f:[0, T]\times \T \rightarrow \R$ is in $C^{\gamma, \beta}$ if we have 
\begin{equation*}
 \|f\|_{C^{\gamma, \beta}} : = |f(0, 0)|+ \sup_{\substack{0\leq s,t \leq T \\ x,  y\in \T \\  (t, x) \neq (s, y) }} \,  \frac{|f(t, x)-f(s, y)|}{ |t-s|^\gamma + |x-y|^{\beta}}<\infty. 
\end{equation*}
\end{defn}
In other words,  $C^{\gamma, \beta}$ is the set of functions that are uniformly bounded, H\"older continuous with the exponent $\gamma$ in time and the exponent $\beta$ in space. 

Let $\psi:\R \to \R$ be a non-negative, symmetric  and smooth function such that the support of $\psi$ is in $[-1, 1]$ and $\int_{\R} \psi(x) \, dx=1$. For any positive integer $n$ and $f\in C^{\gamma, \beta}$, we set   $\psi_n(x):=n \psi(nx)$ and define 
\begin{equation}\label{smooth}
f_n(t,x)=\iint_{\R^2} \tilde f(s,y)\psi_n(x-y)\psi_n(t-s) d y d s,
\end{equation}
where $\tilde f$ is the periodization of $f$ in the spatial variable $x$ and we also define  $\tilde f(s, x)=f(0, x)$ for $s< 0$ and $x\in \R$ and $\tilde f(s, x)=f(T, x)$ for $s>T$ and $x\in \R$. We have the following bounds on the derivatives of the above function. The proof is straightforward and is therefore omitted. 

\begin{lem}\label{deri-approx}
Suppose  $f\in C^{\gamma, \beta}$. Then there exists a constant $C>0$ such that for all $x\in \T$ and $t\in [0, T]$ 
\[ \left|\frac{\partial f_n(t,x)}{\partial x}\right|\leq Cn,\qquad  \left|\frac{\partial f_n(t,x)}{\partial t}\right|\leq Cn, \qquad \left|\frac{\partial^2 f_n(t,x)}{\partial x^2}\right|\leq Cn^2.
\]
\end{lem}
The following lemma shows we can approximate $f\in C^{\gamma, \beta}$ by smooth mollifications of $f$.  
\begin{lem}\label{lem:approx}
Let $f:[0,T]\times\T\rightarrow \R$ be in $C^{\gamma, \beta}$ for some $\gamma, \beta \in (0,  1]$. Consider the sequence of smooth functions $\{f_n\}_{n=1}^{\infty}$ defined by \eqref{smooth}. Let $\beta_1 \in (0, \beta)$ and $\gamma_1 \in (0, \gamma)$. Then, for any fixed $\epsilon>0$, there exist constants $C_1(\epsilon)$ and $C_2(\epsilon)$  such that 
 we have 
\begin{equation}\label{eq:f-fn-time}
\sup_{0\leq t\leq T}\HO^{\left(\frac 12-\beta_1\right)}_t (f_n-f)\leq  \epsilon\quad \text{as}\quad n\geq C_1(\epsilon),
\end{equation}
and 
\begin{equation}\label{eq:f-fn-space}
\sup_{n\in \T}\HS^{\left(\frac 12-2\gamma_1 \right)}_x (f_n-f)\leq  \epsilon\quad \text{as}\quad n\geq C_2(\epsilon).
\end{equation} 
\end{lem} 

\begin{proof}
We start by making the following observation. Since for each $t\geq 0$, $f(t, \cdot)$ is H\"older($\beta$) continuous, for $\beta_1<\beta$ we have
\begin{align*}
\frac{|f(t,x)-f(t,y)|}{|x-y|^{\beta_1}}&\leq C_1|x-y|^{\beta-\beta_1},
\end{align*}
where $C_1$ is a positive constant that is independent of $t$. We also have 
\begin{align*}
\frac{|f_n(t,x)-f_n(t,y)|}{|x-y|^{\beta_1}}&=\frac{\left|\iint [f(s,x-z)-f(s,y-z)]\psi_n(z)\psi_n(t-s)dz ds\right|}{|x-y|^{\beta_1}}\\
&\leq C_1|x-y|^{\beta-\beta_1}.
\end{align*}
We therefore obtain  
\begin{align*}
\sup_{0\leq t\leq T}  \left| \frac{(f_n(t, x)-f(t, x))-(f_n(t, y)-f(t,y))}{|x-y|^{\beta_1}}\right| \leq \epsilon
\end{align*}
whenever we choose $|x-y|\leq\left( \frac{\epsilon}{2C_1}\right)^{\frac{1}{\beta-\beta_1}}$. We now consider $f(t, x) - f_n(t, x)$: 

\begin{align*}
f(t,x)-f_n(t,x)&=\iint [f(t, x)-f(s, y)]\psi_n(t-s)\psi_n(x-y) \, d y d s\\
& = \iint \left[ f(t, x) - f(t-s, x-y) \right] \psi_n(s) \psi_n(y) \, d y ds\\
&=\iint \left[ \frac{f(t, x) - f(t-s, x)}{|s|^\gamma}  \right] |s|^\gamma \psi_n(s) \psi_n(y) \, d y ds\\
&\qquad  + \iint \left[ \frac{f(t-s, x) - f(t-s, x-y)}{|y|^\beta}  \right] |y|^\beta \psi_n(s) \psi_n(y) \, d y ds. 
\end{align*}
Since $f\in C^{\gamma, \beta}$,  $\psi_n(x)=0$ if $|x| > 1/n$ and $\int \psi_n(x) \, dx =1$, there exists some constant $C_2>0$ such that 
\begin{equation}\label{eq:f-fn-sup}
 \sup_{0\leq t \leq T}\sup_{x\in \T} |f(t, x)-f_n(t, x)|\leq C_2 \left( n^{-\gamma} + n^{-\beta}\right).
 \end{equation}
Hence, for all $t\in [0, T]$ and $x,y\in \T$ satisfying $|x-y|\geq \left( \frac{\epsilon}{2C_1}\right)^{\frac{1}{\beta-\beta_1}}$,  there exists some constant $C_3>0$ which only depends on $\beta, \beta_1$ such that  
\begin{align*}
\left| \frac{(f_n(t, x)-f(t, x))-(f_n(t, y)-f(t,y))}{|x-y|^{\beta_1}}\right| &\leq 2\left( \frac{2C_1}{\epsilon}\right)^{\beta_1/(\beta-\beta_1)} \sup_{x\in \T} |f(t, x)-f_n(t, x)|\\
&\leq C_3 \left( n^{-\gamma} + n^{-\beta}\right)\epsilon^{\beta_1/(\beta_1-\beta)}. 
\end{align*} 
We have therefore proved \eqref{eq:f-fn-time} for all large enough $n \geq C_1(\epsilon)$ where 
\begin{equation}\label{eq:C1}
C_1(\epsilon):=\max \left\{  \left(2C_3 \epsilon^{\beta/(\beta_1-\beta)}\right)^{1/\gamma}, \left(2C_3\epsilon^{\beta/(\beta_1-\beta)}\right)^{1/\beta}   \right\}.
\end{equation} 
For \eqref{eq:f-fn-space}, we follow the same proof above but switch $\beta$ by $\gamma$ to get \eqref{eq:f-fn-space}. Here, we need  $n\geq C_2(\epsilon)$ where 
\begin{equation}\label{eq:C2}
C_2(\epsilon):=\max \left\{  \left(2C_3 \epsilon^{\gamma/(\gamma_1-\gamma)}\right)^{1/\gamma}, \left(2C_3\epsilon^{\gamma/(\gamma_1-\gamma)}\right)^{1/\beta}   \right\}.
\end{equation}
This completes the proof of the lemma.
\end{proof}

\begin{rem} 
It is easy to see from  \eqref{eq:f-fn-time}, \eqref{eq:f-fn-space} and \eqref{eq:f-fn-sup}  that  every $f\in C^{\gamma, \beta}$ can be approximated  by its smooth mollification $f_n$ in $\|\cdot\|_{C^{\gamma_1, \beta_1}}$ for all $\gamma_1<\gamma$ and $\beta_1<\beta$. That is, for any $\epsilon>0$ and for  every $f\in C^{\gamma, \beta}$, there exists a constant $C(\epsilon)>0$ such that  
\begin{equation}
 \|f_n - f\|_{C^{\gamma_1, \beta_1}}  \leq \epsilon \qquad \text{for $n\geq C(\epsilon)$}. 
 \end{equation}  
\end{rem}

We can now drop the assumption that $h$ is smooth in Proposition \ref{prop-support}. We obtain bounds on the small ball probabilities when $h \in C^{\gamma, \beta}$. The upper bounds remain the same as before, but the lower bounds now depend on $\beta$ and $\gamma$. We now have to treat the spatial and temporal regularities of $u-h$ differently. We first consider the spatial difference of $u-h$, i.e., $\HO_t^{(\theta)}(u-h)$.

\begin{thm}\label{thm-support-space}
Consider the solution to \eqref{eq:she}. Let   $0<\theta<\frac12$ and  $0<\epsilon<1$. Suppose  $h:[0, T]\times \T\rightarrow \R$ is in  $C^{\gamma,\beta}$ with $\frac{1}{2}-\theta<\beta\leq 1$ and $\gamma \in(0,1]$. We also assume  $\HO_0^{(\theta)}(u-h) \le \frac{\epsilon}{4}$. Then we have the following: 
\begin{itemize}
\item[(a)] Suppose that the function $\sigma(t,x,u)$ is independent of $u$ but satisfies Assumption \ref{assump1}.  Then there exist positive constants $C_1, C_2, C_3, C_4>0$ depending on $\mathscr{C}_1, \mathscr{C}_2, \theta,  \beta$ and $\gamma$ such that 
\begin{equation*}  C_1\exp\left(- C_2 T\left[ \frac{1}{\epsilon^{\frac3\theta}}+ \frac{1}{\epsilon^{\frac{4 (\beta \vee \gamma)}{\gamma(\beta +\theta -1/2)}}}  \right] \right) \le P \left(\sup_{0\le t\le T} \HO_t^{(\theta)}\left(u-h\right) \le \epsilon \right) \le C_3\exp\left(- \frac{C_4T}{\epsilon^{\frac 3\theta}}\right).
\end{equation*}
\item[(b)] Suppose that $\sigma(t,x,u)$ is now dependent on $u$ and satisfies both Assumptions \ref{assump1} and \ref{assump2}. Then for any $\eta>0$, there exist positive constants $C_1, C_2, C_3, C_4>0$ depending on $\mathscr{C}_1, \mathscr{C}_2, \theta, \beta$ and $\gamma$ such that 
\begin{equation*}
 C_1\exp\left(- C_2 T\left[\frac{1}{\epsilon^{\frac3\theta+\eta}}+\frac{1}{\epsilon^{\frac{4 (\beta \vee \gamma)}{\gamma(\beta +\theta -1/2)}}}   \right]\right) \le  P \left(\sup_{0\le t\le T} \HO_t^{(\theta)}\left(u-h\right) \le \epsilon \right)\le C_3\exp\left(- \frac{C_4T}{\epsilon^{\frac3\theta}|\log \epsilon|^{\frac32}}\right).
 \end{equation*}
 \item[(c)] Suppose that $\sigma(t,x,u)$ is again dependent on $u$ and satisfies both Assumptions \ref{assump1} and \ref{assump2}. Then there is a $\mathcal{D}_0>0$ such that for all $\mathcal{D}<\mathcal{D}_0$, there exists positive constants $C_1, C_2, C_3, C_4>0$ depending on $\mathscr{C}_1, \mathscr{C}_2, \theta, \beta$ and $\gamma$ such that 
 \begin{equation*}
C_1\exp\left(- C_2 T\left[\frac{1}{\epsilon^{\frac3\theta}}+\frac{1}{\epsilon^{\frac{4 (\beta \vee \gamma)}{\gamma(\beta +\theta -1/2)}}}   \right]\right) \le  P \left(\sup_{0\le t\le T} \HO_t^{(\theta)}\left(u-h\right) \le \epsilon \right)\le C_3\exp\left(- \frac{C_4T}{\epsilon^{\frac3\theta}|\log \epsilon|^{\frac32}}\right).\end{equation*}
\end{itemize}
\end{thm}
\begin{proof} The proof is similar to the proof of Theorem 1.2 of \cite{athr-jose-muel}, but here we use the approximation procedure presented above.  That is,  we set  $\beta_1:=\frac{1}{2}-\theta$ in Lemma \ref{lem:approx} and  $n=  C_1(\epsilon)$ where $C_1(\epsilon)$ is given in \eqref{eq:C1}, and  define a smooth function $h_n$ by \eqref{smooth}. Then,  we have 
\begin{equation} \label{h-hn}
\sup_{0\leq t\leq T}\HO_t^{(\theta)}\left(h_n-h\right)\leq \frac \epsilon 2.
\end{equation}
Since $\HO_t^{(\theta)}(u-h) \leq \HO_t^{(\theta)}(u-h_n) + \HO_t^{(\theta)}(h-h_n)$, \eqref{h-hn} implies that  
\begin{align*}
P\left(\sup_{0\leq t\leq T}\HO_t^{(\theta)}\left(u-h_n\right)\leq \frac{\epsilon}{2}\right) \leq P\left(\sup_{0\leq t\leq T}\HO_t^{(\theta)}\left(u-h\right)\leq \epsilon\right).
\end{align*}
By Lemma \ref{deri-approx},  there exists a constant $C$  so that 
\begin{equation}\label{eq:bound_h}
\sup_{t\in[0, T]}\sup_{x\in \T}\left|\left( \partial_t - \partial_x^2 \right)   h_n(t, x)\right|\leq Cn^2.
\end{equation}
Here, $\left(\partial_t - \frac{1}{2}\partial_x^2 \right)   h_n(t, x)$ is the drift term when we consider the differential form of  $u(t, x) - h_n(t, x)$. That is, if we let $\tilde u_n:=u-h_n$, then $\tilde u$ satisfies 
\begin{align*}
\partial_t \tilde u_n& = \partial_t u -\partial_t h_n \\
&=\frac{1}{2}\partial_x^2 u  + \sigma(t, x, u) \dot W - \partial_t h_n\\
&=\frac{1}{2}\partial_x^2 \tilde u_n - \left( \partial_t h_n -\frac{1}{2}\partial_x^2 h_n \right)+ \tilde\sigma_n(t, x, \tilde u_n) \dot W, 
\end{align*} 
where $\tilde \sigma_n(t, x, z):=\sigma(t, x, z + h_n(t, x))$.  A close inspection of the proof of Theorem 1.2 of \cite{athr-jose-muel} shows that in the case of (a)
\begin{equation*}  C_1\exp\left(- C_2 T\left[\frac{1}{\epsilon^{\frac3\theta}}+n^4\right]\right) \le P \left(\sup_{0\le t\le T} \HO_t^{(\theta)}\left(u-h\right) \le \epsilon \right).
\end{equation*} 
Here, $n^4$ comes from \eqref{eq:bound_h}. 
Recalling the choice of $n$ finishes the proof of the lower bound in  part (a). The arguments for the lower bounds in (b) and (c) are similar.  

Let us  now consider the upper bounds. First, we prove the upper bound in part (a). Here we also use the approximation procedure.  That is, we choose and fix $n$ large enough to get \eqref{h-hn}. Then, by triangle inequality, we have 
\begin{align*}
P\left(\sup_{0\leq t\leq T}\HO_t^{(\theta)}\left(u-h\right)\leq \frac{\epsilon}{2}\right)\leq P\left(\sup_{0\leq t\leq T}\HO_t^{(\theta)}\left(u-h_n\right)\leq \epsilon \right).
\end{align*}
 Let $v_n(t, x):=u(t, x) - h_n(t, x)$. Then, $v$ satisfies 
\begin{equation}\label{v-drift}
\begin{split} \partial_t v_n (t,x)&= \frac12\,\partial_x^2 v_n(t,x)- g_n(t,x)+  \sigma \big(t,x \big) \cdot \dot{W}(t,x),\quad t
\in \R_+,\, x\in \T, 
\end{split}
\end{equation}
where $g_n(t, x) = \left(\partial_t -  \frac12\partial_x^2 \right) h_n (t, x)$.  
To get the upper bound in part (a),  we just follow the proof of the upper bound of  Theorem \ref{thm1} (a). Note that the upper bound of Theorem \ref{thm1} (a) is obtained once \eqref{eq:cvar:lbd} is proved.  The only difference from \eqref{v-drift} to \eqref{eq:she} is that we have the additional drift term $g_n(t, x)$ in \eqref{v-drift}. However, the drift term does not have any effect in obtaining the upper bound. More precisely,  similar to   \eqref{eq:cvar:lbd},  we need to show that there exists a constant $C>0$ such that 
\begin{equation} \label{eq:var-bar}
  \Var \left(\bar \Delta_j\, \Big \vert \, \mathcal G_{j-1} \right) \geq C\epsilon^{1/\theta},
\end{equation}
where 
\begin{align*}
& \bar\Delta_j:=\tilde \Delta_j + \Big[\left(G_{t_1}*u_0\right)(x+\delta) -\left(G_{t_1}*u_0\right)(x)\Big] \\
&\qquad \qquad +\left[ \int_0^{t_1} \int_{\T} \left( G_{t_1 -s }(y-x_j-\delta) - G_{t_1 -s }(y-x_j-\delta) \right) g_n(s, y) \, dy \, ds   \right], 
 \end{align*} and $\tilde \Delta_j$ is given in \eqref{noise_diff}. Here, since $g_n$ is deterministic, we have 
 \[ \Var \left( \bar\Delta_j\, \Big \vert \, \mathcal G_{j-1} \right) = \Var \left(\tilde  \Delta_j\, \Big \vert \, \mathcal G_{j-1} \right).\] 
  Thus, \eqref{eq:cvar:lbd} implies \eqref{eq:var-bar}, which leads to the upper bound in part (a), that is the same upper bound in Theorem \ref{thm1}.  
 
For the upper bound in part (b), we  follow  the proof of the upper bound in Theorem \ref{thm2} (a). That is,  we add $-h(t, x)$ to the mild form of $u(t, x)$ and also add $-h(t, x)$  to  $V^{(\beta)}$ and $V^{(\beta, l)}$ in \eqref{truncate} and \eqref{iterates}, then  follow the proof of the upper bound in  Theorem \ref{thm2} (a).    It is easy to see that Propositions \ref{prop:V-Vl} and \ref{prop:u-V}, and Lemma \ref{lemma:ind} still hold. In addition, as in  \eqref{eq:diff-u-x},   we can regard $\left( u(t_1, x_{2j+1}) - u(t_1, x_{2j}) \right) - \left(h(t_1, x_{2j+1}) - h(t_1, x_{2j})   \right) $ as 
\[ M_0- \left(h(t_1, x_{2j+1}) - h(t_1, x_{2j}) \right) + B_{\langle M \rangle_{t_1}}.\]
Following the same proof of \eqref{eq:diff-u-x}, we obtain
\begin{equation}\label{eq:u-h} 
P\left( \left|\left( u(t_1, x_{2j+1}) - u(t_1, x_{2j}) \right) - \left(h(t_1, x_{2j+1}) - h(t_1, x_{2j})   \right)  \right|  \leq 5\epsilon^{1/2\theta} \right) \leq \gamma, 
\end{equation} where $\gamma$ is given in \eqref{eq:diff-u-x}.

 The proof of the upper bound in part (c) is exactly the same as the one for the upper bound in part (b). 
\end{proof}

 Similar to Theorem \ref{thm-support-space}, we now provide small ball probabilities of the temporal H\"older semi-norms of $u-h$.  We skip the proof since one can basically follow the proof of Theorem  \ref{thm-support-space}.  
  
\begin{thm}\label{thm-support-time}
Consider the solution to \eqref{eq:she}. Let   $0<\theta<\frac12$ and  $0<\epsilon<1$. Suppose  $h:[0, T]\times \T\rightarrow \R$ is in  $C^{\gamma,\beta}$ with $\frac{1}{4}-\frac{\theta}{2}<\gamma \leq 1$ and $\beta \in(0,1]$. We also assume  $\HS_0^{(\theta)}(u-h) \le \frac{\epsilon}{4\Lambda}$. Then we have the following: 
\begin{itemize}
\item[(a)] Suppose that the function $\sigma(t,x,u)$ is independent of $u$ but satisfies Assumption \ref{assump1}.  Then there exist positive constants $C_1, C_2, C_3, C_4>0$ depending on $\mathscr{C}_1, \mathscr{C}_2, \theta,  \beta$ and $\gamma$ such that 
\begin{equation*}  C_1\exp\left(- C_2 T\left[ \frac{1}{\epsilon^{\frac3\theta}}+ \frac{1}{\epsilon^{\frac{4 (\beta \vee \gamma)}{\beta\left(\gamma +\frac{\theta}{2} -\frac{1}{4}\right)}}}  \right] \right) \le P \left(\sup_{ x \in \T} \HS_x^{(\theta)}\left(u-h\right) \le \epsilon \right) \le C_3\exp\left(- \frac{C_4T}{\epsilon^{\frac 3\theta}}\right).
\end{equation*}
\item[(b)] Suppose that $\sigma(t,x,u)$ is now dependent on $u$ and satisfies both Assumptions \ref{assump1} and \ref{assump2}. Then for any $\eta>0$, there exist positive constants $C_1, C_2, C_3, C_4>0$ depending on $\mathscr{C}_1, \mathscr{C}_2, \theta, \beta$ and $\gamma$ such that 
\begin{equation*}
 C_1\exp\left(- C_2 T\left[\frac{1}{\epsilon^{\frac3\theta+\eta}}+\frac{1}{\epsilon^{\frac{4 (\beta \vee \gamma)}{\beta\left(\gamma +\frac{\theta}{2} -\frac{1}{4}\right)}}}   \right]\right) \le  P \left(\sup_{x\in \T} \HS_x^{(\theta)}\left(u-h\right) \le \epsilon \right)\le C_3\exp\left(- \frac{C_4T}{\epsilon^{\frac3\theta}|\log \epsilon|^{\frac32}}\right).
 \end{equation*}
 \item[(c)] Suppose that $\sigma(t,x,u)$ is again dependent on $u$ and satisfies both Assumptions \ref{assump1} and \ref{assump2}. Then there is a $\mathcal{D}_0>0$ such that for all $\mathcal{D}<\mathcal{D}_0$, there exists positive constants $C_1, C_2, C_3, C_4>0$ depending on $\mathscr{C}_1, \mathscr{C}_2, \theta, \beta$ and $\gamma$ such that 
 \begin{equation*}
C_1\exp\left(- C_2 T\left[\frac{1}{\epsilon^{\frac3\theta}}+\frac{1}{\epsilon^{\frac{4 (\beta \vee \gamma)}{\beta\left(\gamma +\frac{\theta}{2} -\frac{1}{4}\right)}}}   \right]\right) \le  P \left(\sup_{x\in \T} \HS_x^{(\theta)}\left(u-h\right) \le \epsilon \right)\le C_3\exp\left(- \frac{C_4T}{\epsilon^{\frac3\theta}|\log \epsilon|^{\frac32}}\right).\end{equation*}
\end{itemize}
\end{thm}
We end with a remark.
\begin{rem}Support theorems involving the H\"older semi-norm used in Theorem \ref{thm:main} can be obtained by a combination of Theorems \ref{thm-support-space} and \ref{thm-support-time}. We leave these to the reader.
\end{rem}

{\bf Acknowledgement:} We thank the referees for a careful reading of the paper and for various suggestions which led to improvement in the exposition. We also thank them for the question raised in Remark \ref{openq}.  This work was initiated when M.F. and M.J. visited POSTECH in September 2019. They would like to thank the mathematics department for the hospitality during their visit. M.J. was partially supported by SERB MATRICS grant MTR/2020/000453 and a CPDA grant from I.S.I., and K.K. is partially supported by the NRF (National Research Foundation of Korea) grants 2019R1A5A1028324 and 2020R1A2C4002077.

\bibliography{hol_small}

\newcommand{\etalchar}[1]{$^{#1}$}
\providecommand{\bysame}{\leavevmode\hbox to3em{\hrulefill}\thinspace}
\providecommand{\MR}{\relax\ifhmode\unskip\space\fi MR }
\providecommand{\MRhref}[2]{%
  \href{http://www.ams.org/mathscinet-getitem?mr=#1}{#2}
}
\providecommand{\href}[2]{#2}
\begin{thebibliography}{DKM{\etalchar{+}}09}

\bibitem[AJM]{athr-jose-muel}
Siva Athreya, Mathew Joseph, and Carl Mueller, \emph{Small ball probabilities
  and a support theorem for the stochastic heat equation {\tt
  https://arxiv.org/pdf/2006.07978.pdf}}, To appear in Annals of Probability.

\bibitem[BAGL94]{AGM}
G\'{e}rard Ben~Arous, Mihai Gr\u{a}dinaru, and Michel Ledoux, \emph{H\"{o}lder
  norms and the support theorem for diffusions}, Ann. Inst. H. Poincar\'{e}
  Probab. Statist. \textbf{30} (1994), no.~3, 415--436.

\bibitem[BMSS95]{ball-mill-sanz}
Vlad Bally, Annie Millet, and Marta Sanz-Sol\'{e}, \emph{Approximation and
  support theorem in {H}\"{o}lder norm for parabolic stochastic partial
  differential equations}, Ann. Probab. \textbf{23} (1995), no.~1, 178--222.

\bibitem[BR92]{bald-royn}
Paolo Baldi and Bernard Roynette, \emph{Some exact equivalents for the
  {B}rownian motion in {H}\"{o}lder norm}, Probab. Theory Related Fields
  \textbf{93} (1992), no.~4, 457--484.

\bibitem[CJK13]{conu-jose-khos}
Daniel Conus, Mathew Joseph, and Davar Khoshnevisan, \emph{On the chaotic
  character of the stochastic heat equation, before the onset of
  intermittency}, The Annals of Probability \textbf{41} (2013), no.~3B,
  2225--2260.

\bibitem[DKM{\etalchar{+}}09]{spde-mini}
Robert Dalang, Davar Khoshnevisan, Carl Mueller, David Nualart, and Yimin Xiao,
  \emph{A minicourse on stochastic partial differential equations}, Lecture
  Notes in Mathematics, vol. 1962, Springer-Verlag, Berlin, 2009, Held at the
  University of Utah, Salt Lake City, UT, May 8--19, 2006, Edited by
  Khoshnevisan and Firas Rassoul-Agha.

\bibitem[HSWX20]{HSWX}
Randall Herrell, Renming Song, Dongsheng Wu, and Yimin Xiao, \emph{Sharp
  space-time regularity of the solution to stochastic heat equation driven by
  fractional-colored noise}, Stoch. Anal. Appl. \textbf{38} (2020), no.~4,
  747--768.

\bibitem[KL93]{kuelbs-Li}
James Kuelbs and Wenbo Li, \emph{Small ball estimates for brownian motion and
  the brownian sheet}, J. Theoret. Probab. \textbf{6} (1993), 547--597.

\bibitem[KLS95]{kuel-li-shao}
James Kuelbs, Wenbo Li, and Qi~Man Shao, \emph{Small ball probabilities for
  {G}aussian processes with stationary increments under {H}\"{o}lder norms}, J.
  Theoret. Probab. \textbf{8} (1995), no.~2, 361--386.

\bibitem[LaM17]{lata-matl}
Rafa\l Lata\l~a and Dariusz Matlak, \emph{Royen's proof of the {G}aussian
  correlation inequality}, Geometric aspects of functional analysis, Lecture
  Notes in Math., vol. 2169, Springer, Cham, 2017, pp.~265--275.

\bibitem[Lot17]{Lot}
Sergey Lototsky, \emph{Small ball probabilities for the infinite-dimensional
  ornstein-uhlenbeck process in sobolev spaces.}, Stoch. Partial Differ. Equ.
  Anal. Comput \textbf{5} (2017), no.~2.

\bibitem[LS00]{LiSha98}
Wenbo Li and Qi~Man Shao, \emph{A note on the {G}aussian correlation
  conjecture}, Progr. Probab., vol.~47, pp.~163--171, Birkh\"{a}user Boston,
  Boston, MA, 2000.

\bibitem[LS01]{LS}
Wenbo Li and Qi~Man. Shao, \emph{Gaussian processes: inequalities, small ball
  probabilities and applications}, Stochastic processes: theory and methods,
  vol.~19, North-Holland, Amsterdam, 2001, pp.~533--597.

\bibitem[Mar04]{Mar}
Andreas Martin, \emph{Small ball asymptotics for the stochastic wave equation},
  J. Theoret. Probab. \textbf{17} (2004), no.~3, 693--703.

\bibitem[RB00]{bhim-rao}
Ayyagari~Ramachandra Rao and Pochiraju Bhimasankaram, \emph{Linear algebra},
  second ed., Texts and Readings in Mathematics, vol.~19, Hindustan Book
  Agency, New Delhi, 2000.

\bibitem[Roy14]{roye}
Thomas Royen, \emph{A simple proof of the {G}aussian correlation conjecture
  extended to some multivariate gamma distributions}, Far East J. Theor. Stat.
  \textbf{48} (2014), no.~2, 139--145.

\bibitem[TX17]{TX}
Ciprian Tudor and Yimin Xiao, \emph{Sample paths of the solution to the
  fractional-colored stochastic heat equation}, Stochastic and Dynamics
  \textbf{17} (2017), no.~1.

\bibitem[Wal86]{wals}
John~B. Walsh, \emph{An introduction to stochastic partial differential
  equations}, \'{E}cole d'\'et\'e de probabilit\'es de {S}aint-{F}lour,
  {XIV}---1984, Lecture Notes in Math., vol. 1180, Springer, Berlin, 1986,
  pp.~265--439.

\end{thebibliography}
\bibliographystyle{amsalpha}

\end{document}